\documentclass[letterpaper,11pt,onecolumn]{article}
\usepackage[margin=1in]{geometry}

\usepackage[T1]{fontenc}
\usepackage[latin1]{inputenc}
\usepackage[english]{babel}

\usepackage{amsmath}
\usepackage{amsthm}
\usepackage{amsbsy}
\usepackage{amsfonts}
\usepackage{amssymb}
\usepackage{array}
\usepackage{bm}
\usepackage{booktabs}
\usepackage{csquotes}
\usepackage[inline]{enumitem}
\usepackage{float}
\usepackage{graphicx}
\usepackage{mathrsfs}
\usepackage{pgf}

\usepackage[font={footnotesize}]{caption}
\usepackage{subcaption}

\usepackage{tikz}
\usetikzlibrary{patterns}

\usepackage[pdftex]{hyperref}
\hypersetup{colorlinks=false}

\def\mhend{\color{black}}
\newcommand{\ol}[1]{\overline{#1}}
\newcommand{\tr}{\theta_{\operatorname{max}}}
\def\dom{\operatorname{D}}
\def\rfi{\mathfrak{n}}
\def\Hdir{H^1_{\Gamma_\mathrm{in}}(\Omega)}
\def\af{u}
\def\tf{v}
\def\phitf{\xi}
\def\deltanew{\delta}
\newcommand{\Ar}{T}
\newcommand{\ap}{a_\Omega}
\newcommand{\innerprod}[1]{(#1)}

\def\ninn{{n\in\mathbb{N}}}

\newcommand{\ldhrev}[1]{{\color{blue}{#1}}}

\newtheorem{theorem}{Theorem}[section]
\newtheorem{assumption}[theorem]{Assumption}
\newtheorem{lemma}[theorem]{Lemma}

\newtheorem{remark}[theorem]{Remark}
\newtheorem{corollary}[theorem]{Corollary}

\newtheorem{definition}[theorem]{Definition}

\newcommand{\eremk}{\hbox{}\hfill\rule{0.8ex}{0.8ex}}

\def\be{\begin{equation}}
\def\ee{\end{equation}}
\def\ba{\begin{array}}
\def\ea{\end{array}}
\def\bea{\begin{eqnarray}}
\def\eea{\end{eqnarray}}
\def\beas{\begin{eqnarray*}}
\def\eeas{\end{eqnarray*}}

\newcommand{\half}{\frac{1}{2}}

\newcommand{\impedance}{\varsigma}

\newcommand{\ome}{{\omega}}
\newcommand{\ptl}{{\partial}}

\newcommand{\doubleIN}{\mathbb{N}}
\newcommand{\doubleIR}{\mathbb{R}}
\newcommand{\doubleIC}{\mathbb{C}}

\newcommand{\ds}{\displaystyle}

\catcode `@=11

\newbox \itemlist@label
\newdimen \itemlist@labelpad

\def\itemlist@makelabel#1{%
\setbox\itemlist@label =\hbox{#1}%
\ifdim \wd\itemlist@label >\labelwidth
\itemlist@labelpad=\textwidth
\advance\itemlist@labelpad by -\rightmargin
\advance\itemlist@labelpad by -\@totalleftmargin
\advance\itemlist@labelpad by \labelwidth
\hbox to \itemlist@labelpad {#1\hfill}%
\else #1\hfill
\fi
}

\newcounter{exercisenumber}

\renewcommand{\theexercisenumber}{\thesection.\arabic{exercisenumber}}

\newcommand{\dtn}{\operatorname{DtN}}
\newcommand{\DtN}{\dtn}

\renewcommand{\Re}{\operatorname*{Re}}
\renewcommand{\Im}{\operatorname*{Im}}

\let\tilde\widetilde

\numberwithin{equation}{section}

\graphicspath{{figures/}}

\usepackage[backend=bibtex,
	giveninits=true,
	maxbibnames=4,
	date=year,
	doi=false,
	url=false,
	isbn=false,
	eprint=false]{biblatex}
\begin{document}
\baselineskip=16pt
\parskip=4pt

\newpage

\begin{center}
{\Large {\bf  
{\Large
Length-explicit stability analysis of Helmholtz problems in leaky circular waveguides \\[12pt]
L. Demkowicz$^a$, M. Halla$^b$  and J.M. Melenk$^c$
\\[12pt]}}}
{\large 
$^a$Oden Institute, The University of Texas at Austin\\[5pt]
$^b$Karlsruhe Institute of Technology\\[5pt]
$^c$Technische Universit\"at Wien 
}
\end{center}

\begin{abstract}
Motivated by the study and simulation of long, coiled optical fibers we consider in this article a simplified model that is prevalent in the engineering community.
Mathematically, the problem is specified as follows:
Time-harmonic wave propagation is modeled by the Helmholtz equation; the waveguide is a bounded circular section with a transparent boundary condition on one end; the dissipation of energy is modeled by an impedance boundary condition on the outer hull of the waveguide.
We show a stability estimate that is explicit in terms of the angular length of the waveguide.
The analysis is based on a separation of variables ansatz and the study of the related (nonselfadjoint) modal eigenvalue problem.
The key property there is to show that the modes form a Riesz basis in both $L^2$ and $H^1$ spaces.
To this end we apply perturbation theory for selfadjoint operators and the concept of local subordination of perturbations
[B.\ Mityagin and P.\ Siegl, JAM 139 (2019)].
Since the possibility of nontrivial Jordan chains cannot be ruled out, our whole methodology is conducted accordingly.
In addition, in contrast to previous works, we include a bounded but heterogeneous part of the waveguide into our considered setting.
\end{abstract}



\subsection*{Acknowledgments}  L.~Demkowicz was supported with AFOSR grant FA9550-23-1-0103.
M. Halla acknowledges funding from Deutsche Forschungsgemeinschaft (DFG, German Research Foundation), projects 541433971 and 258734477 - SFB 1173.
J.M. Melenk acknowledges funding by the Austrian Science Fund (FWF) under 
grant F65 ``taming complexity in partial differential systems'' (\href{https://doi.org/10.55776/F65}{DOI:10.55776/F65}).

%
%

\section{Introduction}


Electromagnetic waveguides are important for applications in telecommunication \cite{ImperialeJoly14} and energy transfer \cite{cmacs2003baseline} (see, e.g., the bibliography of \cite{Leclerc25} for further references).
In particular, fiber optic waveguides provide foundations for several photonic applications, including sensing, scientific experimentation, 
healthcare, defense technologies, and others.
For amplifier applications the waveguide consists of a few centimeters subject to various devices such as signal and pump input signals, followed by several meters of a uniform fiber \cite{McComb09}.
In order to meet compact packaging conditions, the latter part of the fiber is typically coiled (around a cooling device).
To simulate such applications a nonlinear Maxwell system has to be solved in the entire waveguide \cite{RamanGain19}, where the stability constant of the linearized system significantly impacts the involved fixed-point iteration in the solution process of the nonlinear system.
Thus an analysis of this stability constant in terms of the fiber's length is crucial to determine the feasibility of numerical methods.
To cope with the high computational demand of realistic simulations \cite{Henneking21} the discontinuous Petrov Galerkin (DPG) method has successfully been applied.
In this context an explicit knowledge of the dependence of the stability constant upon the length of the waveguide allows one to tune the test norm for the ultraweak formulation to essentially enhance the discrete stability of the DPG method, see \cite[Sec.~5]{Melenk_Demkowicz_Henneking_25} for details and numerical illustrations.
Critical to successful simulations of problems with 10M wavelengths is the use of the \emph{full envelope ansatz} \cite{Henneking_Grosek_Demkowicz_25} which reduces the problem to ``only'' thousands of wavelengths but results in modified Maxwell equations to be solved.
Due to \cite{Melenk_Demkowicz_Henneking_25} the stability constant for this modified problem is equal to the one of the original problem.


In mathematical terms a realistic model for a fiber amplifier is given as a helical waveguide with a step index profile, a dissipative boundary condition at the outer hull and a termination with a transparent boundary condition.
However, in contrast to the study of \emph{frequency explicit} stability estimates (see, e.g., \cite{SauterTorres18,ChaumontFMoiolaSpence23}) the study of fiber \emph{length explicit} stability estimates is much less mature.
The latter was initiated for the case of a straight, homogeneous, closed fiber \cite{Melenk_Demkowicz_Henneking_25}.
Therein, a key ingredient of the analysis are basis properties of the modes \cite{Kim17}.
However, to consider the step index profile, a heterogeneous waveguide has to be dealt with, for which the analysis of the modes \cite{Delitsyn00,SheSmirSmo22,HallaMonk24} is much more intricate.
Another challenge is the helical nature, which even for a scalar Helmholtz equation leads to a \emph{quadratic} modal eigenvalue problem \cite{GopalaNeunteufel25}.
Nevertheless, in the prevalent engineering literature this aspect is neglected and a circularly bent slab waveguide (leading to the scalar Helmholtz problem considered in this article) is used as a reference model
\cite{Marcatili69,
HeiblumHarris75,
Marcuse76,
TakumaMiyagiKawakami81,
Marcuse82,
FaustiniMartini97,
Taylor84,
SchermerCole07}.
The main focus in these studies has been on losses due to bending, which are usually respected by imposing an impedance boundary condition where several material formulas have been proposed.
Indeed, it is recognized that the impedance boundary condition serves as an approximation to an exact transparent boundary condition of an \emph{open} waveguide and we refer to \cite{Mora-Paz_Demkowicz_Taylor_Grosek_Henneking_25} for computational studies on the validity of impedance type models.
Let us note that open, bent waveguides have some distinct properties in comparison to open, straight waveguides \cite{BambergerBonnet90,JolyPoirier95}.
Taking into account the impedance boundary condition the modal eigenvalue problem becomes non-selfadjoint, which further complicates the analysis of modal basis properties.
In this direction a scalar, straight waveguide with impedance boundary condition has been studied in \cite{Demkowicz_Gopalakrishnan_Heuer_24} by conducting rather explicit analytical computations.
Therein the authors derived an explicit transcendental equation for the wavenumbers, which enabled them to apply the Glazman Theorem for dissipative operators \cite[p.~328]{GohbergKrein69}.
For bent waveguides the former approach seems infeasible and hence we will employ a more general approach, i.e., to apply perturbation theory of selfadjoint operators.

In classical eigenvalue perturbation theory as in \cite{Kato} one tracks \emph{a finite number} of eigenvalues under suitable small perturbations of the operator.
However, when one is interested in basis properties of the root vectors we have to deal with \emph{all} eigenvalues and root vectors simultaneously.
The smallness of the perturbation is hence not measured in magnitude, but rather in terms of the degree of compactness.
This notation is formalized as the \emph{$p$-subordination condition} and goes back at least to \cite{Markus_88} (see also the more recent book \cite{Jeribi21}, and the introduction of \cite{MityaginSiegl25}for further references).
For our analysis we will use a recent improvement: the \emph{local form subordination condition} of \cite{Mityagin_Siegl_19}.

The remainder of this articles is structured as follows.
In Section~\ref{section:problem_formulation} we introduce the waveguide problem to be investigated.
In Section~\ref{section:change_of_variables} we derive the associated modal eigenvalue problem and prove basis properties of the modes by means of \cite{Mityagin_Siegl_19} and Appendix~\ref{appendix:2}.
This will enable us to perform a length-explicit stability analysis.
First, in Section~\ref{section:analysis} we make the simplifying assumption that all wavenumbers are semi-simple, i.e., all Jordan chains have length one, to illustrate the approach of our analysis.
Then, in Section~\ref{section:analysis_with_Jordan_chains}, we extend our analysis to the general case admitting Jordan chains, see Theorem~\ref{thm:stability-circular-with-Jordan}.
In Section~\ref{section:interior} we further extend our analysis and include an interior heterogeneous part into our considered setting, see Theorem~\ref{thm:interior}.
Additionally, Appendix~\ref{appendix:2} reviews the relevant part of the work \cite{Mityagin_Siegl_19} and extends it by Lemmas~\ref{lemma:L2phi_est}, \ref{lemma:H1phi_est}, \ref{lem:renormed_bari_basis} and Theorems \ref{thm:RieszBasisInV}, \ref{thm:H1Bari_basis}.
Appendix~\ref{appendix:3} relates the Jordan chains of  an operator $A$ to those of the adjoint operator $A^\star$. 
Furthermore, Appendix~\ref{appendix:1} contains a numerical evidence demonstrating that although the presence of eigenvalues that are not semisimple is very unlikely, 
theoretically they may indeed arise and, therefore, cannot be excluded from the analysis. \\
\emph{Notation:} For domains $D$, we employ standard Hilbertian Sobolev spaces $H^k(D)$, $k  \in \doubleIN$, with inner products $\innerprod{\cdot,\cdot}_{H^k(D)}$ 
and norms $\|\cdot\|_{H^k(D)}$. We write $L^2(D) = H^0(D)$. The $L^2$-inner product is sometimes abbreviated 
$\innerprod{u,v}_D:= \innerprod{u,v}_{L^2(D)}$ and, if necessary, understood as a duality pairing. 
For the interval $(r_1,r_2) \subset \doubleIR_{>0}$, we use the special weighted $L^2$-inner products 
$\innerprod{u,v}_{r^{\pm 1}}:= \int_{r_1}^{r_2} r^{\pm 1} u(r) \overline{v}(r)\,dr$ and corresponding norms $\|\cdot\|_{r^{\pm}}$. 
\section{Formulation of the problem
\label{section:problem_formulation}
}

Our starting point is the waveguide
%
\begin{figure}[h!]
\begin{tikzpicture}[scale=0.6]
\draw[thick,->] (0,0)--(7,0); 
\draw[thick,->] (0,0)--(0,7); 
\draw[thick] (6,0) arc[start angle=0, end angle=60, radius=6];
\draw[thick] (4,0) arc[start angle=0, end angle=60, radius=4];
\draw[thick] (2,3.46)--(3,5.19); 
\draw[dashed] (0,0)--(2,3.46); 
\draw[->,thin] (1.8,0) arc[start angle=0, end angle=60, radius=1.77];
\node[right] at (0.2,0.4) {$\tr$}; 
\node[below] at (4,0) {$r_1$}; 
\node[below] at (6,0) {$r_2$}; 
\node[below] at (5,-0.5) {$\Gamma_{\rm in}^C$}; 
\node[right] at (5.5,3.5) {$\Gamma_{\rm imp}^C$}; 
\node[left] at (2.5,5.0) {$\Gamma_{\rm out}^C$}; 
\node[left] at (3.9,1.4) {$\Gamma_{\rm hard}^C$}; 
\node[below] at (7,0) {$x$};  
\node[left] at (0,7) {$y$}; 
\node[right] at (3.3,3) {$\Omega^C$}; 
\end{tikzpicture}
\hfill 
\begin{tikzpicture}[scale=0.5]
\draw[thick,->] (0,0)--(7,0); 
\draw[thick,->] (0,0)--(0,7); 
\draw[thick] (0,2) -- (5.5,2) -- (5.5, 4) -- (0,4) --cycle; 
\node[below] at (7,0) {$\theta$}; 
\node[left] at (0,7) {$r$}; 
\draw[thin] (5.5,-0.1) -- (5.5, 0.1); 
\node[below] at (5.5,0) {$\tr$}; 
\node[left] at (0,2) {$r_1$}; 
\node[left] at (0,4) {$r_2$}; 
\node[below] at (3.5,2) {$\Gamma_{\rm hard}$}; 
\node[above] at (3.5,4) {$\Gamma_{\rm imp}$}; 
\node[right] at (0,3) {$\Gamma_{\rm in}$}; 
\node[right] at (5.5,3) {$\Gamma_{\rm out}$}; 
\node[right] at (2.5,3) {$\Omega$}; 

\end{tikzpicture}
\hfill 
\begin{tikzpicture}[scale=0.5]
\draw[thick,->] (-2,0)--(7,0); 
\draw[thick,->] (0,0)--(0,7); 
\draw[thick] (0,2) -- (5.5,2) -- (5.5, 4) -- (0,4) --cycle; 
\node[below] at (7,0) {$\theta$}; 
\node[left] at (0,7) {$r$}; 
\draw[thin] (5.5,-0.1) -- (5.5, 0.1); 
\node[below] at (5.5,0) {$\tr$}; 
\node[left] at (0.1,1.7) {$r_1$}; 
\node[left] at (0.1,4.3) {$r_2$}; 
\draw[dashed] (-2,2) -- (0,2) -- (0,4) -- (-2,4)--cycle; 
\draw[thin] (-2,-0.1) -- (-2,0.1); 
\node[below] at (-2,0) {$\theta_0$}; 
\node[below] at (3.5,2) {$\Gamma_{\rm hard}$}; 
\node[above] at (3.5,4) {$\Gamma_{\rm imp}$}; 
\node[left] at (-2,3) {$\Gamma_{\rm in}$}; 
\node[right] at (5.5,3) {$\Gamma_{\rm out}$}; 
\node[right] at (2.5,3) {$\Omega_{\tr}$}; 
\node[right] at (-1.5,3) {$\Omega_0$}; 

\end{tikzpicture}
\hfill 
\caption{
\label{fig:waveguide}
Left: 
Circular waveguide with impedance BC on outer side and sound hard BC on the inner side in Cartesian coordinates. 
Center: Circular waveguide in polar coordinates. 
Right: Setup of Section~\ref{section:interior} with a heterogeneous region $\Omega_0$ is attached to the waveguide.}
\end{figure}
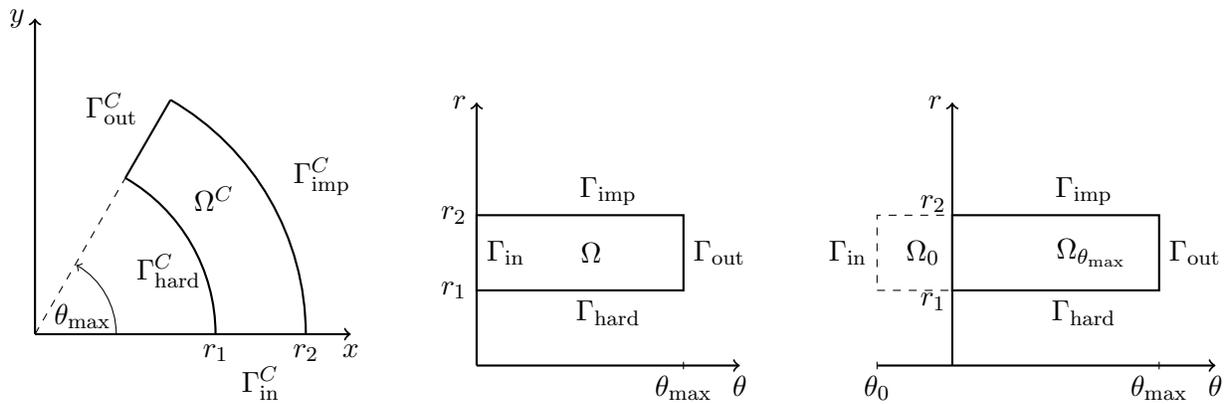
\begin{align*}
\Omega^C &:= \{ (r \cos \theta, r \sin \theta) \colon r \in (r_1,r_2),\, \theta \in (0,\tr )\},\\
\Gamma^C_\mathrm{hard} &:= \{ (r_1 \cos \theta, r_1 \sin \theta) \colon \theta \in (0,\tr )\},\\
\Gamma^C_\mathrm{imp} &:= \{ (r_2 \cos \theta, r_2 \sin \theta) \colon \theta \in (0,\tr )\},\\
\Gamma^C_\mathrm{in} &:= (r_1,r_2)\times\{0\},\\
\Gamma^C_\mathrm{out} &:= \{ (r \cos \tr, r \sin \tr) \colon r\in(r_1,r_2)\},
\end{align*}
as illustrated in the left part of Fig.~\ref{fig:waveguide}, where $0 < r_1 < r_2 < \infty$ are the inner and outer radii of the waveguide  and $\tr\in(0,2\pi)$ is the angular length of the waveguide.
We consider the Helmholtz problem to find a solution $\af\in H^1(\Omega)$ to
\begin{subequations}	
\label{eq:Helmholtz-phys}
\begin{align}
\label{eq:HH-phys-PDE}
-\Delta \af-\omega^2\rfi \af &= f &&\hspace{-30mm}\text{in }\Omega^C,\\
\partial_n \af &= 0 &&\hspace{-30mm}\text{on }\Gamma^C_\mathrm{hard},\\
\partial_n \af &= i\omega\impedance \af &&\hspace{-30mm}\text{on }\Gamma^C_\mathrm{imp},\\
\af &= 0 &&\hspace{-30mm}\text{on }\Gamma^C_\mathrm{in},\\
\partial_n \af &= \DtN \af &&\hspace{-30mm}\text{on }\Gamma^C_\mathrm{out},
\end{align}
\end{subequations}
where $\omega>0$ denotes the angular frequency (with the time-dependent convention $e^{i\omega t}$), $\impedance>0$ is an impedance constant, $\rfi\in L^\infty(\Omega^C)$ is the refractive index which we assume to be (in polar coordinates) independent of $\theta$, $f\in (H^1(\Omega^C))'$ is the load, and $\DtN \in L(H^{1/2}(\Gamma^C_\mathrm{out}),(H^{1/2}(\Gamma^C_\mathrm{out}))')$ is the Dirichlet-to-Neumann operator (which will be specified in Section~\ref{sec:dtn-operator}) implying a transparent boundary condition at $\Gamma_\mathrm{out}^C$.
To model long coiled fibers we henceforth switch to polar coordinates
\begin{align*}
\Omega &:= (r_1,r_2)\times (0,\tr ),\quad
\Gamma_\mathrm{hard} := \{r_1\}\times (0,\tr ),\quad
\Gamma_\mathrm{imp} := \{r_2\}\times (0,\tr ),\\
\Gamma_\mathrm{in} &:= (r_1,r_2)\times\{0\},\quad
\Gamma_\mathrm{out} := (r_1,r_2)\times\{\tr\},
\end{align*}
and allow $\tr>2\pi$ as indicated in the right part of Fig.~\ref{fig:waveguide}.  
In polar coordinates the former problem reads: Find $\af\in H^1(\Omega)$ such that
\begin{subequations}	
\label{eq:Helmholtz-polar}
\begin{align}
\label{eq:HH-polar-PDE}
-\partial_r(r\partial_r \af)-r^{-1}\partial^2_\theta \af-\omega^2\rfi r \af &= r f &&\hspace{-20mm}\text{in }\Omega=(r_1,r_2)\times(0,\tr),\\
r\partial_r \af &= 0 &&\hspace{-20mm}\text{on }\Gamma_\mathrm{hard},\\
r\partial_r \af &= ri\omega\impedance \af &&\hspace{-20mm}\text{on }\Gamma_\mathrm{imp},\\
\af &= 0 &&\hspace{-20mm}\text{on }\Gamma_\mathrm{in},\\
r^{-1}\partial_\theta \af &= \DtN \af &&\hspace{-20mm}\text{on }\Gamma_\mathrm{out},
\end{align}
\end{subequations}
where we do not introduce new symbols for $\af$, $f$, $\rfi$, and $\DtN$ in polar coordinates.
Let
\begin{align*}
\Hdir:=\{\af\in H^1(\Omega)\colon \af=0\text{ on }\Gamma_\mathrm{in} \}.
\end{align*}
Then the variational formulation of \eqref{eq:Helmholtz-polar} reads: 
Find $\af\in\Hdir$ such that
\begin{align}
\label{eq:Helmholtz-polar-var}
\begin{aligned}
\ap(u,v) &:= 
\innerprod{ r\partial_r \af,\partial_r \tf }_{\Omega}
+\innerprod{ r^{-1}\partial_\theta \af,\partial_\theta \tf }_{\Omega}
-\omega^2 \innerprod{ r\rfi \af, \tf }_{\Omega}
+i\omega\impedance r_2 \innerprod{ \af, \tf }_{\Gamma_{\rm imp}}
-\innerprod{ \DtN \af, \tf }_{\Gamma_{\rm out}} \\
%
%
& = \innerprod{ r f, \tf }_{\Omega},
\end{aligned}
\end{align}
for all $\tf\in\Hdir$. 
The goal of this paper is to prove the {\sl a priori} estimates of the form
\begin{align}
\label{eq:stability_system}
\|\af\|_{H^1(\Omega)} \leq C \tr \Vert f \Vert_{(\Hdir)'},
\end{align}
where $C>0$ is independent of $\tr$, $f$, and the solution $\af$ to \eqref{eq:Helmholtz-polar-var} but may depend upon the parameters $r_1$, $r_2$, $\omega$, $\rfi$ and $\impedance$. 
In fact, \eqref{eq:stability_system} is the result of our analysis for \emph{realistic} applications: Actually, our analysis, formulated 
in Theorems~\ref{thm:stability-circular} and \ref{thm:stability-circular-with-Jordan}, 
will yield a bound that is \emph{uniform} in $\tr$, but with an unpleasantly large constant $C$ (see Rem.~\ref{rem:constbound}). 
With moderate constant we obtain a factor $\tr^{J}$ with $J$ being the maximal length of modal Jordan chains 
(see Thm.~\ref{thm:stability-circular-with-Jordan}) and in applications it usually holds $J=1$.\mhend
\begin{remark}
Our analysis focuses on the second order Helmholtz problem (\ref{eq:Helmholtz-polar}). Of interest is also the first order formulation 
as done in the precursor works \cite{Melenk_Demkowicz_Henneking_25,Demkowicz_Gopalakrishnan_Heuer_24}, for which 
(\ref{eq:stability_system}) implies corresponding stability results. 
\eremk
\end{remark}
\section{Modal analysis}
\label{section:change_of_variables}
The convenient approach to investigate waveguide problems, which we follow here, is an analysis of the modes.
Here, a mode is defined to be a solution $\af$ obtained from a separation of variables ansatz
\begin{align}\label{eq:modeansatz}
\af(r,\theta) = p(\theta) \phi(r)
\end{align}
to the homogeneous problem in the semi-infinite waveguide $(r_1,r_2)\times (\theta_0,+\infty)$:
\begin{subequations}
\label{eq:PDEmode}
\begin{align}
-\partial_r(r\partial_r \af)-r^{-1}\partial_{\theta}\partial_{\theta}\af -\omega^2 r\rfi \af&=0 \quad\text{in } (r_1,r_2)\times (\theta_0,+\infty),\\
\partial_r \af &=0 \quad\text{on } \{r_1\}\times (\theta_0,+\infty),\\
\partial_r \af -i\omega\impedance \af&=0 \quad\text{on } \{r_2\}\times (\theta_0,+\infty),
\end{align}
\end{subequations}
(where $\theta_0$ is an arbitrary parameter specifying the start of the waveguide, e.g., $\theta_0=0$ or $\theta_0=\tr$).
Note that no boundary condition on $(r_1, r_2) \times \{\theta_0\}$ is imposed. With a slight abuse notation, 
sometimes only the $r$-dependent part $\phi$ of \eqref{eq:modeansatz} is also referred to as the mode. 
Plugging \eqref{eq:modeansatz} into \eqref{eq:PDEmode} we obtain
\begin{equation}
\label{eq:separation-of-variables}
-\frac{r (r \phi')'}{\phi} - \omega^2 r^2 \rfi = \frac{p''}{p} = \lambda, 
\end{equation}
where $\lambda \in \doubleIC$ is the separation constant.
This leads to the modal eigenvalue problem:
Find $(\lambda,\phi)\in\mathbb{C}\times H^1(r_1,r_2)\setminus\{0\}$ such that
\begin{subequations}
\label{eq:bessel-EVP}
\begin{align}
\label{eq:bessel-EVP-ODE}
-(r\phi')' -\omega^2 r\rfi \phi -r^{-1}\lambda \phi & =  0,  \quad  r \in (r_1,r_2),\\
\phi' & = 0\quad \text{at } r = r_1,\\
\label{eq:bessel-EVP-impBC}
\phi' {+} i \omega \impedance \phi & = 0\quad \text{at } r=r_2 \, .
\end{align}
\end{subequations}
The weak form of \eqref{eq:bessel-EVP} reads: 
Find $(\lambda,\phi)\in\mathbb{C}\times H^1(r_1,r_2)\setminus\{0\}$ such that
\begin{align}
\label{eq:bessel-EVP-weak}
\innerprod{ \phi',\phitf' }_r-\omega^2 \innerprod{ \rfi \phi,\phitf }_r
+i\omega\impedance r_2 \phi(r_2)\ol{\phitf(r_2)}
=\lambda \innerprod{ \phi,\phitf }_{r^{-1}}
\end{align}
for all $\phitf\in H^1(r_1,r_2)$,
where $\innerprod{ \cdot,\cdot }_{r^{\pm1}}$ denote weighted, Hermitian $L^2(r_1,r_2)$ scalar products. 
The left-hand side of \eqref{eq:bessel-EVP-weak} corresponds to a densely defined, closed
operator $\Ar\colon (L^2(r_1,r_2), \innerprod{ \cdot,\cdot}_{r^{-1}}) \rightarrow 
(L^2(r_1,r_2), \innerprod{ \cdot,\cdot}_{r^{-1}}) $ given by 
\begin{subequations}
\label{eq:Ar}
\begin{align}
\Ar \phi &:= - r (r \phi^\prime)^\prime - \ome^2 \rfi r^2 \phi, \\
\dom(\Ar) & := \{\phi \in H^2(r_1,r_2)\colon \phi^\prime(r_1) = 0, \quad 
\phi^\prime(r_2 ) + i \ome\impedance \phi(r_2) = 0\}. 
\end{align}
\end{subequations}
We note that the eigenvalues of (\ref{eq:bessel-EVP-weak}) are not real: 
\begin{lemma}
\label{lem:lambdaisnonreal}
All eigenvalues $\lambda$ to \eqref{eq:bessel-EVP-weak} satisfy $\lambda\notin\mathbb{R}$. Additionally, all eigenvalues $\lambda$ have geometric 
multiplicity $1$, i.e., exactly one Jordan chain is associated with each eigenvalue $\lambda$. 
\end{lemma}
\begin{proof}
\emph{Proof of $\lambda \not\in \doubleIR$:}
Let $(\lambda,\phi)$ be an eigenpair with $\lambda\in\mathbb{R}$.
Then by letting $\phitf=\phi$ and taking the imaginary part of \eqref{eq:bessel-EVP-weak} we obtain that $\phi(r_2)=0$.
Then due to \eqref{eq:bessel-EVP-impBC} also $\phi'(r_2)=0$.
Since $\phi$ solves the ordinary differential equation \eqref{eq:bessel-EVP-ODE} it follows that $\phi=0$ in $(r_1,r_2)$, which is a contradiction.

\emph{Proof of the multiplicity of $\lambda$:} Let $(\lambda,\phi)$ be an eigenpair. The function $\phi$ solves a second order ODE with the 
two initial conditions $\phi(r_2) = \phi(r_2)$ and $\phi^\prime(r_2) = - i \ome \impedance \phi(r_2)$ at $r = r_2$, so that $\phi(r_2)$ determines $\phi$ fully.
\end{proof}
Once $\lambda$ is known, we solve equation \eqref{eq:separation-of-variables} for the factor $p$ and obtain
\be
p(\theta) =  C_1 e^{i \beta \theta} + C_2 e^{- i \beta \theta},
\qquad C_1,C_2 \in \doubleIC,
\label{eq:Theta}
\ee
where $\lambda = - \beta^2$ with the square root convention that $\Im\beta>0$.
(Note that $\Im\beta=0$ does not occur due to Lem.~\ref{lem:lambdaisnonreal}.)
The modes $e^{i\beta\theta}\phi(r)$ are called outgoing and the modes $e^{-i\beta\theta}\phi(r)$ are called incoming.
The radiation condition in $(r_1,r_2)\times(\theta_0,+\infty)$ demands the energy of solutions to be finite, and hence an \emph{outgoing} solution $u$ should consist only of outgoing modes, i.e., $u(r,\theta)=\sum_{n=1}^\infty e^{i\beta_n\theta}\phi_n(r)$ in $(r_1,r_2)\times(\theta_0,+\infty)$.
Thus to derive a transparent boundary on $\Gamma_\mathrm{out}$ we extend $\Omega$ to $(r_1,r_2)\times(0,+\infty)$ and demand that the solution be outgoing in $(r_1,r_2)\times(\tr,+\infty)$.
This leads to the definition of the Dirichlet-to-Neumann operator:
\begin{align}
\label{eq:DtN-eigenmodes}
\DtN\Big(\sum_{n=1}^\infty c_n \phi_n\Big):=
{\frac{1}{r}} \sum_{n=1}^\infty c_n i\beta_n \phi_n, \quad c_n\in\mathbb{C}.
\end{align}
This formula assumes the eigenmodes $\phi_n$ to be eigenfunctions of the 
eigenvalue problem (\ref{eq:bessel-EVP-ODE}) --- a corresponding formula in case 
Jordan chains appear will be given in Section~\ref{section:analysis_with_Jordan_chains}.
In this direction we introduce the following convention.
\begin{definition}[root vectors]\label{def:rootvecs}
We say that 
$\{\phi_n\}_{n\in\mathbb{N}}$ are the root vectors (generalized eigenvectors) of an operator $A$, 
if 
\begin{enumerate*}[label=\alph*)]
\item
the spectrum $\sigma(A)$ of $A$ consists of countably many eigenvalues each with a finite algebraic multiplicity, 
\item the enumeration of eigenvalues $\{\lambda_n\}_{n\in\mathbb{N}}=\sigma(A)$ respects the algebraic multiplicity, 
\item for each $m\in\mathbb{N}$ the set $\{\phi_n\colon \lambda_n=\lambda_m\}$ is a basis of the root space associated to the eigenvalue $\lambda_m$, 
and 
\item for each $n\in\mathbb{N}$ either $\phi_{n+1}$ is an eigenvector or a subsequent element of a Jordan chain, i.e., $(A-\lambda_{n+1})\phi_{n+1}=\phi_n$.
\end{enumerate*}
We say that $\{(\lambda_n,\phi_n)\}_{n\in\mathbb{N}}$ are the normalized eigenvalue-root vector pairs of $A$, if in addition all \emph{eigenvectors} are normalized.
The wording ``\emph{the} root vectors'' is used for convenience although their choice is not unique.
\end{definition}
To justify that each $\phi\in H^{1/2}((r_1,r_2)\times \{\tr\})$ admits an expansion $\phi=\sum_{n=}^\infty c_n \phi_n$ 
and to ensure that $\DtN$ is indeed a well-defined bounded operator (see Lemma~\ref{lemma:mapping-properties-DtN}), we have to analyze the properties of the modal eigenvalue problem \eqref{eq:bessel-EVP-weak}.
To this end we will apply perturbation theory for selfadjoint operators and in particular the concept of \emph{local form subordinate} perturbations \cite{Mityagin_Siegl_19}, which we slightly extend in Section~\ref{appendix:2}.
To that end, we introduce a change of variables to carve out a suitably ``unperturbed'' part.
Consider the change of variables 
\begin{align}
\label{eq:change-of-variables}
\widetilde\phi(t):=\phi(r),\qquad
t := \frac{\pi}{x_2-x_1} \left(r_0 \ln \left( \frac{r}{r_0} \right) - x_1\right),
\end{align}
%
where
\begin{align}
x_1 := r_0 \ln \left( \frac{r_1}{r_0} \right), \qquad
x_2 := r_0 \ln \left( \frac{r_2}{r_0} \right), \qquad
\deltanew:=\frac{1}{\pi} \ln \left( \frac{r_2}{r_1} \right).
\end{align}
Then, in the new coordinate $t\in(0,\pi)$, the equation \eqref{eq:bessel-EVP-weak} reads:
Find $(\widetilde\lambda,\widetilde\phi)\in \mathbb{C}\times H^1(0,\pi)\setminus\{0\}$ such that, 
for all $\widetilde\phitf \in H^1(0,\pi)$, 
\begin{align}
\label{eq:eigenproblem}
\underbrace{
\innerprod{ \widetilde\phi',\widetilde\phitf' }_{L^2(0,\pi)} + \innerprod{ \widetilde\phi,\widetilde\phitf}_{L^2(0,\pi)}}_{\widetilde{a}(\widetilde\phi,\widetilde\phitf):=}
\underbrace{
-\innerprod{ (\omega^2 r_1^2 \deltanew^2 e^{2\deltanew t}\widetilde\rfi+1) \widetilde\phi, \widetilde\phitf }_{L^2(0,\pi)}
{+} i \omega \impedance \deltanew r_2 \widetilde\phi(\pi) \ol{\widetilde\phitf(\pi)}}_{\widetilde{b}(\widetilde\phi,\widetilde\phitf):=}
=\widetilde{\lambda} \innerprod{ \widetilde\phi,\widetilde\phitf }_{L^2(0,\pi)}, 
\end{align}
with the relation
\begin{align}\label{eq:lambda-tlambda}
\widetilde\lambda=\deltanew^2\lambda.
\end{align}
In the following we will derive several \emph{asymptotic} properties of the eigenvalues $\{\widetilde\lambda_n\}_{n\in\mathbb{N}}$ and root vectors $\{\widetilde\phi_n\}_{n\in\mathbb{N}}$ of the eigenvalue problem \eqref{eq:eigenproblem}.
Then corresponding properties of the eigenvalues $\{\lambda_n\}_{n\in\mathbb{N}}$ and root vectors $\{\phi_n\}_{n\in\mathbb{N}}$ of \eqref{eq:bessel-EVP-weak} follow from the relations \eqref{eq:change-of-variables} and \eqref{eq:lambda-tlambda} in Corollary~\ref{cor:basis} below. 
The eigenvalue problem to find $(\widetilde\mu,\widetilde\psi)\in\mathbb{C}\times H^1(0,\pi)\setminus\{0\}$ such that
\begin{align}\label{eq:evp-psi}
\widetilde{a}(\widetilde\psi,\widetilde\phitf)=\widetilde\mu \innerprod{ \widetilde\psi,\widetilde\phitf }_{L^2(0,\pi)}
\qquad \forall\widetilde\phitf\in H^1(0,\pi),
\end{align}
serves as our \emph{unperturbed, selfadjoint} eigenvalue problem.
It is straight-forward to check that its $L^2(0,\pi)$-normalized eigenpairs are given by
\begin{align}
\label{eq:mutilde}
\widetilde\mu_1:=1, \qquad\widetilde\psi_1(t):=\sqrt{\frac{1}{\pi}}, \qquad
\widetilde\mu_n := (n-1)^2+1, \qquad \widetilde\psi_n(t) := \sqrt{\frac{2}{\pi}} \cos ((n-1) t), \quad n=2,3,\ldots\, .
\end{align}
The form $\widetilde{a}(\cdot,\cdot)$ can be interpreted as a densely defined, unbounded form $L^2(0,\pi)\times L^2(0,\pi)\to\mathbb{C}$ 
with domain space $\dom(\widetilde{a})=H^1(0,\pi)$.
It also induces an unbounded, densely defined, closed operator $\widetilde{A}\colon L^2(0,\pi)\to L^2(0,\pi)$ given by
\begin{align}
\label{eq:tildeA}
\widetilde{A} \widetilde\phi &:= -\phi^{\prime\prime},\qquad
\dom(\widetilde{A}) := \{\widetilde\phi \in H^2(0,\pi)\colon \widetilde\phi^\prime(0) = \widetilde\phi^\prime(\pi)=0\}. 
\end{align}
Following Kato's second representation theorem \cite[Thm.~2.23]{Kato} we have $\dom(\widetilde{a})=\dom(\widetilde{A}^{1/2})$ and $\widetilde{a}(\widetilde\phi,\widetilde\phitf)=\innerprod{ \widetilde{A}^{1/2}\widetilde\phi,\widetilde{A}^{1/2}\widetilde\phitf }_{L^2(0,\pi)}$ for all $\widetilde\phi,\widetilde\phitf\in\dom(\widetilde{A}^{1/2})$.
Then the spectral theory for unbounded, selfadjoint operators yields that $\{\widetilde\psi_n\}_{n\in\mathbb{N}}$ forms an orthonormal basis in $L^2(0,\pi)$.
On the other hand, we can work with bounded operators on $H^1(0,\pi)$ and then \eqref{eq:evp-psi} reads in operator form $\widetilde\psi=\widetilde\mu E^*E \widetilde\psi$ with the compact embedding operator $E\in L(H^1(0,\pi),L^2(0,\pi))$.
Hence the spectral theory for compact, selfadjoint operators yields that $\{\widetilde\mu_n^{-1/2}\widetilde\psi_n\}_{n\in\mathbb{N}}$ forms an orthonormal basis in $H^1(0,\pi)$.
Next we strive to establish similar properties for $\{\widetilde\lambda_n\}_\ninn$ and $\{\widetilde\phi_n\}_\ninn$.
To this end, we note that the left-hand side of \eqref{eq:eigenproblem} induces a densely defined, closed operator $\widetilde{T}\colon L^2(0,\pi) \rightarrow L^2(0,\pi)$ given by 
\begin{subequations}
\label{eq:tildeA}
\begin{align}
	\widetilde{T} \widetilde\phi &:= -\widetilde\phi^{\prime\prime} - \ome^2 r_1^2 \deltanew^2 e^{2\deltanew\cdot}\widetilde{\rfi} \widetilde\phi,  \\
	\dom(\widetilde{T}) & := \{\widetilde\phi \in H^2(0,\pi)\colon \widetilde\phi^\prime(0) = 0, \quad \widetilde\phi^\prime(\pi) + i\ome\impedance\deltanew r_2 \widetilde\phi(\pi) = 0\},
\end{align}
\end{subequations}
and $\widetilde{\mathfrak{t}}(\widetilde\phi,\widetilde\phitf):=\widetilde{a}(\widetilde\phi,\widetilde\phitf)+\widetilde{b}(\widetilde\phi,\widetilde\phitf)=\innerprod{ \widetilde{T}\widetilde\phi,\widetilde\phitf }_{L^2(0,\pi)}$ for all $\widetilde\phi\in\dom(\widetilde{T})$, $\widetilde\phitf\in\dom(\widetilde{\mathfrak{t}})=\dom(\widetilde a)$ due to Kato's first representation theorem \cite[Thm.~2.1]{Kato}.
\begin{theorem}\label{thm:basis}
Let $\{(\widetilde\mu_n, \widetilde\psi_n)\}_\ninn$ be the eigenpairs for the self-adjoint eigenvalue problem \eqref{eq:evp-psi} given by \eqref{eq:mutilde}, which are scaled such that $\|\widetilde{\psi}_n\|_{L^2(0,\pi)} = 1$. 
Then there exist $L^2$-normalized (in the sense of Definition~\ref{def:rootvecs}) eigenvalue--root vector pairs 
$\{(\widetilde\lambda_n,\widetilde\phi_n)\}_\ninn$ 
of the eigenvalue problem \eqref{eq:eigenproblem}
with the following properties:
\begin{enumerate}[nosep, label=(\roman*),leftmargin=*]
\item there exists $n_0>0$ such that $\widetilde\lambda_n$ is simple for all $n>n_0$;
\item $\{\widetilde\phi_n\}_\ninn$ forms a Bari basis of $L^2(0,\pi)$ and in particular $\sum_{n=1}^\infty \|\widetilde\psi_n-\widetilde\phi_n\|^2_{L^2(0,\pi)}<+\infty$;
\item $\{\widetilde\mu_n^{-1/2}\widetilde\phi_n\}_\ninn$ forms a Bari basis of $H^1(0,\pi)$ and in particular $\sum_{n=1}^\infty \|\widetilde\mu_n^{-1/2}\widetilde\psi_n-\widetilde\mu_n^{-1/2}\widetilde\phi_n\|^2_{H^1(0,\pi)}$ $<$ $+\infty$;
\item $\displaystyle\lim_{n \to \infty} \frac{\vert \widetilde{\lambda}_n \vert^\half}{\widetilde\mu_n^\half} = \lim_{n \to \infty} \frac{\vert \widetilde{\lambda}_n \vert^\half}{\Vert \widetilde{\phi}_n \Vert_{H^1}} = 1$,
and the sequence $\{\widetilde{\phi}_n/\|\widetilde{\phi}_n\|_{H^1(0,\pi)}\}_\ninn$ remains quadratically close to the sequence $\{\widetilde\psi_n/\widetilde\mu_j^\half\}_\ninn$ in the $H^1(0,\pi)$-norm.
\item $\lim_{n \rightarrow \infty} \frac{\widetilde{\lambda}_n}{\widetilde{\mu}_n} = 1$. 
\end{enumerate}
\end{theorem}
\begin{proof}
We will apply \cite{Mityagin_Siegl_19} and Appendix~\ref{appendix:2} and explicitly construct the pairs $(\widetilde{\lambda}_n, \widetilde{\phi}_n)$ 
as perturbations of the pairs $(\widetilde{\mu}_n,\widetilde{\psi}_n)$.
To this end we note that the eigenvalue separation assumption~\eqref{eq:separation_condition} is satisfied with $\kappa = 1$ and $\gamma = 2$.
Indeed, 
\begin{align*}
\widetilde{\mu}_{n+1} - \widetilde{\mu}_{n} = n^2 - (n-1)^2 = 2n-1 \geq n \qquad ( = \kappa n^{\gamma-1}) \, .
\end{align*}
The exponent $\alpha$ in the {\em local form-subordination condition} ~(\ref{eq:subordination_condition})
must satisfy
$$
2 \alpha + \gamma >1 \quad \mbox{(which implies the requirement $\alpha > -\half $)}
$$
with the following condition on the form $b(\cdot,\cdot)$:
\begin{equation}
\label{eq:sec:4-local-subordination}
\vert b(\widetilde\psi_m,\widetilde\psi_n) \vert \leq \frac{M_b}{m^\alpha n^\alpha}, \qquad M_b> 0 \, .
\end{equation}
The Cauchy-Schwarz inequality and $\vert \widetilde\psi_n(\pi)\vert \in\{ \sqrt{\frac{1}{\pi}},\sqrt{\frac{2}{\pi}}\}$ imply that condition \eqref{eq:sec:4-local-subordination} is satisfied for $\alpha \leq 0$. 
Hence the admissible range is $\alpha \in (-\half,0]$.
We will select $\alpha = 0$ for the analysis.
Then the first claim follows from \cite[Prop.~3.1]{Mityagin_Siegl_19}.
The second and third claims follow from Thm.~\ref{thm:L2Bari_basis} and Thm.~\ref{thm:H1Bari_basis}.
The fourth claim follows from Lem.~\ref{lem:renormed_bari_basis}.
The fifth claim follows from (\ref{eq:lem:renormed_bari_basis-10}).
\end{proof}
\mhend
Theorem~\ref{thm:basis} about the eigenvalue problem (\ref{eq:eigenproblem}) 
readily implies a corresponding result for the eigenpairs of (\ref{eq:bessel-EVP-weak}), 
i.e., the operator $\Ar$ given in (\ref{eq:Ar}). To do so, it is convenient to observe 
that the change of variables (\ref{eq:change-of-variables}) induces the following isometries:
For $\phi$, $\phitf$ and $\widetilde{\phi}$, $\widetilde{\phitf}$ related by \eqref{eq:change-of-variables}, there holds 
\begin{subequations}
\begin{align}
\label{eq:cor-basis-10}
\deltanew^{-1} \innerprod{ \phi,\phitf}_{r^{-1}} & = (\widetilde{\phi},\widetilde{\phitf})_{L^2(0,\pi)},  \\
\label{eq:cor-basis-20}
\innerprod{ \phi,\phitf}_{H^1,r}:= \deltanew \innerprod{ \phi^\prime, \phitf^\prime }_{r} + \deltanew^{-1} \innerprod{ \phi,\phitf}_{r^{-1}} & = (\widetilde{\phi},\widetilde{\phitf})_{H^1(0,\pi)}  . 
\end{align}
\end{subequations}
We note that $\|\cdot\|_{r^{-1}}
$ is an equivalent norm on $L^2(r_1,r_2)$ and 
$\|\cdot\|_{H^1,r}
:=\innerprod{\cdot,\cdot}_{H^1,r}^{1/2}
$ is an equivalent norm on $H^1(r_1,r_2)$.  
\begin{corollary}
\label{cor:basis}
Let the pairs $\{(\widetilde{\lambda}_n,\widetilde{\phi}_n)\}_\ninn$, 
$\{(\widetilde{\mu}_n,\widetilde{\psi}_n)\}_\ninn$ be the $L^2$-normalized (in the sense of Definition~\ref{def:rootvecs}) eigenvalue--root vector pairs of \eqref{eq:eigenproblem}, \eqref{eq:evp-psi} as in Theorem~\ref{thm:basis}. Define 
pairs $\{(\lambda_n,\phi_n)\}_\ninn$ and 
$\{(\mu_n,\psi_n)\}_\ninn$ by the change of variables \eqref{eq:change-of-variables} via 
$\phi_n(r) = \widetilde{\phi}_n(t)$ and $\psi_n(r) = \widetilde{\psi}_n(t)$ and set 
$\lambda_n = \widetilde{\lambda}_n \deltanew^{-2}$, $\mu_n = \widetilde{\mu}_n \deltanew^{-2}$. 
Then:
\begin{enumerate}[nosep, label=(\roman*),leftmargin=*]
\item 
\label{item:cor-basis-i}
The pairs $\{(\lambda_n,\phi_n)\}_\ninn$ are eigenvalue--root vector pairs of \eqref{eq:bessel-EVP}, i.e., 
of the operator $\Ar$ of \eqref{eq:Ar}. 
\item 
\label{item:cor-basis-ii}
The pairs $\{(\mu_n,\psi_n)\}_\ninn$ are eigenpairs of the self-adjoint operator 
\eqref{eq:evp-psi-transformed} below. In particular, 
the eigenfunctions $\psi_n, n\in\mathbb{N}$ are real-valued
and $\{\psi_n/\|\psi_n\|_{r^{-1}}\}_\ninn$ is an orthonormal basis of $(L^2(r_1,r_2), \innerprod{ \cdot,\cdot}_{r^{-1}})$ 
and an orthogonal basis of $(H^1(r_1,r_2), \innerprod{ \cdot,\cdot}_{H^1,r})$. 
\item 
\label{item:cor-basis-iii}
There exists $n_0>0$ such that $\lambda_n$ is simple for all $n>n_0$.
\item 
\label{item:cor-basis-iv}
$\{\phi_n/\|\phi_n\|_{r^{-1}}\}_\ninn$ forms a Bari basis of $(L^2(r_1,r_2), \innerprod{ \cdot,\cdot}_{r^{-1}})$
and in particular it holds that $\sum_{n=1}^\infty \|\psi_n/\|\psi_n\|_{r^{-1}}$ $-$ $\phi_n/\|\phi_n\|_{r^{-1}}\|^2_{r^{-1}}<+\infty$.
\item 
\label{item:cor-basis-v}
$\{\phi_n/\|\phi_n\|_{H^1,r}\}_\ninn$ forms a Bari basis of $(H^1(r_1,r_2), \innerprod{ \cdot,\cdot}_{H^1,r})$
and in particular it holds that $\sum_{n=1}^\infty \|\psi_n/\|\psi_n\|_{H^1,r}$ $-$ $\phi_n/\|\phi_n\|_{H^1,r}\|^2_{H^1,r}<+\infty$. 
\item 
\label{item:cor-basis-vi}
$|\lambda_n|  \sim \sqrt{n}$ and $\|\phi_n\|_{H^1} \sim \sqrt{|\lambda_n|} \|\phi_n\|_{L^2}$ with implied constants independent of $n\in\mathbb{N}$. 
In particular, $\inf_\ninn |\lambda_n| > 0$. 
\item 
\label{item:cor-basis-vii}
$\{\phi_n/\|\phi_n\|_{L^2(r_1,r_2)}\}_\ninn$ is a Riesz basis of $L^2(r_1,r_2)$ and 
$\{\phi_n/\|\phi_n\|_{H^1(r_1,r_2)}\}_\ninn$ is a Riesz basis of $H^1(r_1,r_2)$. 
\item 
\label{item:cor-basis-viia}
$\lim_{n \rightarrow \infty} \frac{\lambda_n}{\mu_n} = 1$. 
\end{enumerate}
Finally, 
\begin{enumerate}[nosep, label=(\roman*),leftmargin=*]
\setcounter{enumi}{8}
\item 
\label{item:cor-basis-viii}
Statements \ref{item:cor-basis-i}--\ref{item:cor-basis-viia} also hold for the root vectors $\{\chi_n\}_\ninn$ of the adjoint 
operator $\Ar^\star$ given in (\ref{eq:Arstar}). 
\end{enumerate}
\end{corollary}
\begin{proof}
\emph{Proof of \ref{item:cor-basis-i}:}
We show the following, stronger statement: for functions
$\phi$, $\phitf$, $g$  and $\widetilde{\phi}$, $\widetilde{\phitf}$, $\widetilde{g}$
related by the change of variables \eqref{eq:change-of-variables} there holds 
\begin{align}
\label{eq:cor-basis-100}
\Ar \phi = g \quad \Longleftrightarrow \quad \deltanew^{-2} \widetilde{T} \widetilde{\phi} = \widetilde{g} 
\end{align}
To see this, one starts from the weak form of \eqref{eq:cor-basis-100}, i.e., for all test functions $\phitf \in H^1$ 
\begin{align*}
	\innerprod{ g,\phitf}_{r^{-1}} & = \innerprod{ \Ar \phi, \phitf}_{r^{-1}}  = \innerprod{ \phi^\prime,\phitf^\prime}_r - \ome^2 \innerprod{ \rfi \phi,\phitf}_r + 
i \ome r_2 \impedance \phi(r_2) \overline{\phitf}(r_2).  
\end{align*}
The change of variables \eqref{eq:change-of-variables} leads to 
\begin{align*}
\deltanew(\widetilde{g},\widetilde{\phitf})_{L^2(0,\pi)} & = 
\deltanew^{-1} \left[ \innerprod{\widetilde{\phi}^\prime,\widetilde{\phitf}^\prime}_{L^2(0,\pi)} - \ome^2 r_1^2 \deltanew^2 \innerprod{ \widetilde{\rfi} e^{2 \deltanew t} \widetilde{\phi},\widetilde{\phitf}}_{L^2(0,\pi)} 
+ i \ome r_2 \impedance \deltanew \widetilde{\phi}(\pi) \overline{\widetilde{\phitf}}(\pi)
\right], 
\end{align*}
which is the weak form of the right-hand side of \eqref{eq:cor-basis-100}. 

\emph{Proof of \ref{item:cor-basis-ii}:} In view of \eqref{eq:cor-basis-10}, \eqref{eq:cor-basis-20}, the change of variables \eqref{eq:change-of-variables} transforms \eqref{eq:evp-psi} to 
\begin{align}
\label{eq:evp-psi-transformed}
\deltanew^2 \widetilde{\mu} \innerprod{ \phi,\xi} _{r^{-1}} & = \innerprod{ \phi^\prime,\xi^\prime}_{r} + \deltanew^{-2} \innerprod{\phi,\xi} _{r^{-1} }
\quad \forall \xi \in H^1(r_1,r_2). 
\end{align}
The result follows. 

\emph{Proof of \ref{item:cor-basis-iii}:} Follows from Theorem~\ref{thm:basis}. 

\emph{Proof of \ref{item:cor-basis-iv}:}
Follows from Theorem~\ref{thm:basis} and (\ref{eq:cor-basis-10}). 

\emph{Proof of \ref{item:cor-basis-v}:} In view of the isometries \eqref{eq:cor-basis-10}, \eqref{eq:cor-basis-20}
the choice of the norm $\|\cdot\|_{H^1,r}$ makes $\{\psi_n\}_\ninn$ an orthogonal basis of $(H^1(r_1,r_2), \|\cdot\|_{H^1,r})$ so that 
$\{\psi_n/\|\psi_n\|_{H^1,r}\}_\ninn$ is an orthonormal basis of $(H^1(r_1,r_2),\|\cdot\|_{H^1,r})$, and 
\eqref{eq:cor-basis-20} shows 
$\|\psi_n/\|\psi_n\|_{H^1,r} - \phi_n/\|\phi_n\|_{H^1,r}\|_{H^1,r} = 
\|\widetilde{\psi}_n/\|\widetilde{\psi}_n\|_{H^1} - \widetilde{\phi}_n/\|\widetilde{\phi}_n\|_{H^1}\|_{H^1}$, so that 
the statement follows from Theorem~\ref{thm:basis}.  

\emph{Proof of \ref{item:cor-basis-vi}:} The assertion $|\lambda_n| \sim n$ follows from Theorem~\ref{thm:basis} and
\eqref{eq:mutilde}. For the assertion $\|\phi_n\|_{H^1(r_1,r_2)} \sim \sqrt{|\lambda_n|} \|\phi_n\|_{L^2(r_1,r_2)}$, 
we note that  Theorem~\ref{thm:basis} implies $|\widetilde\lambda_n|^{1/2} \|\widetilde{\phi}_n\|_{L^2(0,\pi)} 
\sim \|\widetilde{\phi}_n\|_{H^1(0,\pi)}$. The result then follows with \eqref{eq:cor-basis-10}, \eqref{eq:cor-basis-20}. 

\emph{Proof of \ref{item:cor-basis-vii}:} Since every Bari basis is a Riesz basis, we have that 
$\{\phi_n/\|\phi_n\|_{r^{-1}}\}_n$ is a Riesz basis of $(L^2(r_1,r_2), \innerprod{ \cdot,\cdot}_{r^{-1}})$. Since 
$\|\cdot\|_{L^2(r_1,r_2)}$ and $\|\cdot\|_{r^{-1}}$ are equivalent norms, also 
$\{\phi_n/\|\phi\|_{L^2(r_1,r_2)}\}_n$ is a Riesz basis of $L^2(r_1,r_2)$ (see, e.g., \cite[Chap.~{VI}, Thm.~{2.1}]{GohbergKrein69})
The statement that 
$\{\phi_n/\|\phi_n\|_{H^1(r_1,r_2)}\}_n$ is a Riesz basis is shown similarly. 

\emph{Proof of \ref{item:cor-basis-viia}:} Follows directly from Theorem~\ref{thm:basis}.

\emph{Proof of \ref{item:cor-basis-viii}:} As the operator $\Ar^\star$ is obtained from $\Ar$ by merely changing $i$ to $-i$, the same arguments
as in the analysis of $\Ar$ apply. 
\end{proof}

A key quantity in our stability analysis will be played by 
\begin{align}
\label{eq:gamma_kappa}
\gamma_\kappa& := 1+ \frac{\Im \kappa}{1 + \Re \kappa} , 
\qquad \kappa \in \doubleIC\setminus\{-1\}. 
\end{align}
We recall that $\beta_n = \sqrt{- \lambda_n}$ with the square root taken such that $\Im \beta_n > 0$. We have:
\begin{lemma}
\label{lemma:control-Imbeta_n}
There are $ C_\beta$, $C^\prime_\beta > 0$ depending only on $\Ar$ such that 
\begin{enumerate}[nosep, label=(\roman*),leftmargin=*] 
\item 
\label{item:lemma:control-Imbeta_n-i} 
$\displaystyle \lim_{n \rightarrow \infty} \frac{\Re \beta_n}{\Im \beta_n}  = 0$ 
and $\lim_{n \rightarrow} |\beta_n|/\sqrt{\mu_n} = 1$. 
\item 
\label{item:lemma:control-Imbeta_n-ii} 
$\displaystyle \forall n \in \doubleIN: \quad |\Re \beta_n | \leq C_\beta |\Im \beta_n|.  $
\item 
\label{item:lemma:control-Imbeta_n-iii} 
For all $L >0$ and all $n \in \doubleIN$ there holds with $\gamma_\kappa$ given by \eqref{eq:gamma_kappa}: 
\begin{align}
\label{eq:item:lemma:control-Imbeta_n-iii-10} 
\gamma_{-i \beta_n L} &= 1+ \frac{\Re \beta_n L}{1 + \Im \beta_n L}
\leq 1+\min\{C_\beta, \max\{1,C^\prime_\beta L\} \}
\leq 2+\min\{C_\beta, C^\prime_\beta L \}. 
\end{align}
\end{enumerate}
\end{lemma}
\begin{proof}
\emph{Proof of \ref{item:lemma:control-Imbeta_n-i}:} 
Since all $\mu_n$ are real, 
Corollary~\ref{cor:basis}\ref{item:cor-basis-viia} implies $\lim_{n \rightarrow \infty} \Im \lambda_n/|\lambda_n| = 0$. In turn, this implies 
$\lim_{n \rightarrow \infty} \Im (\beta_n/(i |\beta_n|)) = 0$, which implies the first part. The second part follows from $|\beta_n| = \sqrt{|\lambda_n|}$. 

\emph{Proof of \ref{item:lemma:control-Imbeta_n-ii}:} 
\ref{item:lemma:control-Imbeta_n-i} implies $|\Re \beta_n| \leq \Im \beta_n$ for $n \ge n_0$ for some $n_0$. The remaining cases $ n \in \{1,\ldots,n\}$ 
follow from the fact that $\Im \beta_n > 0$ by Lemma~\ref{lem:lambdaisnonreal}. 

\emph{Proof of \ref{item:lemma:control-Imbeta_n-iii}:} From \ref{item:lemma:control-Imbeta_n-i} we get 
$\gamma_{-i \beta_n L} \leq 1+|\Re (\beta_n L)|/\Im (\beta_n L) \leq 1+ C_\beta$. Due to \ref{item:lemma:control-Imbeta_n-i} 
and \ref{item:lemma:control-Imbeta_n-ii} there exists $n_0>0$ such that that $\gamma_{-i \beta_n L} \leq 2$ for all $n \ge n_0$.
For $n \leq n_0$ we estimate 
\begin{align*}
\gamma_{-i \beta_n L} \leq 1+\frac{|\Re (\beta_n L)|}{1 + \Im (\beta_n L)} 
& \leq 1+|\Re (\beta_n L)| \leq 1+\max_{n \in \{1,\ldots,n_0\}} |\beta_n| L =: 1+C^\prime_\beta L . 
\qedhere
\end{align*}

\end{proof}
\begin{remark}
\label{rem:control-Imbeta_n}
The quantity $\gamma_{-i \beta_n \tr}$  in \ref{item:lemma:control-Imbeta_n-iii} is the key quantity in our subsequent stability analysis. 
While $\gamma_{-i \beta_n \tr}$ is bounded uniformly in $n$ and $\tr$ by the constant $C_\beta$ according to \eqref{eq:item:lemma:control-Imbeta_n-iii-10}, 
the bound may be pessimistic in practice for the following reason: For large $n$, the eigenvalues $\lambda_n \approx \mu_n$ so that the $\beta_n$ are 
large and almost purely imaginary. For small $n$, however, in practice ``almost propagating'' modes can exist with very small imaginary part, 
\cite{Mora-Paz_Demkowicz_Taylor_Grosek_Henneking_25}, leading to 
large $C_\beta$. 
\eremk
\end{remark}
\begin{remark}
Compared with the variational formulation for the \emph{straight} waveguide with impedance BC studied in \cite{Demkowicz_Gopalakrishnan_Heuer_24}, the \emph{only} difference in \eqref{eq:eigenproblem} is the presence of the non-constant coefficient $r_1^2\deltanew^2 e^{2\deltanew t}\widetilde\rfi$ in the inertia term and the additional factor $r_2\deltanew$ in the impedance term.
It may not be out of place to mention that for a data corresponding to optical fibers, both coefficients differ from unity in the 4th or 5th significant digit, see \cite{Mora-Paz_Demkowicz_Taylor_Grosek_Henneking_25}.
\eremk
\end{remark}



%
%

\section{Stability analysis for semisimple eigenvalues}
\label{section:analysis}

\subsection{Decoupling the modes. The case without Jordan chains.}
\label{sec:decoupling-no-jordan-chains}
In this section, we assume that the operator $\Ar$ does not have nontrivial Jordan chains and refer
to Section~\ref{section:analysis_with_Jordan_chains} for the general case. 
Let $\{(\lambda_n,{\phi}_n)\}_\ninn$, be the infinite sequence of eigenpairs of $\Ar$, which by Corollary~\ref{cor:basis} form a basis of 
$L^2(r_1,r_2)$. Let $\{(\overline{\lambda}_n, \chi_n)\}_\ninn$ be 
eigenpairs of the adjoint operator $\Ar^\star$ given by 
\begin{subequations}
\label{eq:Arstar}
\begin{align}
\Ar^\star \phi &:= - r (r \phi^\prime)^\prime - \ome^2 \rfi r^2 \phi, \\
\dom(\Ar) &:= \{\phi \in H^2(r_1,r_2)\colon \phi^\prime(r_1) = 0, \quad
\phi^\prime(r_2 ) - i \ome\impedance \phi(r_2) = 0\}.
\end{align}
\end{subequations}
Note that $\chi_n = \overline{\phi}_n$ due to the structure of $\Ar$ and $\Ar^\star$. 
We normalize 
both eigenvectors such that 
\begin{align*}
\|\phi_n\|_{r^{-1}} = 1 = \|\overline{\phi}_n\|_{r^{-1}}  = \|\chi_n\|_{r^{-1}} 
\quad \forall n \in \doubleIN, 
\end{align*}
and set 
\begin{align*}
c_n &:= \innerprod{ \phi_n, \chi_n}_{r^{-1}} = \innerprod{ \phi_n, \overline{\phi}_n}_{r^{-1}}. 
\end{align*}
As will be discussed in more detail in 
Section~\ref{section:analysis_with_Jordan_chains} 
(cf.\ \eqref{eq:dtn-biorthogonality}), we have the biorthogonality
\begin{equation}
\label{eq:biorthogonality-eigenfunctions}
\innerprod{{\phi}_n, {\chi}_m}_{r^{-1}} = 0 \qquad \mbox{ for } n \ne m, 
\end{equation}
which also implies 
\begin{equation}
\innerprod{{\Ar \phi}_n, {\chi}_m}_{r^{-1}} = 0 \qquad \mbox{ for } n \ne m. 
\end{equation}
The strategy to prove~\eqref{eq:stability_system} relies on expanding the solution $u$
into the eigenmodes,
\begin{equation}
\label{eq:series-ansatz}
u(r,\theta) = \sum_{n=0}^\infty p_n(\theta)\,  {\phi}_n(r) 
\end{equation}
and testing in~\eqref{eq:Helmholtz-polar-var}
with $(r,\theta) \mapsto q(\theta) {\chi}_m(r)$, where 
$$
q \in H^1_{(0}(0,\tr) := \{q \in H^1(0,\tr) \colon q(0) = 0 \} \, .
$$
In terms of the operator $\Ar$, this leads to
\begin{align*}
\sum_{n=1}^\infty \innerprod{ \Ar \phi_n,\chi_m}_{r^{-1}} \innerprod{p_n,q}_{L^2(0,\tr)}
&+\innerprod{ \phi_n,\chi_m}_{r^{-1}} \innerprod{p_n^\prime,q^\prime}_{L^2(0,\tr)}\\
&-\innerprod{ \dtn \phi_n,\chi_m}_{L^2(r_1,r_2)} p_n(\tr) \overline{q}(\tr) 
=\innerprod{r f, \chi_m q}_{\Omega}. 
\end{align*}
We use \eqref{eq:DtN-eigenmodes} to express the operator $\dtn$ and the biorthogonality
\eqref{eq:biorthogonality-eigenfunctions} as well as $\lambda_n = -\beta^2_n$ to arrive at
the following variational problem for the coefficients $p_m$: 
\begin{align}
\label{eq:problem-for-p_k-eigenfunctions}
 c_m \left[ -\beta^2_m (p_m,q)_{L^2(0,\tr)} + (p_m^\prime,q^\prime)_{L^2(0,\tr)} - i\beta_m 
p_m(\tr) \overline{q}(\tr)\right] & = (r f, \chi_m q)_{\Omega} \\
\nonumber 
& = (w,\chi_m q)_{H^1(\Omega)} 
\quad \forall q \in H^1_{(0}(0,\tr), 
\end{align}
where we have defined $w \in \Hdir$ to be the representer of $r f$, i.e.,
\begin{align}
\label{eq:riesz-representer-of-rhs}
(r f, v)_{\Omega} & = (w,v)_{H^1(\Omega)} \qquad \forall v \in \Hdir, 
\end{align}
for which we have the equivalence
$\|f\|_{(\Hdir)'} \sim \|w\|_{H^1(\Omega)}$. 

\begin{theorem}
\label{thm:stability-circular}
Fix $r_1$, $r_2$, $\omega$, $\rfi$, $\impedance$, and assume $\tr > c_0$ for some fixed $c_0 >0$. 
Assume that all eigenvalues $\lambda_n$ of $\Ar$ given by \eqref{eq:Ar} have algebraic multiplicity $1$.
Then there is a constant $C > 0$ depending solely on $r_1$, $r_2$, $\omega$, $\rfi$, $\impedance$, $c_0$ such that for any $f \in (\Hdir)'$ the solution 
$u$ of the circular waveguide problem \eqref{eq:Helmholtz-polar} satisfies 
$$
\|u\|_{H^1(\Omega)} \leq C(1+ \min\{C_\beta, C^\prime_\beta \tr\}) \|f\|_{(\Hdir)'}, 
$$
where $C_\beta$, $C^\prime_\beta$ are the constants from Lemma~\ref{lemma:control-Imbeta_n}, which depend solely on $\Ar$. 
\end{theorem}
\begin{proof}
	It is sufficient to bound $\Vert u \Vert_{H^1(\Omega)} \lesssim (1+\min\{C_\beta,C^\prime_\beta \tr\})  \Vert w \Vert_{H^1(\Omega)}$ where 
$w$ is the $\Hdir$ Riesz representer of the actual right-hand side given in \eqref{eq:riesz-representer-of-rhs}. 
We start from \eqref{eq:problem-for-p_k-eigenfunctions}.  

\emph{Step 1:} We show that $\inf_{n\in\mathbb{N}} |c_n| > 0$. 
First, we note that completeness of $\{{\phi}_n\}_\ninn$ in $L^2$ 
implies that $c_n = \innerprod{{\phi}_n,\bar{{\phi}}_n}_{r^{-1}} \ne 0$ for each $n \in \doubleIN$. 
Next, we use from Corollary~\ref{cor:basis} 
that the basis $\{\phi_n\}_\ninn$ is quadratically close to the orthogonal basis $\{\psi_n\}_\ninn$ 
so that 
$\|\psi_n - {\phi}_n\|_{r^{-1}} \rightarrow 0$ 
as $n \rightarrow \infty$ and additionally $\overline{\psi}_n = \psi_n$ so that 
\begin{align}
\label{eq:cj-bounded}
\begin{aligned}
  c_n & = \innerprod{ {\phi}_n, \bar{{\phi}}_n }_{r^{-1}}
= \innerprod{ {\psi}_n + {\phi}_n - {\psi}_n, {\psi}_n + \bar{{\phi}}_n - {\psi}_n}_{r^{-1}} \\
  & \ds = \underbrace{\innerprod{ \psi_n,\psi_n}_{r^{-1}}}_{ = \|\psi_n\|_{r^{-1}}^2=1} + \underbrace{\innerprod{ \psi_n , \bar{{\phi}}_n - \psi_n}_{r^{-1}} + \innerprod{ {\phi}_n - \psi_n, \psi_n }_{r^{-1}} + \innerprod{ {\phi}_n - \psi_n, \bar{{\phi}}_n - \psi_n}_{r^{-1}}}_{ \to 0} \, .
\end{aligned} 
\end{align} 

\emph{Step 2: (Estimate of $p_m$ from \eqref{eq:problem-for-p_k-eigenfunctions}:}
We recall $\gamma_\kappa$ from \eqref{eq:gamma_kappa} and introduce 
\begin{align*}
\gamma_{\operatorname{max}}:= \sup_{n \in \doubleIN } \gamma_{-i \beta_n \tr} 
\stackrel{\text{Lem.~\ref{lemma:control-Imbeta_n}}}{\leq} 2+\min\{C_\beta, C^\prime_\beta \tr\}.  
\end{align*}
The following estimate can be found in
\cite[Lem.~{3.2}]{Melenk_Demkowicz_Henneking_25}
(see also Lemma~\ref{lemma:paper1-lemma3.2} ahead for a generalization) with $I:= (r_1,r_2)$
\begin{align}
\label{eq:Markus_estimate}
\int_0^{\tr} &\vert p_m' \vert^2 \,d\theta  + \vert {\beta}_m \vert^2 \int_0^{\tr} \vert p_m \vert^2 \, d\theta 
\\
\nonumber 
&\leq 
C \vert c_m \vert^{-2} {\gamma_{-i\beta_m \tr}^{2} }  \left\{\int_{0}^{\tr}  {|\beta_m|^{-2}} \vert \innerprod{w(\cdot,\theta),\chi_m}_{H^1(I)} \vert^2 \, d\theta +  \int_0^{\tr}  \vert \innerprod{\partial_\theta w,\chi_m}_{L^2(I)} \vert^2 \, d\theta \right\}, 
\end{align}
where $C$ is independent of $m$. By Step~1, the factor $1/|c_m|$ is uniformly bounded in $m$ and may therefore be absorbed into the constant $C$. 

\emph{Step 3:} 
From Corollary~\ref{cor:basis} we get $|\lambda_n| \sim \|\phi_n\|^2_{H^{1}}$ and therefore from the representation \eqref{eq:series-ansatz} we get
\begin{align}
\label{eq:thm:stabilty-circular-10}
\begin{aligned}
\ds \Vert u \Vert^2_{H^1(I \times (0,\tr))} 
&\ds = \Vert \sum_{n=1}^\infty p_n {\phi}_n \Vert^2_{H^1(I \times (0,\tr))} = \int_0^{\tr} \big\{ \Vert \sum_{n=1}^\infty p_n {\phi}_n \Vert^2_{H^1(I)} + \Vert \sum_{n=1}^\infty p_n^\prime {\phi}_n \Vert^2_{L^2(I)} 
\big \} \, d\theta \\
& \ds \lesssim \sum_{n=1}^\infty \big\{   \vert {\lambda}_n \vert 
   \int_0^{\tr} \vert p_n(\theta) \vert^2 \, d\theta + \int_0^{\tr} \vert p_n^\prime(\theta) \vert^2 \, d\theta \big\}  \\
& \stackrel{(\ref{eq:Markus_estimate})}{\lesssim} 
\gamma_{\operatorname{max}}^2  \sum_{n=1}^\infty \big\{
\frac{1}{\vert {\lambda}_n \vert } \int_0^{\tr} \vert \innerprod{w, {\chi}_n}_{H^1(I)} \vert^2 \, d\theta \big\}
+\int_0^{\tr} \left\vert \innerprod{\partial_\theta w, {\chi}_n}_{L^2(I)} \right\vert^2 \, d\theta 
\qquad 
\\ 
&\ds =  \gamma_{\operatorname{max}}^2 \int_0^{\tr}  \big\{ \sum_{n=1}^\infty
\frac{1}{\vert {\lambda}_n \vert }  \vert \innerprod{w, {\chi}_n}_{H^1(I)} \vert^2
+\left\vert \innerprod{\partial_\theta w, {\chi}_n}_{L^2(I)} \right\vert^2
\big\} \, d\theta. 
\end{aligned}
\end{align}
We want to estimate the right-hand side of \eqref{eq:thm:stabilty-circular-10} by $\Vert w \Vert_{H^1(\Omega)}^2$.
Estimating the second term is easy.
Due to the $L^2$-biorthogonality \eqref{eq:biorthogonality-eigenfunctions} of ${\phi}_n$ and ${\chi}_n$, we have
$$
r \partial_\theta w = \sum_{n=1}^\infty \frac{1}{c_n} \innerprod{ r \partial_\theta w, {\chi}_n}_{r^{-1}} {\phi}_n = \sum_{n=1}^\infty \frac{1}{c_n} \innerprod{\partial_\theta w, \chi_n}_{L^2(I)} \phi_n
$$
and the estimate
$$
\sum_{n=1}^\infty \left\vert \innerprod{\partial_\theta w, {\chi}_n}_{L^2(I)} \right\vert^2 \lesssim \Vert r \partial_\theta w \Vert^2_{L^2(I)}
\lesssim \|\partial_\theta w\|^2_{L^2(I)}
$$
follows from the fact that $\{\phi_n\}_\ninn$ is a Riesz basis of $L^2(r_1,r_2)$. 
Similarly, $\{\chi_n/|\lambda_n|^{1/2}\}_\ninn$ is a Riesz basis of $H^1(I)$ by Corollary~\ref{cor:basis}. Let $\{\widehat{\chi}_n\}_\ninn$ be the basis of $H^1(I)$ biorthogonal to $\{\chi_n/|\lambda_n|^{1/2}\}_\ninn$, i.e., there holds
\begin{align*}
\forall \phitf \in H^1(I) \colon \quad 
\phitf = \sum_{n=1}^\infty \innerprod{\phitf,{\chi}_n/|\lambda_n|^{1/2}}_{H^1(I)} \widehat{\chi}_n. 
\end{align*}
Since $\{\widehat{\chi}_n\}_\ninn$ is again a Riesz basis of $H^1(I)$, we have the norm equivalence 
\begin{align*}
\forall \phitf \in H^1(I) \colon \quad 
\|\phitf\|^2_{H^1(I)} &\sim  \sum_{n=1}^\infty |(\phitf,\chi_n/|\lambda_n|^{1/2})_{H^1(I)}|^2 = \sum_{n=1}^\infty |(u,\chi_n)_{H^1(I)}|^2 |\lambda_n|^{-1}. 
\end{align*}
This allows us to estimate the first term of the right-hand side of \eqref{eq:thm:stabilty-circular-10} by $\Vert w \Vert_{H^1(\Omega)}^2$.
 \end{proof}
%

%
%

\section{Stability analysis for nontrivial Jordan chains}
\label{section:analysis_with_Jordan_chains}


In this section we upgrade the analysis of the preceding Section~\ref{section:analysis} 
to the case when, in the preasymptotic range, some of the eigenvalues may come with nontrivial Jordan chains.
\subsection{Jordan chains}
We can rewrite \eqref{eq:PDEmode} in a more compact form: 
\be
\Ar u - \ptl_\theta^2 u = 0  \quad \mbox{ and hence } \quad  \innerprod{ \Ar u,\phitf}_{r^{-1}} - \innerprod{ \ptl_\theta^2 u,\phitf} _{r^{-1}} = 0 \quad \forall \phitf \in \dom(\Ar^\ast) \, .
\label{eq:operator_form}
\ee
A Jordan chain corresponding to the eigenvalue $\lambda_n$ of $\Ar$ is denoted 
$\{(\lambda_n\, , \phi_{n,j})\}_{j=1}^{J_n}$, and those corresponding to the eigenvalue $\bar{\lambda}_n$ of $\Ar^\ast$ are denoted 
$\{(\bar{\lambda}_n\, , \chi_{n,j}) = (\bar{\lambda}_n\, ,\bar{\phi}_{n,j})\}_{j=1}^{J_n}$; they satisfy 
\begin{align*}
\Ar \phi_{n,j} &= \lambda_n \phi_{n,j} ( + \phi_{n,j-1} ) \, , \quad j=1,\ldots, J_n , 
\end{align*}
where, for $j=1$, the term in parentheses is missing as $\phi_{n,1}$ is the eigenvector. (Below, we will set $\phi_{n,0}:= 0$ to simplify the notation.)
An analogous formula holds for the root 
vectors $\{\chi_{n,j}\}_{j=1}^{J_n}$ associated with the eigenvalue $\overline{\lambda_n}$. The pairs 
$(\lambda_n,\phi_{n,1})$ and 
$(\overline{\lambda_n}, \chi_{n,1})$ are the eigenpairs for the corresponding Jordan chain of length $J_n$. It is shown in Appendix~\ref{appendix:3}
that indeed corresponding Jordan blocks of $\Ar$ and $\Ar^\star$ have the same length. We recall from Lemma~\ref{lem:lambdaisnonreal} that the geometric multiplicity 
of the eigenvalues of $\Ar$ and $\Ar^\star$ is indeed $1$ so that there is only one Jordan block per eigenvalue. 
Importantly, as it is shown in Appendix~\ref{appendix:3},
one has the root vector sets $\{\phi_{n,j} \colon n \in \doubleIN, j=1,\ldots,J_n\}$ and 
$\{\chi_{n,j} \,|\, n \in \doubleIN, j=1,\ldots,J_n\}$, which are biorthogonal 
in the following sense: 
\begin{subequations}
\label{eq:dtn-biorthogonality}
\begin{align}
\label{eq:dtn-biorthogonality-a}
\innerprod{ \phi_{n,j}, \chi_{n',j'}} _{r^{-1}} & = 0 \quad \forall n \ne n', \quad j =1,\ldots,J_n, \quad j' = 1,\ldots,J_{n'}, \\
\label{eq:dtn-biorthogonality-b}
\innerprod{ \phi_{n,j}, \chi_{n,J_n +1 - j'}} _{r^{-1}} & = 0 \quad j \ne j', \\
\label{eq:dtn-biorthogonality-c}
c_{n,j} := \innerprod{ \phi_{n,j}, \chi_{n,J_n + 1- j}} _{r^{-1}} & \ne  0 \quad j=1,\ldots,J_n. 
\end{align}
\end{subequations}
Additionally, the functions $\{\phi_{n,j}\colon n \in \doubleIN, j \in \{1,\ldots,J_n\}\}$ (and analogously 
$\{\chi_{n,j}\colon n \in \doubleIN, j \in \{1,\ldots,J_n\}\}$) form Riesz bases in $L^2$ and $H^1$ 
as shown in Corollary~\ref{cor:basis}. We point out that in contrast to the setting of Corollary~\ref{cor:basis}, where
the root vectors are indexed by a single index $n$, we now use index pairs $(n,j)$ to emphasize the Jordan block structure.
\begin{lemma}
\label{lemma:control-cnj}
It holds that $\inf \{ |c_{n,j}| \colon n \in \doubleIN, \quad j=1,\ldots,J_n\} > 0$. 
\end{lemma}
\begin{proof}
Follows by the arguments given in Step~1 of the proof of Theorem~\ref{thm:stability-circular}. 
\end{proof}
\subsection{The $\dtn$-operator}
\label{sec:dtn-operator}
We expand the solution $u$ of \eqref{eq:operator_form} (or, equivalently,
\eqref{eq:PDEmode})
in terms of the root vectors, i.e., make the 
ansatz 
$$
u(r,\theta) = \sum_{n=1}^\infty \sum_{j=1}^{J_n} p_{n,j}(\theta) \phi_{n,j}(r) \, ,
\qquad (r,\theta) \in (r_1,r_2) \times (\tr,\infty), 
$$
substitute this ansatz into equation \eqref{eq:operator_form}, and test with
$\phitf = \chi_{m,k},\, k=1,\ldots,J_m$, to get 
\begin{equation}
\label{eq:dtn-derivation-20}
\sum_{n=1}^\infty \sum_{j=1}^{J_n} p_{n,j} \innerprod{ \underbrace{\Ar \phi_{n,j}}_{ = \lambda_n \phi_{n,j}  (+   \phi_{n,j-1})}, \chi_{m,k}} _{r^{-1}}  
- \sum_{n=1}^\infty \sum_{j=1}^{J_n} p_{n,j}'' \innerprod{ \phi_{n,j},\chi_{m,k}} _{r^{-1}} = 0 \, .
\end{equation}
By the biorthogonality properties \eqref{eq:dtn-biorthogonality}, we get in \eqref{eq:dtn-derivation-20}
that $\chi_{n,1}$ couples only with $\phi_{n,J_n}$, whereas for $j =2,\ldots,J_n$, the function $\chi_{n,j}$ couples both with $\phi_{n,J_n-j+1}$ and $\phi_{n,J_n-j+2}$.
As a result, for the Jordan chain corresponding to $\lambda_n$, we obtain a system of coupled ODEs for the coefficients $p_{n,j}(\theta)$:
\begin{subequations}
\label{eq:dtn-derivation-30}
\begin{align}
\label{eq:dtn-derivation-30-a}
-\beta_n^2 p_{n,J_n} - p_{n,J_n}'' & = 0\, , \\
\label{eq:dtn-derivation-30-b}
-\beta_n^2 p_{n,J_n-1} - p_{n,J_n-1}'' & = -p_{n,J_n}\, , \\
\nonumber 
\vdots\\
\label{eq:dtn-derivation-30-d}
-\beta_n^2 p_{n,2} - p_{n,2}'' & = -p_{n,3}\, , \\
\label{eq:dtn-derivation-30-e}
-\beta_n^2 p_{n,1} - p_{n,1}'' & = -p_{n,2}\, .
\end{align}
\end{subequations}
To derive the $\dtn$-operator we have to elaborate the map
$\sum_{j=1}^{J_n} p_{n,j}(\tr) \phi_{n,j}$ $\mapsto$
$\sum_{j=1}^{J_n} p^\prime_{n,j}(\tr) \phi_{n,j}$. 
To that end, 
we solve successively the equations \eqref{eq:dtn-derivation-30} for given 
values $p_{n,j}(\tr)$, $j=1,\ldots,J_n$. We start with \eqref{eq:dtn-derivation-30-a}. 
The general solution for $p_{n,J_n}$ on $(\tr,\infty)$ is 
$$
p_{n,J_n}(\theta) = A_{J_n} e^{i \beta_n (\theta-\tr)} + \underbrace{B_{J_n} e^{-i \beta_n (\theta-\tr)} }_{\text{incoming wave}} \, .
$$
Dropping the incoming wave contribution, we obtain:
$$
p_{n,J_n}(\theta) = A_{J_n} e^{i \beta_n (\theta-\tr)}, \, 
\qquad A_{J_n} = p_{n,J_n}(\tr). 
$$
This results in a $\DtN$ condition similar to that in Section~\ref{section:problem_formulation}: 
$$
p_{n,J_N}'(\tr) = i \beta_n A_{J_n} = i \beta_n p_{n,J_n}(\tr) . 
$$
Next, the equation for the coefficient $p_{n,J_n-1}$ corresponding to the $(J_n-1)$-th root vector is:
$$
- \beta_n^2 p_{n,J_n-1} - p_{n,J_n-1}'' = - p_{n,J_n} 
= - A_{J_n} e^{i \beta_n (\theta-\tr)}
$$
with the general solution (with the incoming wave eliminated):
$$
p_{n,J_n-1} = \big[ \frac{A_{J_n}}{2 i \beta_n} (\theta -\tr)+ A_{J_n-1} \big] e^{i \beta_n (\theta-\tr)} 
\quad \Rightarrow \quad p_{n,J_n-1}' =  \frac{A_{J_n}}{2 i \beta_n}  e^{i \beta_n (\theta-\tr)}  + i \beta_n p_{n,J_n-1} \, .
$$
Consequently,
$$
p_{n,J_n-1}'(\tr) = i \beta_n p_{n,J_n-1}(\tr) + \frac{p_{n,J_n}(\tr)}{2 i \beta_n}.  
$$
The $\DtN$ condition for $p_{n,J_n-1}$ results thus in a non-homogeneous impedance BC where $p_{n,J_n}$ enters as a load.

Similarly, the equation for the coefficient $p_{n,J_n-2}$ corresponding to the $(J_n-2)$-th root vector is:
$$
- \beta_n^2 p_{n,J_n-2} - p_{n,J_n-2}'' = -p_{n,J_n-1} = -\big[ \frac{A_{J_n}}{2 i \beta_n} (\theta -\tr)+ A_{J_n-1} \big] e^{i \beta_n (\theta-\tr)} , 
$$
which gives:
$$
\begin{array}{rlll}
p_{n,J_n-2} & \ds = \left[ \frac{A_{J_n}}{(2 i \beta_n)^2} \frac{(\theta-\tr)^2}{2} + \left( -\frac{A_{J_n}}{(2 i \beta_n)^3} + \frac{A_{J_n-1}}{2 i \beta_n} \right) (\theta -\tr) + A_{J_n-2} \right] e^{ i \beta_n (\theta-\tr)} \\[12pt]
& \ds = \left\{ A_{J_n} \left[ \frac{1}{(2 i \beta_n)^2} \frac{(\theta-\tr)^2}{2} - \frac{1}{(2  i \beta_n)^3} (\theta -\tr) \right] + A_{J_n-1} \frac{1}{2 i \beta_n} (\theta -\tr)+ A_{J_n-2} \right\} e^{i \beta_n (\theta-\tr)} 
\end{array}
$$
and
$$
p_{J_n-2}' =  \left\{ A_{J_n} \left[ \frac{(\theta-\tr)}{(2 i \beta_n)^2}   - \frac{1}{(2  i \beta_n)^3}  \right] + \frac{A_{J_n-1}}{2 i \beta_n}   \right\} e^{i \beta_n (\theta-\tr)} + i \beta_n p_{J_n-2} \, .
$$
This results in 
\begin{align*}
p_{n,J_n-2}'(\tr) &= i \beta_n p_{n,J_n-2}(\tr) - \frac{A_{J_n}}{(2 i \beta_n)^3} + \frac{A_{J_n-1}}{2 i \beta_n} \\
& = i \beta_n p_{n,J_n-2}(\tr) - \frac{p_{n,J_n}(\tr)}{(2 i \beta_n)^3}  + \frac{p_{n,J_n-1}(\tr)}{2 i \beta_n} \, .
\end{align*}
Not only that we get a non-homogeneous BC but both $p_{n,J_n}$ and $p_{n,J_n-1}$ enter as a load.
Proceeding in this way, we arrive at
$$
p^\prime_{n,J_n-j}(\tr) =  i \beta_n p_{n,J_n-j }(\tr)
+ \sum_{k=0}^{j-1} (-1)^{j-1-k}  d_{k,j} \frac{p_{n,J_n-k}(\tr)}{(2 i \beta_n)^{2j-2k-1}}, 
\qquad j=0,1,\ldots,J_n-1, 
$$
for some coefficients $d_{k,j} \in \doubleIR$. 
\begin{remark}
\label{remk:catalan}
As shown in \cite{Taylor_Zhang_Demkowicz_25}, 
the coefficients $d_{k,j}$ are explicitly available: $d_{k,j} = C_{j-1-k}$ with the Catalan numbers 
$$
C_{j-1-k}= 
\frac{1}{2j - 2k - 1} \left( \begin{array}{c} 2j - 2k - 1\\j-k\end{array} \right) \, .
$$
\eremk
\end{remark}
In view of $\dtn u = r^{-1} \partial_\theta u$
the $\dtn$-operator in the presence of Jordan chains for the transversal operator $\Ar$ takes the following form: 
\begin{align}
\label{eq:dtn-operator-jordan-chains}
\dtn \sum_{n=1}^\infty \sum_{j=1}^{J_n} p_{n,j} \phi_{n,j}
= \frac{1}{r} \sum_{n=1}^\infty \sum_{j=1}^{J_n} 
\underbrace{
\left( i \beta_n p_{n,j} 
+ \sum_{k=j+1}^{J_n} (-1)^{-j+k-1} d_{J_n-k,J_n-j} \frac{p_{n,k}}{(2i\beta_n)^{2(k-j)-1}}\right)
           }_{=:F_{n,j}}
\phi_{n,j}. 
\end{align}
\begin{lemma}
\label{lemma:mapping-properties-DtN}
The linear operator $\dtn$ defined by 
\eqref{eq:dtn-operator-jordan-chains} for each root vector
set $\{\phi_{n,j}\}_{j=1}^{J_n}$, $n \in \doubleIN$, 
can be extended uniquely to a bounded linear operator
$H^1(r_1,r_2) \rightarrow L^2(r_1,r_2)$ and
$L^2(r_1,r_2) \rightarrow \left(H^1(r_1,r_2)\right)^\prime$.
By interpolation, it is thus a bounded linear operator
$H^{1/2}(r_1,r_2) \rightarrow \left(H^{1/2}(r_1,r_2)\right)^\prime$.
The operator norms 
depend only upon the operator $\Ar$, 
i.e., $r_1$, $r_2$, $\omega$, $\rfi$, $\impedance$.
\end{lemma}
\begin{proof} 
We make use of the shorthand $F_{n,j}$ from 
(\ref{eq:dtn-operator-jordan-chains}).  

\emph{Step 1:} There is $C > 0$ independent of $n$ such that 
\begin{align}
\label{eq:lemma:mapping-properties-DtN-5}
\sum_{j=1}^{J_n} |F_{n,j}|^2 \|\phi_{n,j}\|^2_{L^2} \leq C |\beta_n|^2 
 \sum_{j=1}^{J_n} |p_{n,j}|^2 \|\phi_{n,j}\|^2_{L^2}. 
\end{align}
To see this, first note that $J_n = 1$ for $n \ge N_0$ by 
Corollary~\ref{cor:basis}\ref{item:cor-basis-iii}. 
so that the sum in  
(\ref{eq:lemma:mapping-properties-DtN-5}) and in 
(\ref{eq:dtn-operator-jordan-chains}) reduce to a single term and any
$C \ge 1$ may be chosen. 
For the remaining finitely many cases $n \leq N_0$, we observe that 
by  
Corollary~\ref{cor:basis}\ref{item:cor-basis-vi}
one has $\inf_{n} |\beta_n| > 0$ so that 
\begin{align}
\label{eq:lemma:mapping-properties-DtN-7} 
|F_{n,j}| \|\phi_{n,j}\|_{L^2} \leq C |\beta_n| \sum_{k=j}^{J_n} |p_{n,k}| 
\|\phi_{n,j}\|_{L^2} 
\leq C |\beta_n| \sqrt{J_n} 
\max_{k=1}^{J_n} \frac{\|\phi_{n,j}\|_{L^2}}{\|\phi_{n,k}\|_{L^2}} 
\left(\sum_{k=1}^{J_n} |p_{n,k}|^2 \|\phi_{n,k}\|^2_{L^2} \right)^{1/2}. 
\end{align}
Since the functions $\{\phi_{n,j}\}_{j=1}^{J_n}$ do not vanish 
and only finitely many cases for $n$ need to be considered in 
(\ref{eq:lemma:mapping-properties-DtN-7}), the desired result 
(\ref{eq:lemma:mapping-properties-DtN-5}) follows. 

\emph{Step 2:} Expand $\phi \in H^1(r_1,r_2)$ as
$\phi = \sum_{n=1}^\infty \sum_{j=1}^{J_n} p_{n,j} \phi_{n,j}$. 
Since $\{\phi_{n,j}/\|\phi_{n,j}\|_{H^1}\}_{j,n}$ is a Riesz basis of $H^1(r_1,r_2)$ by Corollary~\ref{cor:basis}\ref{item:cor-basis-v} 
and since $|\beta_n|^2 \|\phi_{n,j}\|^2_{L^2} \sim \|\phi_{n,j}\|^2_{H^1}$ by Corollary~\ref{cor:basis}\ref{item:cor-basis-vi}
we get 
\begin{align} 
\label{eq:lemma:mapping-properties-DtN-20}
\|\phi\|^2_{H^1} 
\sim \sum_{n=1}^\infty \sum_{j=1}^{J_n} |p_{n,j}|^2 \|\phi_{n,j}\|^2_{H^1}
\sim \sum_{n=1}^\infty \sum_{j=1}^{J_n} |\beta_n|^2 |p_{n,j}|^2 \|\phi_{n,j}\|^2_{L^2}. 
\end{align} 
Using that $\{\phi_{n,j}\colon j=1,\ldots,J_n, n \in \doubleIN\}$ is a Riesz-basis 
in $L^2$ by Corollary~\ref{cor:basis}\ref{item:cor-basis-iv}
with the shorthand $F_{n,j}$ from 
\eqref{eq:dtn-operator-jordan-chains}  we arrive at 
\begin{align*}
\|\dtn \phi\|^2_{L^2} &\lesssim \sum_{n=1}^\infty \sum_{j=1}^{J_n} 
|F_{n,j}|^2 \|\phi_{n,j}\|^2_{L^2} 
\stackrel{
(\ref{eq:lemma:mapping-properties-DtN-5})}{ \lesssim }
\sum_{n=1}^\infty
\sum_{j=1}^{J_n} |\beta_n |^2 |p_{n,j}|^2 \|\phi_{n,j}\|^2_{L^2} 
\stackrel{(\ref{eq:lemma:mapping-properties-DtN-20})}{ \lesssim }
\|\phi\|^2_{H^1}. 
\end{align*}
Hence, $\dtn\in L(H^1(r_1,r_2),L^2(r_1,r_2))$. 

\emph{Step 3:} $\dtn$ can be extended to a bounded linear  operator 
$L^2(r_1,r_2) \rightarrow \left(H^1(r_1,r_2)\right)^\prime$. To see this, 
expand $\phi \in L^2(r_1,r_2)$ and $\phitf \in H^1(r_1,r_2)$ as 
$\phi = \sum_{n=1}^\infty \sum_{j=1}^{J_n} p_{n,j} \phi_{n,j}$ and 
$\phitf = \sum_{n=1}^\infty \sum_{j=1}^{J_n} q_{n,j} \chi_{n,j}$.  By the Riesz basis properties
asserted in Corollary~\ref{cor:basis}\ref{item:cor-basis-iv}, \ref{item:cor-basis-v}, \ref{item:cor-basis-viii}
we have 
\begin{align}
\label{eq:lemma:mapping-properties-DtN-50}
\|\phi\|^2_{L^2} &\sim 
\sum_{n=1}^\infty \sum_{j=1}^{J_n} |p_{n,j}|^2 \|\phi_{n,j}\|^2_{L^2(r_1,r_2)}, 
&
\|\phitf\|^2_{H^1} &\sim 
\sum_{n=1}^\infty \sum_{j=1}^{J_n} |q_{n,j}|^2 \|\chi_{n,j}\|^2_{H^1(r_1,r_2)}. 
\end{align}
We estimate, using the shorthand $F_{n,j}$ from 
\eqref{eq:dtn-operator-jordan-chains} and the biorthogonality \eqref{eq:dtn-biorthogonality}
\begin{align*}
\left| \innerprod{ \dtn \phi, r \phitf} _{r^{-1}}\right|  & = 
\left| \sum_{n=1}^\infty \sum_{j=1}^{J_n} F_{n,j} \overline{q}_{n,J_n+1-j} 
\innerprod{ \phi_{n,j}, \chi_{n,J_n + 1 - j}} _{r^{-1}} 
\right| \\
& 
\lesssim 
\sum_{n=1}^\infty \sum_{j=1}^{J_n} 
|F_{n,j}| \|\phi_{n,j}\|_{r^{-1}} |q_{n,J_n + 1 -j}| \|\chi_{n,J_n + 1 -j}\|_{r^{-1}}  \\
&\lesssim 
\left(\sum_{n=1}^\infty \sum_{j=1}^{J_n} |\beta_n|^{-2} |F_{n,j}|^2 \|\phi_{n,j}\|^2_{L^2}
\right)^{1/2} 
\left(\sum_{n=1}^\infty \sum_{j=1}^{J_n} |\beta_n|^{2} |q_{n,j}|^2 \|\chi_{n,j}\|^2_{L^2} 
\right)^{1/2}  \\
&\stackrel{ \eqref{eq:lemma:mapping-properties-DtN-5} }
{\lesssim } 
\left(\sum_{n=1}^\infty \sum_{j=1}^{J_n} |p_{n,j}|^2 \|\phi_{n,j}\|^2_{L^2} 
\right)^{1/2}  
\left(\sum_{n=1}^\infty \sum_{j=1}^{J_n} |\beta_n|^{2} |q_{n,j}|^2 
\frac{\|\chi_{n,j}\|^2_{L^2} }{\|\chi_{n,j}\|^2_{H^1}} 
\|\chi_{n,j}\|^2_{H^1} 
\right)^{1/2}  \\
&\stackrel{
\eqref{eq:lemma:mapping-properties-DtN-50}, 
\text{Cor.~\ref{cor:basis}\ref{item:cor-basis-vi}}
}
{\lesssim }
\|\phi\|_{L^2} \|\phitf\|_{H^1}. 
\end{align*}
\end{proof}
%

\subsection{Stability analysis}
\label{sec:stability-for-jordan-chains}
We generalize \cite[Lem.~{3.2}]{Melenk_Demkowicz_Henneking_25}: 
\begin{lemma}
\label{lemma:paper1-lemma3.2}
Let $L  > 0$ and $\kappa \in \doubleIC_{\ge 0}:= \{z \in \doubleIC\,|\, \Re z \ge 0\}$ satisfy  $|\kappa| L \ge c_0 > 0$.  
Define on $H^1(0,L)$ 
\begin{align*}
a^{1D}_\kappa (u,v) &:= \innerprod{u^\prime,v^\prime}_{L^2(0,L)} + \kappa^2 \innerprod{u,v}_{L^2(0,L)} + \kappa u(L) \overline{v}(L),  \\
\|u\|^2_{1,|\kappa|} &:= \|u^\prime\|^2_{L^2(0,L)} + |\kappa|^2 \|u\|^2_{L^2(0,L)}. 
\end{align*}
Let $V \in \{H^1(0,L),H^1_{(0}(0,L)\}$. 
Then, there is $C > 0$ independent of $\kappa$, $L$ (but dependent on $c_0$) such that for all $w_1$, $w_2 \in L^2(0,L)$ and $g \in \doubleIC$ the 
solution $p \in V$ of 
\begin{align*}
a^{1D}_\kappa(p,v) &= \innerprod{w_1,v}_{L^2(0,L)} + \innerprod{w_2,v^\prime}_{L^2(0,L)} + g \overline{v}(L) \quad \forall v \in V
\end{align*}
satisfies 
\begin{align*}
\|p\|_{1,|\kappa|} \leq C \gamma_{\kappa L} \left[ |\kappa|^{-1} \|w_1\|_{L^2(0,L)} + \|w_2\|_{L^2(0,L)} + |\kappa|^{-1/2} |g|\right]
\end{align*}
with $\gamma_{\kappa L}$ given by \eqref{eq:gamma_kappa}. 
\end{lemma}
\begin{proof}
\cite[Lem.~{3.2}]{Melenk_Demkowicz_Henneking_25} covers the case $g = 0$. To include the case $g \ne 0$, we note that 
\cite[Lem.~{3.2}]{Melenk_Demkowicz_Henneking_25} establishes the inf-sup constants for the sesquilinear form $a^{1D}_{\kappa}$. To include the 
impedance boundary condition involving $g$, we observe that the condition $|\widehat{\kappa}|:= |\kappa| L \ge c_0 > 0$ 
and the Sobolev embedding theorem on the reference interval $\hat{I} = (0,1)$ imply
\begin{equation}
\label{eq:sobolev-embedding-reference}
\forall \widehat{v} \in H^1(0,1) \colon \quad 
\sqrt{|\widehat{\kappa}|} \|\widehat{v}\|_{L^\infty(0,1)} 
\leq C \left[ \|\widehat{v}\|_{H^1(\hat{I})} + |\widehat\kappa| \|\widehat{v}\|_{L^2(\hat{I})}\right] 
\end{equation}
for a constant $C > 0$ depending solely on $c_0$.  Inspection of the procedure in 
\cite[Lem.~{3.2}]{Melenk_Demkowicz_Henneking_25} then shows the result. 
We finally point out that the constant $\gamma_{\kappa L}$ of 
\eqref{eq:gamma_kappa} differs from the corresponding one in 
\cite[Lem.~{3.2},(i)]{Melenk_Demkowicz_Henneking_25} by a factor $L$ due to a typographical error there.  
\end{proof}
In the following, we will use Lemma~\ref{lemma:paper1-lemma3.2} with $L = \tr$ and $\kappa = -i \beta_n$ so that we will work with the norm
\begin{align}
\|v\|^2_{1,|\beta_n|} & = |\beta_n|^2 \|v\|^2_{L^2(0,\tr)} + \|v^\prime\|^2_{L^2(0,\tr)}. 
\end{align}
The analysis of the variational formulation \eqref{eq:Helmholtz-polar-var} proceeds 
similarly to the procedure in Section~\ref{sec:decoupling-no-jordan-chains}. Using the representation $w \in \Hdir$ of the right-hand side 
of \eqref{eq:Helmholtz-polar-var} already defined in \eqref{eq:riesz-representer-of-rhs}, we have to solve 
\begin{align*}
\ap(u,v) = (w,v)_{H^1(\Omega)} \quad \forall v \in \Hdir. 
\end{align*}
First, we calculate for $p$, $q \in H^1_{(0}(0,\tr)$ and ${\phi} \in \dom(\Ar)$, ${\chi} \in \dom(\Ar^\star)$ 
\begin{align*}
&\ap( p(\theta) {\phi}(r), q(\theta) {\chi}(r)) = \\
& \quad \innerprod{ \Ar {\phi}, {\chi}}_{r^{-1}} \innerprod{p,q}_{L^2(0,\tr)} 
+ \innerprod{ {\phi}, {\chi}}_{r^{-1}} \innerprod{p^\prime,q^\prime}_{L^2(0,\tr)}
- \innerprod{ \dtn p(\tr) \phi, q(\tr) \chi }_{L^2(r_1,r_2)}. 
\end{align*}
Next, for each Jordan chain, the $\dtn$-operator from \eqref{eq:dtn-operator-jordan-chains} takes the following 
form 
\begin{align*}
\innerprod{\dtn \sum_{j=1}^{J_n} p_{n,j} \phi_{n,j} , \chi}_{L^2(r_1,r_2)} & = 
 \innerprod{ \sum_{j=1}^{J_n} F_{n,j}  \phi_{n,j},\chi  }_{r^{-1}} 
 = \sum_{j=1}^{J_n} \sum_{k=j}^{J_n} D_{k,j,n} p_{n,k} \innerprod{ \phi_{n,j}, \chi}_{r^{-1}}, 
\end{align*}
where 
\begin{align*}
D_{j,j,n} & = i {\beta}_n, 
& 
D_{k,j,n} & = (-1)^{-j+k-1} \frac{d_{J_n - k, J_n-j} }{(2i {\beta}_n)^{2(k-j)-1}}, 
\quad k > j. 
\end{align*}
We consider a modal decomposition of the solution $u$ of \eqref{eq:Helmholtz-polar-var}: 
\begin{equation*} 
u(r,\theta) = \sum_{n=1}^\infty \sum_{j=1}^{J_n} p_{n,j}(\theta) {\phi}_{n,j}(r), 
\qquad p_{n,j} \in H^1_{(0}(0,\tr). 
\end{equation*} 
Inserting this expansion in \eqref{eq:Helmholtz-polar-var} and testing with $(r,\theta) \mapsto {\chi}_{n,J_n+1-j'}(r) q_{j'}(\theta)$, 
where $j'=1,\ldots,J_n$ and $q_{j'} \in H^1_{(0}(0,\tr)$, 
we get in view of the biorthogonality satisfied by the functions ${\phi}_{n,j}$, ${\chi}_{n,j}$ (cf.\ \eqref{eq:dtn-biorthogonality})
and the abbreviations $p_{n,J_n+1} \equiv 0$, ${\phi}_{n,0} = 0$, as well as the definition of the $c_{n,j}$ from \eqref{eq:dtn-biorthogonality-c}
\begin{align*}
	& \ap(\sum_{n=1}^\infty\sum_{j=1}^{J_n} p_{n,j} \phi_{n,j}, q_{j'} \chi_{n,J_n+1-j'})   \\
&=  
\sum_{j=1}^{J_n} \innerprod{\Ar {\phi}_{n,j}, {\chi}_{J_n+1-j'}}_{r^{-1}} \innerprod{ p_{n,j} q_{j'}}_{L^2} + 
\innerprod{ {\phi}_{n,j},{\chi}_{n,J_n+1-j'}}_{r^{-1}} \innerprod{ p^\prime_{n,j}, q_{j'}^\prime} _{L^2} \\
& \qquad \mbox{}
- \innerprod{ {\phi}_{n,j}, {\chi}_{n,J_n+1-j'}}_{r^{-1}} \sum_{k=j}^{J_n} D_{k,j,n} p_{n,k}(\tr) \overline{q}_{j'}(\tr)  \\
&  = 
\sum_{j=1}^{J_n} \innerprod{ {\lambda}_n{\phi}_{n,j} +{\phi}_{n,j-1} , {\chi}_{J_n+1-j'}}_{r^{-1}} \innerprod{ p_{n,j},q_{j'}}_{L^2} + 
c_{n,j'} \delta_{j,j'}  \innerprod{ p^\prime_{n,j}, q_{j'}^\prime}_{L^2} 
- c_{n,j'} \delta_{j,j'}  \sum_{k=j}^{J_n} D_{k,j,n} p_{n,k}(\tr) \overline{q}_{j'}(\tr)  \\
& = \sum_{j=1}^{J_n} c_{n,j} \delta_{j,j'} \left[ {\lambda}_n \innerprod{ p_{n,j}, q_{j'}}_{L^2} + \innerprod{ p^\prime_{n,j}, q^\prime_{j'}}_{L^2} 
- i {\beta}_n p_{n,j}(\tr) \overline{q}_{j'}(\tr) \right] \\
& \quad \mbox{} + 
\bigl(\sum_{j=1}^{J_n} c_{n,j'} \delta_{j-1,j'} \innerprod{ p_{n,j-1},q_{j'}}_{L^2}\bigr) 
     - c_{n,j'}  \sum_{k=j'+1}^{J_n} D_{k,j',n} p_{n,k}(\tr) \overline{q}_{j'}(\tr)   \\
& = c_{n,j'} \left[ -{\beta}_n^2 \innerprod{ p_{n,j'}, q_{j'}}_{L^2} + \innerprod{ p^\prime_{n,j'}, q^\prime_{n,j'}}_{L^2} - i {\beta}_n p_{n,k}(\tr) \overline{q}_{j'}(\tr) 
\right] \\
& \qquad \mbox{} + c_{n,j'} \Bigl[ \innerprod{p_{n,j'+1}, q_{j'}}_{L^2} - \sum_{k=j'+1}^{J_n}  D_{k,j',n} p_{n,k}(\tr) \overline{q}_{j'}(\tr)\Bigr] \\
&\stackrel{!}{=} \innerprod{ w,{\chi}_{J_n+1-j'} q_{j'}}_{H^1(\Omega)} 
 = \innerprod{ w_{1,n,j'}, q_{j'}}_{L^2} + \innerprod{ w_{2,n,j'}, q^\prime_{j'}}_{L^2}, 
\end{align*}
where we set 
\begin{align*}
w_{1,n,j'}(\theta)&:= \innerprod{ w(\cdot,\theta),{\chi}_{n,J_n+1-j'}}_{L^2(r_1,r_2)} + \innerprod{\partial_r w(\cdot,\theta), {\chi}^\prime_{n,J_n+1-j'}}_{L^2(r_1,r_2)}, \\  
w_{2,n,j'}(\theta)&:= \innerprod{\partial_\theta w(\cdot,\theta),{\chi}_{n,J_n+1-j'}}_{L^2(r_1,r_2)}.  
\end{align*}
The above system is an upper triangular system that can be solved by backsubstitution: with the sesquilinear form $a^{1D}_{-i {\beta}_n}$ of 
Lemma~\ref{lemma:paper1-lemma3.2} we have to find, for $j'=J_n,J_n-1,\ldots,1$, the functions $p_{n,j'}$ given by 
\begin{align}
\label{eq:pnj-block}
a^{1D}_{-i {\beta}_n}(p_{n,j'},q) & = 
c_{n,j'}^{-1} \left[
\innerprod{w_{1,n,j'},q}_{L^2} + \innerprod{w_{2,n,j'},q'}_{L^2}\right] - \innerprod{p_{n,j'+1},q}_{L^2} + \sum_{k=j'+1}^{J_n} D_{k,j',n}  p_{n,k}(\tr) \overline{q}(\tr) 
\end{align}
for all $q \in H^1_{(0}(0,\tr)$. 
Solving for these functions $p_{n,j'}$, we have 
\begin{lemma}
\label{lemma:backsubstitution-Jordan-block}
Assume $ \tr |{\beta}_n| \ge c_0>0$. Then
there is $\widehat{C} > 0$ depending only on $\Ar$  and $c_0$  such that
for $j=1,\ldots,J_n$, the solution $p_{n,j}$ of \eqref{eq:pnj-block} satisfy 
for $m=1,\ldots,J_n$
\begin{align*} 
\sum_{j=1}^m \|p_{n,J_n+1-j}\|_{1,|{\beta}_n|} 
 & \leq \widehat{C} \gamma_{-i {\beta}_n \tr}  \sum_{j=1}^{m} 
&\left(1 + \widehat{C} \frac{\gamma_{-i{\beta}_n \tr}}{|{\beta}_n|}\right)^{m-j}   
\left[ |{\beta}_n|^{-1} \|w_{1,n,J_n+1-j}\|_{L^2} + \|w_{2,n,J_n+1-j}\|_{L^2} \right] . 
\end{align*} 
\end{lemma}
\begin{proof}
\emph{Step~1:} 
We observe that the numbers $|c_{n,j}|^{-1}$ are uniformly bounded by Lemma~\ref{lemma:control-cnj} so that 
their impact may be absorbed in generic constants. We also note that $|{\beta}_n | \sim n$ so that we may assume 
$\inf_{n} |{\beta}_n| > 0$ (cf.\ Corollary~\ref{cor:basis}\ref{item:cor-basis-vi}). 
Among others, this implies 
\begin{equation}
\label{eq:estimate-D}
|D_{k,j',n}| \leq C |{\beta}_n|^{-1} , 
\qquad k =  j'+1,\ldots,J_n. 
\end{equation}
\emph{Step~2:} Let $\{a_j\}_{j \in \doubleIN}$, $\{b_j\}_{j \in \doubleIN} \subset \doubleIR_{\ge 0}$, and 
set $S_m:=\sum_{j=1}^m a_j$. Assume that for some $G \ge 0$ there holds $a_{j+1} \leq G (S_j + b_{j+1})$ for all $j$.
Then, by induction on $m$ one has  
\begin{align}
\label{eq:recurion-for-sum}
S_{m+1} \leq (1+G)^m S_1 + \sum_{j=1}^m G b_{j+1} (1+G)^{m-j}\qquad  \forall m \in \doubleIN_0.
\end{align}
\emph{Step 3:}
Using \eqref{eq:sobolev-embedding-reference} and a scaling argument, one arrives at the Sobolev embedding theorem
\begin{align}
\label{eq:sobolev-embedding}
|{\beta}_n|^{1/2} \|v\|_{L^\infty(0,\tr)} \leq C \|v\|_{1,|{\beta}_n|} \qquad \forall v \in H^1(0,\tr). 
\end{align}
\emph{Step 4:} From Lemma~\ref{lemma:paper1-lemma3.2} with $\kappa = - i {\beta}_n$ and $L = \tr$ and the embedding 
\eqref{eq:sobolev-embedding}, we get for $j'=J_n,J_n-1,\ldots,1$
\begin{align*}
\|p_{n,j'}\|_{1,|{\beta}_n|} &\lesssim \gamma_{-i \widehat{\beta}_n \tr} \left[ 
 \underbrace{|{\beta}_n|^{-1} \|w_{1,n,j'}\|_{L^2} + \|w_{2,n,j'}\|_{L^2}}_{=:w_{j'}}  + |{\beta}_n|^{-1} \|p_{n,j'+1}\|_{1,|{\beta}_n|} 
+ \sum_{k=j'+1}^{J_n} \frac{|D_{k,j',n}| }{{|{\beta}_n|}^{1/2}} \|p_{n,k}\|_{1,|{\beta}_n|} 
\right],
\end{align*}
where we used the standard convention that an empty sum is zero and our setting $p_{n,J_n+1} = 0$. 
In view of \eqref{eq:estimate-D} and $|{\beta}_n| \sim n$, this simplifies to 
\begin{align*}
\|p_{n,j'}\|_{1,|{\beta}_n|} &\leq \widehat{C}  \gamma_{-i {\beta}_n L} |{\beta}_n|^{-1} \left[ 
|{\beta}_n| w_{j'} + \sum_{k=j'+1}^{J_n} \|p_{n,k}\|_{1,|{\beta}_n|}
\right]. 
\end{align*}
We set $a_j:= \|p_{n,J_n+1-j}\|_{1,|{\beta}_n|}$, $b_j:= |{\beta}_n| w_{J_n+1-j}$, 
$G:= \widehat{C} \gamma_{-i {\beta}_n L} |{\beta}_n|^{-1}$ 
and apply \eqref{eq:recurion-for-sum} to get 
\begin{align*}
\sum_{j=1}^m \|p_{n,J_n+1-j}\|_{1,|{\beta}_n|} &= 
\sum_{j=1}^{m} a_j  = S_{m} \leq (1 + G)^{m-1} a_1 + \sum_{j=1}^{m-1} G b_{j+1} (1 + G)^{m-1-j}  \\
&\leq (1 + G)^{m-1} \|p_{n,J_n}\|_{1,|{\beta}_n|} + \widehat{C} \sum_{j=1}^{m-1} (1 + G)^{m-1-j} \gamma_{-i{\beta}_n L} w_{J_n-j} \\
& \leq \widehat{C} \sum_{j=0}^{m-1} (1+G)^{m-1-j} \gamma_{-i{\beta}_n L} w_{J_n-j} 
= \widehat{C} \gamma_{- {\beta}_n L} \sum_{j=1}^{m} (1+G)^{m-j} w_{J_n+1-j}. 
\end{align*}
The proof is complete. 
\end{proof}
Combining the results for the different Jordan blocks, we obtain 
\begin{theorem}
\label{thm:stability-circular-with-Jordan}
Fix  $r_1$, $r_2$, $\omega$, $\rfi$, $\impedance$, $c_0>0$.
Assume $\tr > c_0$. 
Let $\{-{\beta}_n^2\}_\ninn$ with $\operatorname{Im} {\beta}_n > 0$ be the eigenvalues of the operator $\Ar$ of \eqref{eq:Ar} and 
$J_n$ be the length of the corresponding Jordan chain. 
Define for the constant $\gamma_{\kappa}$ of \eqref{eq:gamma_kappa} and the constant $\widehat{C}$ of  Lemma~\ref{lemma:backsubstitution-Jordan-block} 
\begin{equation}
\label{eq:thm:stability-circular-with-Jordan-10}
\gamma_{\operatorname{max}}:= \sup_{n \in \doubleIN} \widehat{C} \gamma_{-i{\beta}_n \tr} 
\left(1 + \frac{\widehat{C} \gamma_{-i{\beta}_n \tr}}{|{\beta}_n|}\right)^{J_n-1}. 
\end{equation}
The solution $u$ of the circular waveguide problem \eqref{eq:Helmholtz-polar} with data 
$f \in (\Hdir)'$ satisfies 
\begin{equation}
\label{eq:thm:stability-circular-with-Jordan-20}
\Vert u \Vert_{H^1(\Omega)}  \leq C \gamma_{\operatorname{max}}  \|f\|_{(\Hdir)'} 
\end{equation}
for a constant $C>0$ that does not depend on $f$, $\tr$, but may depend upon $r_1$, $r_2$, $\omega$, $\rfi$, $\impedance$, and $c_0$. 
The constant $\gamma_{\operatorname{max}}$ satisfies with $J:= \max_{n} J_n$
\begin{equation}
\label{eq:thm:stability-circular-with-Jordan-30}
\gamma_{\operatorname{max}} \leq C (2+\min\{C_\beta, C^\prime_\beta \tr\})^J 
\end{equation}
with $C_\beta$, $C^\prime_\beta$ from Lemma~\ref{lemma:control-Imbeta_n} and a $C > 0$ depending on $\Ar$. 
\end{theorem}
\begin{proof}
The proof of \eqref{eq:thm:stability-circular-with-Jordan-20}
consists of combining the contributions from each Jordan chain. It proceeds 
in exactly the same way as the proof of Theorem~\ref{thm:stability-circular} and is therefore omitted. 
The bound \eqref{eq:thm:stability-circular-with-Jordan-30} follows from Lemma~\ref{lemma:control-Imbeta_n}. 
\end{proof}
\begin{remark}[on the bound \eqref{eq:thm:stability-circular-with-Jordan-30}]
\label{rem:constbound}
The form of the bound \eqref{eq:thm:stability-circular-with-Jordan-30} is discussed in Remark~\ref{rem:control-Imbeta_n}: 
In practice, \cite{Mora-Paz_Demkowicz_Taylor_Grosek_Henneking_25}, the constant $C_\beta$ is large so that 
the expression $2 + C^\prime_\beta \tr $ better describes the stability properties. In typical applications, 
we expect $J = 1$ so that the result of Theorem~\ref{thm:stability-circular-with-Jordan} reverts to the statement of 
Theorem~\ref{thm:stability-circular}. 
\eremk
\end{remark}

\section{Stability analysis including an interior heterogeneous domain}\label{section:interior}

In applications a short, initial part of the fiber is usually connected to external devices, which destroys the homogeneous structure of the waveguide.
Hence in this case our analysis from Section~\ref{section:analysis_with_Jordan_chains} is no longer applicable.
To model this setting we consider an extended domain
\begin{align*}
\Omega&:=(r_1,r_2)\times (\theta_0,\tr),\qquad
\Omega_0:=(r_1,r_2)\times (\theta_0,0),\qquad
\Omega_{\tr}:=(r_1,r_2)\times (0,\tr),\\
\Gamma_\mathrm{imp}&:=\{r_2\}\times (\theta_{0},\tr)\},\qquad
\Gamma_\mathrm{in}:=(r_1,r_2)\times \{\theta_0\},
\end{align*}
with $\theta_0<0$, i.e., we attach $\Omega_0$ to our previous geometry $\Omega_{\tr}$ (see Fig.~\ref{fig:waveguide}).
The material coefficients $\rfi\in L^\infty(\Omega)$, $\impedance\in L^\infty(\Gamma_\mathrm{imp})$ are allowed to depend on $\theta$ in $\Omega_0$, but to remain independent of $\theta$ in $\Omega_{\tr}$.
In particular, we consider the variational form
\begin{align}
\label{eq:ses-het}
\ap(\af,\tf):=\innerprod{ r\partial_r \af,\partial_r \tf }_{\Omega}
+\innerprod{ r^{-1}\partial_\theta \af,\partial_\theta \tf }_{\Omega}
-\omega^2 \innerprod{ r\rfi \af, \tf }_{\Omega}
+i\omega r_2 \innerprod{\impedance \af, \tf }_{\Gamma_{imp}}
-\innerprod{ \DtN \af, \tf }_{\Gamma_{out}},
\end{align}
with $\af,\tf \in \Hdir$
and the problem to find $\af\in\Hdir$ such that
\begin{align}\label{eq:a-varprob}
\ap(\af,\tf)=\innerprod{rf,\tf}_{\Omega} \quad\text{for all }\tf\in\Hdir,
\end{align}
with given $f\in (\Hdir)'$.
Note that this is just an exemplary setting and our forthcoming analysis can be easily adapted to other kinds of interior heterogeneities, e.g., deformations of the interior domain.
Let
\begin{align*}
X_{\tr}:=\{\af\in\Hdir\colon \af|_{\Omega_0}=0\}
\end{align*}
and $P_{X_{\tr}}\in L(\Hdir,X_{\tr})$ be defined by
\begin{align*}
\ap(P_{X_{\tr}}\af-\af,\tf)=0 \quad\text{for all}\quad \tf\in X_{\tr}.
\end{align*}
Note that the former projection is well-defined due to Section~\ref{section:analysis_with_Jordan_chains}.
Then with
\begin{align*}
X_0:=\{(I-P_{X_{\tr}})\af\colon\af\in\Hdir\}
\end{align*}
we have the topological decomposition
\begin{align*}
\Hdir=X_0 \oplus^\mathcal{T} X_{\tr}.
\end{align*}
The allure the this decomposition is that the form $\ap(\cdot,\cdot)$ becomes block triangular:
\begin{align}\label{eq:block}
\begin{aligned}
\ap(\af_0+\af_{\tr},\tf_0+\tf_{\tr})=
\ap(\af_0,\tf_0)+\ap(\af_{\tr},\tf_{\tr})&+\ap(\af_{\tr},\tf_0)\\
&\forall \af_0,\tf_0\in X_0, \af_{\tr},\tf_{\tr}\in X_{\tr}.
\end{aligned}
\end{align}
\begin{theorem}\label{thm:interior}
Fix $r_1,r_2,\omega,c_0>0$ and $\impedance\in L^\infty(\Gamma_\mathrm{imp})$, $\rfi\in L^\infty(\Omega)$ such that $\impedance|_{\{r_2\}\times (0,\tr)}>0$ is constant and $\rfi|_{\Omega_{\tr}}$ is independent of $\theta$.
Assume that problem \eqref{eq:a-varprob} with $\tr=0$ is well-posed.
Then, there exists a constant $C>0$ such that the solution $\af$ to \eqref{eq:a-varprob} satisfies
\begin{align*}
\|\af\|_{H^1(\Omega)} \leq C\gamma_\mathrm{max} \|f\|_{(\Hdir)'}
\end{align*}
for all $\tr>c_0$, $f\in(\Hdir)'$, where $\gamma_\mathrm{max}$ is as in \eqref{eq:thm:stability-circular-with-Jordan-10}.
\end{theorem}
\begin{proof}
Due to \eqref{eq:block} it suffices to analyze the stability of $\ap(\cdot,\cdot)|_{X_0\times X_0}$ and $\ap(\cdot,\cdot)|_{X_{\tr}\times X_{\tr}}$.
The latter follows directly from Section~\ref{section:analysis_with_Jordan_chains} by identifying a function with its continuation by zero.
For the former we employ the definition of $X_0$ and integration by parts to obtain that
\begin{align*}
\ap(&\af_0,\tf_0)\\
&=\innerprod{ r\partial_r \af_0,\partial_r \tf_0 }_{\Omega}
+\innerprod{ r^{-1}\partial_\theta \af_0,\partial_\theta \tf_0 }_{\Omega}
-\omega^2 \innerprod{ r\rfi \af_0, \tf_0 }_{\Omega}
+i\omega r_2 \innerprod{\impedance \af_0, \tf_0 }_{\Gamma_{imp}}
-\innerprod{ \DtN \af_0, \tf_0 }_{\Gamma_{out}}\\
&=\innerprod{ r\partial_r \af_0,\partial_r \tf_0 }_{\Omega_0}
+\innerprod{ r^{-1}\partial_\theta \af_0,\partial_\theta \tf_0 }_{\Omega_0}
-\omega^2 \innerprod{ r\rfi \af_0, \tf_0 }_{\Omega_0}
+i\omega r_2 \innerprod{\impedance \af_0, \tf_0 }_{(r_1,r_2)\times(\theta_0,0)}\\
&\hspace{120mm}-\innerprod{ \DtN \af_0, \tf_0 }_{(r_1,r_2)\times\{0\}}.
\end{align*}
By assumption the problem to find $\af_0\in X_0$ such that $\ap(\af_0,\tf_0)=\innerprod{r f,\tf_0 }_{\Omega}$ for all $\tf\in X_0$ is well-posed, which implies the stability estimate $\|\af_0\|_{H^1(\Omega_0)} \leq C_0 \|f\|_{(\Hdir)'}$ with the stability constant $C_0>0$.
It remains to note that $\af_0|_{\Omega_{\tr}}$ solves the Helmholtz equation in $\Omega_{\tr}$ with Dirichlet data $\af_0|_{(r_1,r_2)\times\{0\}}$ and otherwise homogeneous data.
Thus $\|\af_0\|_{H^1(\Omega_{\tr})} \leq C\gamma_\mathrm{max} \|\af_0\|_{H^1(\Omega_0)}$.
The final result the follows from the triangle inequality.
\end{proof}

\printbibliography[heading=bibintoc]
\setcounter{section}{0}

\appendix
%
%

\section{A review of Mityagin - Siegl theory with complementary results }\label{appendix:2}
Let $H$ be a separable Hilbert space with inner product $(\cdot,\cdot)$ and corresponding norm $\Vert \cdot \Vert$.  
Let $a(\cdot,\cdot)\colon V\times V\subset H\times H\to\mathbb{R}$ be an unbounded, Hermitian positive-definite sesquilinear form with the domain space
$$
V:=\{ u \in H \colon a(u,u) < \infty \}
$$
being dense in $H$.
Due to Kato's \emph{second representation theorem} \cite[Thm.~{VI.2.23},Chap.~VI]{Kato} the form $a(\cdot,\cdot)$ induces an unbounded, selfadjoint operator
\begin{align*}
A\colon \dom(A)\subset H \to H \, ,
\end{align*}
with the relation
$$
a[u] := a(u,u) = \Vert A^{\half} u \Vert^2, \quad \dom (a) = V = \dom(A^{\half}) 
$$
We will refer to the energy inner product and the corresponding norm as follows:
$$
(u,v)_V := a(u,v) = (Au,v) \, , \qquad \Vert u \Vert_V^2 :=  (u,u)_V = \Vert A^{\half} u \Vert^2.
$$
Let $b(\cdot,\cdot)\colon \dom(B)\times \dom(B)\subset V\times V\to\mathbb{C}$ be an additional  sesquilinear form.
We consider the variational eigenproblem:
\begin{align*}
\text{Find} \quad (\lambda,u) \in \mathbb{C}\times V\setminus\{0\} \quad\text{such that}\quad a(u,v) + b(u,v) = \lambda (u,v) \quad \text{for all }v \in V \, .
\end{align*}
We make the two following critical assumptions.
\begin{assumption}[Eigenvalue separation condition]
\label{ass:ev-sep}
Let $(\mu_k, \psi_k),\, k=1,\ldots$ be the sequence of $H$-normalized eigenpairs of the operator $A$.
There exist constants
$\gamma > 0$ , $\kappa > 0$, $N_0 > 0$ such that
\begin{align}
\label{eq:separation_condition}
\mu_{k+1} - \mu_k \geq \kappa k^{\gamma - 1} \qquad \forall \, k \geq N_0 \, .
\end{align}
\end{assumption}
Note that Assumption~\ref{ass:ev-sep} implies that, asymptotically, the eigenvalues $\mu_k$ are simple.
\begin{assumption}[Local form-subordination condition]
\label{ass:local-subord}
There exist a constant $\alpha$ with $2 \alpha + \gamma > 1$ and a constant $M_b>0$ such that:
\begin{align}
\label{eq:subordination_condition}
\vert b(\psi_m, \psi_n) \vert \leq \frac{M_b}{m^\alpha \, n^\alpha} \qquad  \forall \, m,n \in \doubleIN \, .
\end{align}
\end{assumption}
\mhend%
%
%
\cite[Lem.~{4.2}]{Mityagin_Siegl_19} shows that the form $b(\cdot,\cdot)$ is $p$-subordinate to the form $a(\cdot,\cdot)$ with some $p = p(\alpha,\gamma) \in [0,1)$, i.e., 
\be
\vert b[u] \vert \leq C a[u]^p  \, \Vert u \Vert^{2(1-p)} .
\label{eq:b-subordinated-to-a}
\ee
Hence, by (generalized) Young's inequality, $b(\cdot,\cdot)$ is {\em relatively bounded with respect to $a(\cdot,\cdot)$ with the bound $0$} (see \cite[Chap.~{VI.1.6}]{Kato}), i.e.,
$$
\forall \, \epsilon>0 \quad \exists\, C_\epsilon \geq 0 \colon \quad \vert b[u] \vert \leq \epsilon a[u] + C_\epsilon \Vert u \Vert^2 \, .
$$
This implies that the form
$
t(\cdot,\cdot) := a(\cdot,\cdot) + b(\cdot,\cdot)
$
is sectorial and closed and it determines uniquely the corresponding $m$-sectorial operator $T$ with compact resolvent, see \cite[Thm.~{VI.3.4}]{Kato}. 
 
 \paragraph*{An Alternative Characterization of the Operator \boldmath{$T$}.}
 We introduce the operators
 $$
 K(z) u := \sum_{k\in \doubleIN} (z - \mu_k)^{-\half} (u, \psi_k) \psi_k \qquad (  = (- R_z(A)) ^\half = i R_z(A)^\half)  
 $$
 where $R_z(A)$ denotes the resolvent of $A$ at $z$.
 Above, the $s$ power of $w \in \doubleIC$ is defined as:
 $$
 w^s := \vert w \vert^s e^{ i s \arg w} \quad \text{where } \arg w \in (-\pi,\pi] \, .
 $$
 Notice that
 $$
 K(z)^2 = (z - A)^{-1},\quad z \in \rho(A) \, .
 $$
 Let $B(z)$, $z \in \rho(A)$, be the operator uniquely determined by the relation:
 $$
  (B(z) u,v ) =  b(K(z) u, K^\ast(z) v) \quad u,v \in H \, .
 $$
 We have:
 \be
 z - T = K(z)^{-1} (I - B(z)) K(z)^{-1} , \quad z \in \rho(A) \, .
 \label{eq:factorization}
 \ee
 This gives:
 \be
 (z - T)^{-1} = K(z) (I - B(z))^{-1} K(z), \quad z \in \rho(A)\, ,
 \label{eq:resolvent_factorization}
 \ee
 provided $I - B(z)$ is invertible.

 \paragraph*{Representation formulas for the eigenvectors.}
 
 Let $\{(\lambda_n, \phi_n)\}_\ninn$ denote the eigenpairs of operator $T$.
 Mityagin and Siegl show that, for sufficiently large $n$, the eigenvalues $\lambda_n$ are simple (of algebraic multiplicity equal one) and the corresponding eigenvectors $\phi_n$ are obtained
 by the Riesz projections:
 \be
 \begin{array}{lll}
 \phi_n 
 & \ds = \frac{1}{2 \pi i} \int_{\Gamma_n} (z - T)^{-1} \psi_n \, dz \\[12pt]
 & \ds = \psi_n + \frac{1}{2 \pi i} \int_{\Gamma_n } ((z - T)^{-1} - (z - A)^{-1} ) \psi_n \, dz \, .
 \end{array}
 \label{eq:3.13}
 \ee
 where $\Gamma_n = \ptl \Pi_n$ is a counterclockwise contour around the eigenvalue $\mu_n$ with 
 $$
 \Pi_n  := \{ z \in \doubleIC \, : \, \mu_n - \frac{\kappa}{2}(n-1)^{\gamma-1} < \Re z < \mu_n + \frac{\kappa}{2} n^{\gamma-1} , \quad \vert \Im z \vert < \frac{\kappa}{2} n^{\gamma-1} \} \, .
 $$
 Mityagin and Siegl also show that $\Vert B(z) \Vert < 1$ on $\Gamma_n$.
 Hence,
 $$
 \begin{array}{lll}
 \phi_n  - \psi_n 
 & \ds = \frac{1}{2 \pi i} \int_{\Gamma_n } ((z - T)^{-1} - (z - A)^{-1} ) \, \psi_n \, dz \\[12pt]
 &\ds = \frac{1}{2 \pi i} \int_{\Gamma_n } K(z) \left( \sum_{m=1}^\infty B^m(z) \right) K(z) \, \psi_n\, dz  \, .
 \end{array}
 $$
 To express the perturbation $\phi_n  - \psi_n$ we define the vectors
 $$
 \phi_n^{(k)} := \frac{1}{2 \pi i} \int_{\Gamma_n } K(z) B^k(z) K(z) \, \psi_n \, dz  \, .
 $$
\subsection{Auxiliary estimates}
 
 Let $n \in \doubleIN, \gamma >0$, and $\omega + \gamma > 1$.  Define:
 $$
 \sigma_{\omega,\gamma}(n) := 
 \left\{
 \begin{array}{ll}
 n^{- \omega - \gamma+1} \ln (en) \quad & \omega \leq 1 \, , \\[5pt]
 n^{- \gamma} & \omega >1 \, .
 \end{array}
 \right.
 $$
 We start with a simple consequence of \cite[Lem.~{4.1} and {4.3}]{Mityagin_Siegl_19}.
\begin{lemma}
\label{lemma:4.5}%
Let $\omega + \gamma >1$.
The following estimate holds:
$$
\sum_{m=1,m\ne n}^\infty \frac{1}{m^{\omega}} \frac{1}{ \vert \mu_m - \mu_n \vert} = O(\sigma_{\omega, \gamma}(n)) \, .
$$
\end{lemma}
\begin{proof}
We reproduce the proof of \cite[Lem.~{4.4}]{Mityagin_Siegl_19}.
We begin by splitting the sum into two parts:
$$
\sum_{m=1,m\ne n}^\infty \frac{1}{m^\omega \vert \mu_m - \mu_n \vert} = \left( \sum_{m=1,m\ne n}^{N_0}   +  \sum_{m=N_0+1,m\ne n}^{\infty }  \right)  \frac{1}{m^\omega \vert \mu_m - \mu_n \vert} \, .
$$
We can assume $n > N_0+1$.
The first part is estimated as follows.
$$
\begin{array}{lll}
\ds \sum_{m=1,m\ne n}^{N_0}  \frac{1}{m^\omega \vert \mu_m - \mu_n \vert}  
&\ds \leq \max \{ 1, N_0^{-\omega} \} \sum_{m=1,m\ne n}^{N_0}  \frac{1}{\mu_n - \mu_m}  \quad  \left(  
\begin{array}{ll} m^\omega > 1 & \omega > 0 \\[6pt] m^\omega \geq N_0^\omega & \omega < 0 
\end{array}
\right) \, ,\\[18pt]
& \ds \leq \frac{\max \{ N_0, N_0^{1-\omega} \} }{ \mu_n - \mu_{N_0} } \\[12pt]
& \ds \leq \frac{\gamma \max \{ N_0, N_0^{1-\omega} \} }{\kappa} 
\left\{ 
\begin{array}{lll}
\ds \frac{1}{(n-1)^\gamma - (N_0-1)^\gamma} & \gamma \geq 1 \\[6pt]
\ds \frac{1}{n^\gamma - N_0^\gamma} & 0 < \gamma < 1 
\end{array}
\right. \quad  (\text{\cite[Lem.~{4.1}]{Mityagin_Siegl_19}}) 
 \, , \\[15pt]
& = O(n^{- \gamma}) \, .
\end{array}
$$
Note that $O(n^{- \gamma})$ is slightly better than $O(\sigma_{\omega,\gamma}(n))$ for $\omega < 1$.
The estimate of the second part follows from \cite[Lem.~{4.1}, {4.3}]{Mityagin_Siegl_19}.
We have for  $\gamma \geq 1$:
$$
\begin{array}{llll}
\ds \sum_{m=N_0+1,m \ne n}^\infty \frac{1}{m^\omega \vert \mu_m - \mu_n \vert} 
& \ds \leq \frac{\gamma}{\kappa} \sum_{m=N_0+1,m \ne n}^\infty \frac{1}{m^\omega \vert (m-1)^\gamma  - (n-1)^\gamma  \vert} \\[12pt]
&\ds = \frac{\gamma}{\kappa} \sum_{m=N_0+1,m \ne n}^\infty \underbrace{\left( \frac{m-1}{m} \right)^\omega}_{ \leq C}  \, \frac{1}{(m-1)^\omega \vert (m-1)^\gamma  - (n-1)^\gamma  \vert} \\[12pt]
& \ds = O(\sigma_{\omega,\gamma}(n)) \, .
\end{array}
$$
Similarly, for  $0 < \gamma < 1$:
$$
\sum_{m=N_0+1,m \ne n}^\infty \frac{1}{m^\omega \vert \mu_m - \mu_n \vert} 
\leq \frac{\gamma}{\kappa} \sum_{m=N_0+1,m \ne n}^\infty \frac{1}{m^\omega \vert m^\gamma  - n^\gamma  \vert} 
= O(\sigma_{\omega,\gamma}(n)) \, .
$$
\end{proof}
The following (new) lemma is an analogue of \cite[Lem.~{4.3}]{Mityagin_Siegl_19}.
\begin{lemma}
\label{lemma:4.3a}%
Let $n \in \doubleIN, \gamma >0$, and
$\omega+2\gamma>1$. Then
\be
\sum_{k=1,k\neq n}^\infty \frac{1}{k^\omega \vert k^\gamma - n^\gamma \vert^2} = O(\tau_{\omega,\gamma}(n))\, ,
\label{eq:lemma:4.3_1}
\ee
where
\be
\tau_{\omega,\gamma}(n) := 
\left\{
\begin{array}{ll}
n^{- (\omega + 2\gamma-2)} \quad & \omega \leq  2 \, , \\[5pt]
n^{- 2\gamma} & \omega >2 \, .
\end{array}
\right.
\label{eq:lemma:4.3.7}
\ee
\end{lemma}
\begin{proof}
Below, $C$ denotes a generic constant that may change from line to line.  We also assume that $n$ is sufficiently large, as necessary.
We start by splitting the sum into three terms:
\begin{align}\label{eq:lemvi-aux1}
\sum_{k=1,k\ne n}^\infty \frac{1}{k^\omega \vert k^\gamma - n^\gamma \vert^2} 
= \sum_{k=1}^{n-1} \frac{1}{k^\omega \vert k^\gamma - n^\gamma \vert^2} 
+ 
\frac{1}{(n+1)^\omega \vert (n+1)^\gamma - n^\gamma \vert^2}
+ \sum_{k=n+2}^\infty \frac{1}{k^\omega \vert k^\gamma - n^\gamma \vert^2} \, .
\end{align}
We note that by the mean-value theorem,
$$
(n+1)^\gamma - n^\gamma = \gamma \, \xi^{\gamma-1} \quad \text{where } \xi \in [n,n+1]\, .
$$
Now,
$$
\begin{array}{lll}
\gamma -1 \geq 0 \Rightarrow \xi^{\gamma-1} \text{ is increasing} \Rightarrow \xi^{\gamma-1} \geq n^{\gamma-1} &\ds  \Rightarrow  \frac{1}{((n+1)^\gamma - n^\gamma)^2} \leq \frac{1}{\gamma \, n^{2(\gamma-1)}}\\[12pt]
\gamma -1 <  0 \Rightarrow \xi^{\gamma-1} \text{ is decreasing} \Rightarrow \xi^{\gamma-1} \geq (n+1)^{\gamma-1} 
&\ds  \Rightarrow  \frac{1}{((n+1)^\gamma - n^\gamma)^2} \leq \frac{1}{\gamma \, (n+1)^{2(\gamma-1)}} \, .
\end{array}
$$
Hence, in both cases the middle term in \eqref{eq:lemvi-aux1} is of order $O(n^{- (\omega +2(\gamma -1))})$.  
Next we estimate the first term in \eqref{eq:lemvi-aux1}.
Note that for $k \leq [\frac{n}{2}]$ we have $n^\gamma-k^\gamma \geq n^\gamma(1-2^{-\gamma})$.
For $[\frac{n}{2}]+1 \leq k \le n-1$,
$$
\begin{array}{lll}
\omega \geq 0 \Rightarrow k^\omega \geq ([\frac{n}{2}]+1)^\omega \\[6pt]
\omega < 0 \Rightarrow k^\omega \geq (n-1)^\omega \, .
\end{array}
$$
Thus
$$
\sum_{k=1}^{n-1} \frac{1}{k^\omega (n^\gamma - k^\gamma)^2} = \left( \sum_{k=1}^{[\frac{n}{2}]} + \sum_{k = [\frac{n}{2}]+1}^{n-1} \right) \frac{1}{k^\omega (n^\gamma - k^\gamma)^2} 
\leq C \left( \frac{1}{n^{2\gamma}} \sum_{k=1}^{n} \frac{1}{k^\omega} + \frac{1}{n^\omega} \sum_{k=[\frac{n}{2}] }^{n-1} \frac{1}{(n^\gamma - k^\gamma)^2} \right),
$$
(with some multiplicative constant $C$ independent of $n$).
We proceed with treating the third term in \eqref{eq:lemvi-aux1}, for which we note that for $k > 2n$ we have 
$|n^\gamma-k^\gamma| = k^\gamma - n^\gamma > n^\gamma (2^\gamma-1)$.  
Further, for $n+2\leq k \leq 2n$ we have
$$
\begin{array}{lll}
\omega \geq 0 \Rightarrow k^\omega \geq (n+2)^\omega \\[6pt]
\omega < 0 \Rightarrow k^\omega \geq (2n)^\omega \, .
\end{array}
$$
Thus
$$
\sum_{k=n+2}^\infty \frac{1}{k^\omega (k^\gamma - n^\gamma)^2} = 
\left( \sum_{k=n+2}^{2n} + \sum_{k=2n+1}^\infty \right) \frac{1}{k^\omega (k^\gamma - n^\gamma)^2} 
\leq C \left( \frac{1}{n^\omega} \sum_{k=n+2}^{2n} \frac{1}{(k^\gamma - n^\gamma)^2} + \sum_{k=2n+1}^\infty \frac{1}{k^{\omega + 2\gamma}} \right) \, .
$$
In order to estimate the four previously obtained sums we observe that, for $i,j \in \doubleIN, i<j$,
$$
\sum_{k=i}^j  f(k) \leq 
\left\{
\begin{array}{ll}
f(j)  + \int_i^j f(x) \, dx \quad & \text{if $f$ is increasing} \\[5pt]
f(i) + \int_i^j f(x) \, dx & \text{if $f$ is decreasing}\, .
\end{array}
\right.
$$
Hence, we have the bounds:
\be
\sum_{k=1}^n  \frac{1}{k^\omega} 
\leq 1 + \int_1^n  x^{- \omega} \, dx \leq C
\left\{
\begin{array}{llll}
n^{1 - \omega} \quad &  \omega < 1 \, , \\[5pt]
\ln n & \omega = 1 \, , \\[5pt]
1  &  \omega > 1 \, .
\end{array}
\right.
\label{eq:lemma4.3bound1}
\ee
$$
\sum_{k=2n+1}^\infty \frac{1}{k^{\omega + 2\gamma}} \leq \frac{1}{(2 n+1)^{\omega + 2\gamma}} + \int_{2n+1}^\infty x^{-(\omega + 2\gamma)} \, dx 
\leq \frac{C}{n^{\omega + 2\gamma -1}} \qquad C = C(\omega,\gamma) \, .
$$
$$
 \sum_{k = [\frac{n}{2}]} ^{n-1} \frac{1}{(n^\gamma - k^\gamma)^2} 
 \leq  \frac{1}{(n^\gamma - (n-1)^\gamma)^2} + \int_{\frac{n}{2}-1}^{n-1} \frac{dx}{(n^\gamma - x ^\gamma)^2} 
\leq \frac{C}{n^{2(\gamma-1)}} + \frac{C}{n^{2 \gamma-1}}   \int_{\half - \frac{1}{n}}^{1 - \frac{1}{n}} \frac{dy}{(1 - y^\gamma)^2}  \, .
$$
Now, using the d'Hospital rule, we easily check that $\lim_{y \to 1}  \frac{1 - y}{1 - y^\gamma} = \frac{1}{\gamma}$. Hence, by the Weierstrass argument, 
the function $\frac{1-y}{1 - y^\gamma}$ is bounded in the interval $[0,1]$ and, consequently,
$$
\int_{\half - \frac{1}{n}}^{1 - \frac{1}{n}} \frac{1}{(1 - y)^2} \, \underbrace{\left( \frac{1-y}{1-y^\gamma}\right)^2  }_{\leq C(\gamma)}  \, dx
\leq C \int_{\half - \frac{1}{n}}^{1 - \frac{1}{n}} \frac{1}{(1 - y)^2} \, dy = C \frac{n^2}{n+2} = O(n) \, .
$$
Therefore, the final estimate is:
$$
\sum_{k = [\frac{n}{2}]} ^{n-1} \frac{1}{n^\gamma - k^\gamma}  \leq C \frac{1}{n^{2(\gamma-1)}} \, .
$$
Similarly,
$$
\sum_{k=n+2}^{2n} \frac{1}{(k^\gamma - n^\gamma)^2} \leq \frac{C}{n^{2(\gamma-1)}} + \frac{C}{n^{2 \gamma -1}} \int_{1 + \frac{2}{n}}^2 \frac{dy}{(y^\gamma -1)^2} \leq C \frac{1}{n^{2(\gamma-1)}} \, .
$$
The order $\alpha$ of the contributing terms of size $O(n^{-\alpha})$ is therefore as follows:
$$
\begin{array}{llll}
\omega + 2 \gamma - 2 & & & (\text{middle term}) \\[6pt]
\left\{
\begin{array}{ll}
\omega + 2 \gamma -1  & \omega < 1 \\
2 \gamma - \epsilon & \omega = 1 \\
2 \gamma & \omega > 1 
\end{array}
\right.
& \text{and} &\omega + 2 \gamma -2  & (\text{left two terms}) \\[18pt]
\omega + 2 \gamma - 2 &  \text{and} & \omega + 2 \gamma - 1 & (\text{right two terms}) \, .
\end{array}
$$
Combining all estimates, we get the final result.
\end{proof}
The following result is an analogue of Lemma~\ref{lemma:4.5} and consequence of Lemma~\ref{lemma:4.3a}.
\begin{lemma}
\label{lemma:4.5a}%
Let $\omega + 2\gamma >1$.
The following estimate holds:
$$
\sum_{m=1,m\ne n}^\infty \frac{1}{m^{\omega}} \frac{1}{ \vert \mu_m - \mu_n \vert^2} = O(\tau_{\omega, \gamma}(n)) \, .
$$
\end{lemma}
\begin{proof}
Using Lemma~\ref{lemma:4.3a} the proof is a repetition of the proof of Lemma~\ref{lemma:4.5}.
\end{proof}

\subsection{Estimation of the eigenvector perturbations }
We now estimate the contributions $\phi_n^{(k)}$ to the eigenvectors $\phi_n$.

\begin{lemma}
\label{lemma:L2phi_est}
Let $2 \alpha + 2\gamma >1$. Then,
$$
\Vert \phi_n^{(1)} \Vert = O( n^{- \alpha} \tau^\half_{2 \alpha,\gamma}(n) ) \, .
$$
Additionally, if $2 \alpha + \gamma >1$, then, 
$$
\Vert  \phi_n^{(k)}   \Vert  = O(n^{- \alpha} \sigma_{2 \alpha,\gamma }^{k-1}(n)\, \tau_{2\alpha,\gamma}^{\half}(n)),\qquad k=2,3,\ldots \, .
$$
\end{lemma}
\begin{proof}
We have
$$
\begin{array}{llll}
(\phi_n^{(1)},\psi_m) 
&\ds = \frac{1}{2 \pi i} \int_{\Gamma_n} (B(z) K(z) \psi_n, K^\ast(z) \psi_m) \, dz \\[12pt]
&\ds = \frac{1}{2 \pi i} \int_{\Gamma_n} b(K^2(z) \psi_n, K^{\ast 2}(z) \psi_m) \, dz \\[12pt]
&\ds = \frac{1}{2 \pi i} \int_{\Gamma_n} \frac{b(\psi_n,\psi_m)}{(z - \mu_n) \, (z - \mu_m) } \, dz \\[15pt]
&\ds = \left\{ \begin{array}{ll} 
\frac{b(\psi_n,\psi_m)}{\mu_n - \mu_m} \quad & m \ne n \\[5pt]
0 &  m=n 
\end{array}
\right.  \, .
\end{array}
$$
Consequently, by Lemma~\ref{lemma:4.5a},
$$
\begin{array}{lll}
\Vert \phi_n^{(1)} \Vert^2
&\ds = \sum_{m=1, m\neq n}^\infty \frac{\vert b(\psi_n,\psi_m)\vert^2}{ \vert \mu_m - \mu_n \vert^2} \\[18pt]
&\ds  \leq \frac{M_b^2}{n^{2 \alpha}} \sum_{m=1, m\neq n}^\infty \frac{1}{m^{2 \alpha} \vert \mu_m - \mu_n \vert^2} \\[18pt]
& \ds = O(n^{-2\alpha} \tau_{2\alpha,\gamma}(n)) \, .
 \end{array}
$$
To treat $k>1$ we note that  
$$
\begin{array}{lll}
B(z) \psi_{j_0} &
\ds = \sum_{j_1} ( B(z) \psi_{j_0},\psi_{j_1} ) \psi_{j_1} \\[12pt]
B^2(z) \psi_{j_0}
&\ds = \sum_{j_1} ( B(z) \psi_{j_0},\psi_{j_1} ) B(z)  \psi_{j_1} \\[12pt]
&\ds = \sum_{j_1} \sum_{j_2}  ( B(z) \psi_{j_0},\psi_{j_1} )\,   ( B(z) \psi_{j_1},\psi_{j_2} ) \psi_{j_2} \\[12pt]
&\ds = \sum_{j_1,j_2} \left( \prod_{l=1}^2  ( B(z) \psi_{j_{l-1}},\psi_{j_l} ) \right) \psi_{j_2} \\[12pt]
\vdots \\[12pt]
B^k(z) \psi_{j_0}  &\ds 
= \sum_{j_1,j_2,\ldots,j_k}   \big( \prod_{l=1}^k ( B(z) \psi_{j_{l-1}},\psi_{j_l} ) \big)  \psi_{j_k} \\[12pt]
\ & \ds = \sum_{j_1,j_2,\ldots,j_k} \big(   \prod_{l=1}^k  \frac{  b(\psi_{j_{l-1}},\psi_{j_l})}{(z - \mu_{j_{l-1}})^{1/2} ( z - \mu_{j_l})^{ 1/2}} \big)  \psi_{j_k} \\[12pt]
& \ds = \sum_{j_1,j_2,\ldots,j_k} \left(   \frac{1}{(z- \mu_{j_0})^{1/2}}  \frac{1}{(z- \mu_{j_k})^{1/2}}   \prod_{l=1}^{k-1}   \frac{1}{z- \mu_{j_l}}  \prod_{l=1}^k  b(\psi_{j_{l-1}},\psi_{j_l}) \right)  \psi_{j_k}  \, .
\end{array}
$$
And, so
$$
K(z) B^k(z) K(z) \psi_{j_0} = \sum_{j_1,j_2,\ldots,j_k} \left(      \prod_{l=0}^{k}   \frac{1}{z- \mu_{j_l}}  \prod_{l=1}^k  b(\psi_{j_{l-1}},\psi_{j_l}) \right)  \psi_{j_k}  \, .
$$
Consequently,
\begin{align}
\label{eq:phink}
\begin{array}{lll}
\phi_n^{(k)} 
& \ds = \frac{1}{2 \pi i} \int_{\Gamma_n} K(z) B^k(z) K(z) \psi_n \, dz \\[12pt]
& \ds = \frac{1}{2 \pi i} \int_{\Gamma_n} \sum_{j_1,\ldots,j_k} \frac{b(\psi_n,\psi_{j_1}) \prod_{l=2}^k b(\psi_{j_{l-1}},\psi_{j_l})}{(z - \mu_n) \prod_{l=1}^k (z - \mu_{j_l})} \, \psi_{j_k} \, dz \\[18pt]
&\ds = \sum_{j_1 \neq n} \frac{b(\psi_n,\psi_{j_1})}{\mu_n - \mu_{j_1}} \, 
\sum_{j_2 \neq n} \frac{b(\psi_{j_1},\psi_{j_2})}{\mu_n - \mu_{j_2}} \, \ldots\, 
\sum_{j_k \neq n} \frac{b(\psi_{j_{k-1}},\psi_{j_k})}{\mu_n - \mu_{j_k}} \, \psi_{j_k}  & (\text{Residue Theorem})\, .
\end{array}
\end{align}
Hence
$$
\begin{array}{lll}
\Vert  \phi_n^{(k)}   \Vert
&\ds \leq 
\sum_{j_1 \neq n} \frac{\vert b(\psi_n,\psi_{j_1})\vert }{\vert \mu_n - \mu_{j_1}\vert } \, 
\sum_{j_2 \neq n} \frac{\vert b(\psi_{j_1},\psi_{j_2})\vert }{\vert \mu_n - \mu_{j_2}\vert } \, \ldots\, 
\Vert \sum_{j_k \neq n} \frac{b(\psi_{j_{k-1}},\psi_{j_k})}{\mu_n - \mu_{j_k}}  \psi_{j_k} \Vert \\[18pt]
&\ds \leq 
\sum_{j_1 \neq n} \frac{\vert b(\psi_n,\psi_{j_1})\vert }{\vert \mu_n - \mu_{j_1}\vert } \, 
\sum_{j_2 \neq n} \frac{\vert b(\psi_{j_1},\psi_{j_2})\vert }{\vert \mu_n - \mu_{j_2}\vert } \, \ldots\, 
\left( \sum_{j_k \neq n} \frac{\vert b(\psi_{j_{k-1}},\psi_{j_k})\vert^2 }{\vert \mu_n - \mu_{j_k}\vert^2 }   \right)^\half \, .
\end{array}
$$
Now,
$$
\begin{array}{lll}
\ds \sum_{j_k \neq n}  \frac{\vert b(\psi_{j_{k-1}},\psi_{j_k})\vert^2 }{\vert \mu_n - \mu_{j_k}\vert^2 } 
&\ds \leq  \frac{M_b}{j_{k-1}^{2\alpha}} \sum_{j_k \neq n}   \frac{1}{j_k^{2 \alpha}} \frac{1}{\vert \mu_n - \mu_{j_k} \vert^2}\quad  \\[18pt]
&\ds \leq   C \frac{1}{j_{k-1}^{2\alpha}} \tau_{2 \alpha,\gamma}(n) & (\text{Lemma~\ref{lemma:4.5a}}) \, ,
\end{array}
$$
with $C$ independent of $j_{k-1}$ and $n$.
We proceed now with the sum for $j_{k-1}$ including the factor $\frac{1}{j_{k-1}^{\alpha}} $
$$
\begin{array}{llll}
\ds \sum_{j_{k-1}\ne n} \frac{\vert b(\psi_{j_{k-2}},\psi_{j_{k-1}})\vert}{\vert \mu_n - \mu_{j_{k-1}} \vert}  \frac{1}{j_{k-1}^{\alpha}} 
&\ds \leq \frac{M_b}{j_{k-2}^{\alpha}}  \sum_{j_{k-1}\ne n} \frac{1}{j_{k-1}^{2\alpha}} \frac{1}{\vert \mu_n - \mu_{j_{k-1}}\vert} \\[18pt]
&\ds \leq C \frac{1}{j_{k-2}^{\alpha}} \sigma_{2 \alpha,\gamma}(n) & (\text{Lemma~\ref{lemma:4.5}}) \, ,
\end{array}
$$
with $C$ independent of $j_{k-2}$ and $n$.
Proceeding by induction with the remaining sums, we get:
$$
\Vert  \phi_n^{(k)}   \Vert  = O(n^{- \alpha} \sigma_{2 \alpha,\gamma }^{k-1}(n)\, \tau_{2\alpha,\gamma}^{\half}(n))  \, .
$$
\end{proof}
\begin{lemma}
\label{lemma:H1phi_est}%
Let $2 \alpha + \gamma >1$. Then,
$$
\Vert \mu_n^{-\half} A^\half \phi_n^{(1)} \Vert^2 =  O(n^{-(2 \alpha + \gamma)} \, \sigma_{2 \alpha,\gamma}(n) ) +  O(n^{- 2 \alpha} \tau_{2 \alpha,\gamma} (n) ) \, ,
$$
and
$$
\Vert \mu_n^{-\half} A^{\half} \phi_n^{(k)}   \Vert^2  = O \left(\sigma_{2 \alpha,\gamma }^{2(k-1)}(n) (  n^{-(2 \alpha + \gamma)} \sigma _{2\alpha,\gamma}(n) + n^{- 2\alpha} \tau_{2 \alpha,\gamma}(n) )  \right),
\quad k=2,3,\ldots \, .
$$
\end{lemma}
\begin{proof}
$$
\begin{array}{lll}
\Vert \mu_n^{-\half} A^\half \phi_n^{(1)} \Vert^2
&\ds \leq \sum_{m=1,m\ne n}^\infty \frac{\mu_m}{\mu_n} \, \frac{M_B^2}{m^{2 \alpha} n^{2 \alpha}} \, \frac{1}{\vert \mu_m - \mu_n \vert^2} \\[12pt]
&\ds \leq \sum_{m=1,m\ne n}^\infty \left( \frac{\vert \mu_m- \mu_n \vert}{\mu_n} + 1 \right)\, \frac{M_B^2}{m^{2 \alpha} n^{2 \alpha}} \, \frac{1}{\vert \mu_m - \mu_n \vert^2} \\[12pt]
& \ds \leq \frac{M_B^2}{n^{2 \alpha} \, \mu_n} \sum_{m=1,m\ne n}^\infty \, \frac{1}{m^{2 \alpha} } \, \frac{1}{\vert \mu_m - \mu_n \vert} 
+ \frac{M_B^2}{n^{2\alpha} }\sum_{m=1,m\ne n}^\infty \frac{1}{m^{2 \alpha} } \, \frac{1}{\vert \mu_m - \mu_n \vert^2} \\[18pt]
& = O(n^{-(2 \alpha + \gamma)} \, \sigma_{2 \alpha,\gamma}(n) ) +  O(n^{- 2 \alpha} \tau_{2 \alpha,\gamma} (n) )  \, .
\end{array}
$$
Recall the formula for $\phi_n^{(k)}$:
$$
\begin{array}{ll}
\phi_n^{(k)} 
& \ds = \frac{1}{2 \pi i} \int_{\Gamma_n} \sum_{j_1,\ldots,j_k} \frac{b(\psi_n,\psi_{j_1}) \prod_{l=2}^k b(\psi_{j_{l-1}},\psi_{j_l})}{(z - \mu_n) \prod_{l=1}^k (z - \mu_{j_l})} \, \psi_{j_k} \, dz \\[18pt]
&\ds = \sum_{j_1 \neq n} \frac{b(\psi_n,\psi_{j_1})}{\mu_n - \mu_{j_1}} \, 
\sum_{j_2 \neq n} \frac{b(\psi_{j_1},\psi_{j_2})}{\mu_n - \mu_{j_2}} \, \ldots\, 
\sum_{j_k \neq n} \frac{b(\psi_{j_{k-1}},\psi_{j_k})}{\mu_n - \mu_{j_k}} \, \psi_{j_k} \, .
\end{array}
$$
This implies
$$
\mu_n^{-\half} A^{\half} \phi_n^{(k)}  = 
\sum_{j_1 \neq n} \frac{b(\psi_n,\psi_{j_1})}{\mu_n - \mu_{j_1}} \, 
\sum_{j_2 \neq n} \frac{b(\psi_{j_1},\psi_{j_2})}{\mu_n - \mu_{j_2}} \, \ldots\, 
\sum_{j_k \neq n} \frac{b(\psi_{j_{k-1}},\psi_{j_k})}{\mu_n - \mu_{j_k}}  \frac{\mu_{j_k}^\half }{\mu_n^\half}\psi_{j_k} \, ,
$$
and, in turn,
$$
\begin{array}{lll}
\Vert \mu_n^{-\half} A^{\half} \phi_n^{(k)}   \Vert
&\ds \leq 
\sum_{j_1 \neq n} \frac{\vert b(\psi_n,\psi_{j_1})\vert }{\vert \mu_n - \mu_{j_1}\vert } \, 
\sum_{j_2 \neq n} \frac{\vert b(\psi_{j_1},\psi_{j_2})\vert }{\vert \mu_n - \mu_{j_2}\vert } \, \ldots\, 
\Vert \sum_{j_k \neq n} \frac{b(\psi_{j_{k-1}},\psi_{j_k})}{\mu_n - \mu_{j_k}}  \frac{\mu_{j_k}^\half }{\mu_n^\half}\psi_{j_k} \Vert \\[18pt]
&\ds \leq 
\sum_{j_1 \neq n} \frac{\vert b(\psi_n,\psi_{j_1})\vert }{\vert \mu_n - \mu_{j_1}\vert } \, 
\sum_{j_2 \neq n} \frac{\vert b(\psi_{j_1},\psi_{j_2})\vert }{\vert \mu_n - \mu_{j_2}\vert } \, \ldots\, 
\left( \sum_{j_k \neq n} \frac{\vert b(\psi_{j_{k-1}},\psi_{j_k})\vert^2 }{\vert \mu_n - \mu_{j_k}\vert^2 }   \frac{\mu_{j_k}}{\mu_n}\right)^\half \, .
\end{array}
$$
Now,
$$
\begin{array}{lll}
\ds \sum_{j_k \neq n} \frac{\mu_{j_k}}{\mu_n} \frac{\vert b(\psi_{j_{k-1}},\psi_{j_k})\vert^2 }{\vert \mu_n - \mu_{j_k}\vert^2 } 
& \ds = \sum_{j_k \neq n} \left(\frac{\mu_{j_k} - \mu_n }{\mu_n} + 1\right)\, \frac{\vert b(\psi_{j_{k-1}},\psi_{j_k})\vert^2 }{\vert \mu_n - \mu_{j_k}\vert^2 }  \\[18pt]
&\ds \leq  \frac{M_B^2}{j_{k-1}^{2\alpha}}  \left(   \frac{1}{\mu_n } \sum_{j_k \neq n}    \frac{1}{j_k^{2 \alpha}} \frac{1}{\vert \mu_n - \mu_{j_k} \vert}
+  \sum_{j_k \neq n}    \frac{1}{j_k^{2 \alpha}} \frac{1}{\vert \mu_n - \mu_{j_k} \vert^2} \right)\\[18pt]
&\ds \leq   C \frac{1}{j_{k-1}^{2\alpha}} \left( n^{- \gamma} \sigma_{2 \alpha,\gamma}(n) +  \tau_{2 \alpha,\gamma}(n) \right)  \, ,
\end{array}
$$
with $C$ independent of $j_{k-1}$ and $n$.
We proceed now with the sum for $j_{k-1}$ including the factor $\frac{1}{j_{k-1}^{\alpha}} $
$$
\begin{array}{llll}
\ds \sum_{j_{k-1}\ne n} \frac{\vert b(\psi_{j_{k-2}},\psi_{j_{k-1}})\vert}{\vert \mu_n - \mu_{j_{k-1}} \vert}  \frac{1}{j_{k-1}^{\alpha}} 
&\ds \leq \frac{M_b}{j_{k-2}^{\alpha}}  \sum_{j_{k-1}\ne n} \frac{1}{j_{k-1}^{2\alpha}} \frac{1}{\vert \mu_n - \mu_{j_{k-1}}\vert} \\[18pt]
&\ds \leq C \frac{1}{j_{k-2}^{\alpha}} \sigma_{2 \alpha,\gamma}(n) & (\text{Lemma~\ref{lemma:4.5}}) \, ,
\end{array}
$$
with $C$ independent of $j_{k-2}$ and $n$.
Proceeding by induction with the remaining sums, we get:
$$
\Vert \mu_n^{-\half} A^{\half} \phi_n^{(k)}   \Vert  = O(n^{- \alpha} \sigma_{2 \alpha,\gamma }^{k-1}(n) )(  n^{-\frac{\gamma}{2}} \sigma^\half _{2\alpha,\gamma}(n) + \tau^\half_{2 \alpha,\gamma}(n) ) ) \, .
$$
\end{proof}

\subsection{Basis properties }

We arrive at the main results of interest.
\begin{theorem}
\label{thm:L2Bari_basis}
Under the assumptions:
$$
\left\{
\begin{array}{lll}
2 \alpha + \gamma  > 3/2 & \alpha \leq 1\, , \\
2 \alpha + \gamma > 1 & \alpha > 1\, , 
\end{array}
\right.
$$
the system $\{ \phi_n \}_1^\infty$ is quadratically close to $\{ \psi_n \}_1^\infty$:
$$
\sum_{n=1}^\infty \Vert \psi_n - \phi_n \Vert^2 < \infty \, .
$$
\end{theorem}
\begin{proof}
Indeed, by Lemma~\ref{lemma:L2phi_est},
$$
\begin{array}{lll}
\Vert \psi_n - \phi_n \Vert^2 
& \ds \leq 2 \Vert \phi_n^{(1)} \Vert^2 + 2 \Vert \sum_{k=2}^\infty \phi_n^{(k)} \Vert^2  \\[9pt]
& \ds \leq O(n^{- 2 \alpha} \tau_{2 \alpha,\gamma}(n))  \left( 1   + \left(  \sum_{k=2}^\infty O(\sigma^{k-1}_{2 \alpha,\gamma}(n) ) \right)^2\right) \, .
\end{array}
$$
The series is summable and
$$
\left\{
\begin{array}{ll}
2 \alpha +2 \alpha + 2 (\gamma -1) > 1\qquad \quad & \alpha \leq 1\, ,  \\
2 \alpha + 2 \gamma >1 & \alpha > 1 \, .
\end{array}
\right.
$$
\end{proof}
A similar results holds for the rescaled eigenvectors and the $V$-norm.
\begin{theorem}
\label{thm:H1Bari_basis}
Under the assumptions:
$$
\left\{
\begin{array}{lll}
2 \alpha + \gamma  > \frac{3}{2} & \alpha \leq 1\, , \\
2 \alpha + \gamma  > 1 & \alpha > 1\, ,
\end{array}
\right.
$$
the system $ \{\mu_n^{-\half}\phi_n \}_1^\infty$ is quadratically close to $\{ \mu_n^{-\half} \psi_n \}_1^\infty$ in the $V$-norm:
$$
\sum_{n=1}^\infty \Vert \mu_n^{-\half} A^{\half} \psi_n - \mu_n^{-\half} A^{\half} \phi_n \Vert^2 < \infty \, .
$$
\end{theorem}
\begin{proof}
In both regimes for $\alpha$, assumption $2 \alpha + \gamma > 1$ for applying Lemma~\ref{lemma:H1phi_est}, is satisfied.
Now, by Lemma~\ref{lemma:H1phi_est},
$$
\begin{array}{lll}
\Vert \mu_n^{-\half} A^{\half}\psi_n - \mu_n^{-\half} A^{\half} \phi_n \Vert^2 
& \ds \leq 2 \Vert \mu_n^{-\half} A^{\half}\phi_n^{(1)} \Vert^2 + 2 \Vert \sum_{k=2}^\infty \mu_n^{-\half} A^{\half} \phi_n^{(k)} \Vert^2  \\[9pt]
& \ds \leq \left( O(n^{-(2 \alpha + \gamma)} \, \sigma_{2 \alpha,\gamma}(n) ) +  O(n^{- 2 \alpha} \tau_{2 \alpha,\gamma} (n) ) \right) 
 \left( 1   + \left(  \sum_{k=2}^\infty O(\sigma^{k-1}_{2 \alpha,\gamma}(n) ) \right)^2\right) \, .
\end{array}
$$
In order to show summability, we need to take into account definitions of $\sigma_{2 \alpha,\gamma}$
and $\tau_{2 \alpha,\gamma}$, and consider three regimes:\\
{\em Case 1:} $\alpha \le \half$. We must have:
$$
\left.
\begin{array}{rlr}
2 \alpha + \gamma + 2 \alpha + \gamma -1 > 1  & \Rightarrow \quad 2 \alpha + \gamma > 1  \text{ and},\\[6pt]
2 \alpha + 2 \alpha + 2 \gamma - 2 > 1 &  \Rightarrow \quad 2 \alpha + \gamma > \frac{3}{2} 
\end{array}
\right\}
\Leftarrow \quad 2 \alpha + \gamma > \frac{3}{2} \, .
$$
{\em Case 2:} $\half < \alpha \le 1$. We must have:
$$
\left.
\begin{array}{rlr}
2 \alpha +  \gamma + \gamma > 1 & \Rightarrow \quad 2\alpha + 2\gamma > 1 \text{ and}, \\[6pt]
2 \alpha + 2 \alpha + 2 \gamma - 2 > 1 & \Rightarrow \quad 2 \alpha + \gamma > \frac{3}{2}
\end{array}
\right\}
\Leftarrow \quad 2 \alpha + \gamma > \frac{3}{2} \, .
$$
{\em Case 3:} $\alpha \ge 1$. We must have:
$$
\left.
\begin{array}{rlr}
2 \alpha + \gamma + \gamma> 1 &  \Rightarrow \quad 2\alpha + 2\gamma > 1 \text{ and}, \\[6pt]
2 \alpha + 2 \gamma > 1 &  \Rightarrow\quad   2\alpha +2 \gamma > 1
\end{array}
\right\}
\Leftarrow \quad 2 \alpha + \gamma > 1 \, .
$$
As we can see, the three regimes collapse to just two  reflected in the assumptions of the theorem.
The series is summable and the leading term is of order $O(n^{-(1+\epsilon}))$.
\end{proof}

\begin{lemma}\label{lem:renormed_bari_basis}
It holds that $\displaystyle\lim_{j \to \infty} \frac{\vert \lambda_j \vert^\half}{\mu_j^\half} = \lim_{j \to \infty} \frac{\vert \lambda_j \vert^\half}{\Vert \phi_j \Vert_{V}} = 1$.
If the assumption of Thm.~\ref{thm:H1Bari_basis} is satisfied, then the sequence $\{\phi_j/\|\phi_j\|_{V}\}_{j=1}^\infty$ remains quadratically close to the sequence $\{\psi_j/\mu_j^\half\}_{j=1}^\infty$ in the $V$-norm.
\end{lemma}
\begin{proof}
Testing the eigenproblem with $\phi_j$ yields
$$
\Vert \phi_j \Vert^2_{V} + b(\phi_j,\phi_j) = \lambda_j \|\phi_j\|^2. 
$$
Due to Lem.~\ref{lemma:L2phi_est} and Lem.~\ref{lemma:H1phi_est} we get the asymptotic norm equivalences $\|\phi_j\|/\|\psi_j\|\to1$ and $\|\phi_j\|_V/\|\psi_j\|_V\to1$ as $j\to+\infty$.   
Since the form $b(\cdot,\cdot)$ is subordinated to the form $a(\cdot,\cdot)$ (see \cite[Lemma~{4.2}]{Mityagin_Siegl_19} and \eqref{eq:b-subordinated-to-a} for the definition of subordination) 
and $\|{\psi}_j\|/\|{\psi}_j\|_{V} \rightarrow 0$ for $j \rightarrow \infty$, we arrive at
$$
\left\vert 1 - \frac{\lambda_j \|\phi_j\|^2}{\Vert \phi_j \Vert^2_{V}} \right\vert = \frac{\vert b(\phi_j,\phi_j) \vert}{\Vert \phi_j \Vert_{V}^2} \to 0 \text{ as } j\to+\infty \, .
$$
Additionally, 
$$
\big\Vert \frac{\psi_{j}}{\mu_j^\half} - \frac{\phi_j }{\vert \lambda_j \vert^{ \half} }\big\Vert_{V}
\leq \big\Vert \frac{\psi_j}{\mu_j^\half}  - \frac{\phi_j}{\mu_j^\half} \big\Vert_{V1}
+\big\Vert \frac{\phi_j}{\vert \lambda_j \vert^\half} \big\Vert_{V} \, \big\vert 1 - \bigl\vert \frac{\lambda_j}{\mu_j} \bigr\vert ^\half \big\vert \, .
$$
Furthermore,
\begin{align}
\label{eq:lem:renormed_bari_basis-10}
\big|1-\bigl|\frac{\lambda_j}{\mu_j} \bigr|^\half \big|
\leq |1+\bigl|\frac{\lambda_j}{\mu_j} \bigr|^\half \big|^{-1} \frac{|\lambda_j-\mu_j|}{\mu_j}
\leq \frac{|\lambda_j-\mu_j|}{\mu_j}
\leq \frac{C}{j^{2\alpha+\gamma}}
\end{align}
due to \cite[Thm.~3.2,Rem.~3.3]{Mityagin_Siegl_19} and $\mu_j\geq C'j^\gamma$.
Hence the remaining claims follow.
\end{proof}

\begin{remark}
All estimates and results discussed here hold only for a sufficiently large $n>N$ where the eigenvalues are of algebraic multiplicity equal one. We do not control the eigenvalues in preasymptotic range where they may come with Jordan chains. Nevertheless, if we use the possible generalized eigenvectors (root vectors) in the preasymptotic range, and the eigenvectors $\phi_n,\, n > N$, the eigensystem still forms Bari bases in both the spaces $X$ and $V$. 
\eremk
\end{remark}

\begin{theorem}\label{thm:RieszBasisInV}
Under the assumption $2\alpha+\gamma>1$ the system $\{\mu_n^{-\half}\phi_n\}_{n=1}^\infty$ (and also $\{\|\phi_n\|_V^{-1}\phi_n\}_{n=1}^\infty$) is a Riesz basis in $V$.
\end{theorem}
\begin{proof}
We follow the lines of the proof of \cite[Thm.~{3.4}]{Mityagin_Siegl_19} and adapt the estimates.
We apply \cite[Thm.~VI.2.1]{Gohberg_Krein_65}, which tells us that it is sufficent to to verify 
that $\{ \mu_n^{-\half}\phi_n\}_{n\in\mathbb{N}}$ is complete in $V$, 
that there exists a $V$-complete system $\{ \hat\phi_n\}_{n\in\mathbb{N}}$ that is $V$-biorthogonal to $\{ \mu_n^{-1/2}\phi_n\} _{n\in\mathbb{N}}$, 
and that we have
\begin{align}
\label{eq:Gohberg_cond}
\forall f\in V, \quad \sum_{n=1}^\infty |(f,\mu_n^{-1/2}\phi_n)_V|^2 < \infty, \quad\sum_{n=1}^\infty |(f,\hat\phi_n)_V|^2 < \infty.
\end{align}
We observe that it is sufficent to prove the assumptions for all $n>N$ with arbitrary $N>0$.
Let $\{ \tilde\phi_n\}_{n\in\mathbb{N}}$ be as in \cite{Mityagin_Siegl_19}, i.e.,
\begin{align*}
\tilde\phi_n:=\frac{1}{(\phi_n^*,\phi_n)}\phi_n^*, \quad \phi_n^*:=\frac{1}{2\pi i}\int_{\Gamma_n} (z-T^*)^{-1}\psi_n\, dx, \quad n>N,
\end{align*}
i.e., the scaled eigenvectors of $T^*$ and $N$ as in \cite[Thm.~{3.2}]{Mityagin_Siegl_19}.
Without loss of generality we assume $N=1$ from now on.
We use $\hat\phi_n:=\frac{\mu_n^{1/2}}{\ol{\lambda_n}}A^{-1}T^*\tilde\phi_n$ for which $(\phi_n,\hat\phi_m)_V=\delta_{nm}$.
Note also that $\{ \hat\phi_n\} _{n\in\mathbb{N}}$ is $V$-complete.
We give the detailed proof for the case $2\alpha\leq1$; the other case is similar.
The estimation of the first series in~(\ref{eq:Gohberg_cond})
\begin{align*}
\sum_{n=1}^\infty |(f,\mu_n^{-\half}\phi_n)_V|^2
=\sum_{n=1}^\infty |(A^{\half}f,\phi_n)|^2
<\infty
\end{align*}
follows precisely the lines in \cite[Thm.~{3.4}]{Mityagin_Siegl_19} with $f$ replaced by $A^\half f$.
Due to the more elaborate definition of the biorthogonal system, the estimation of the second series needs more attention. We start with
\begin{align}\label{eq:est-aux-ab}
\sum_{n=1}^\infty |(f,\hat\phi_n)_V|^2
\leq 2\underbrace{\sum_{n=1}^\infty \frac{\mu_n}{|\lambda_n|^2} |(A^{\half}f,A^{\half}\tilde\phi_n)|^2}_{(\ref{eq:est-aux-ab}a):=}+2\underbrace{\sum_{n=1}^\infty \frac{\mu_n}{|\lambda_n|^2} |b(f,\tilde\phi_n)|^2}_{(\ref{eq:est-aux-ab}b):=}.
\end{align}
To estimate (\ref{eq:est-aux-ab}a) we use\ldhrev{\footnote{$ \rho_n^{*(j)} := \sum_{k=j+1}^\infty \phi_k^{*(k)}$, see \cite{Mityagin_Siegl_19}.}}
 $\phi^*_n=\psi_n+\sum_{k=1}^j \phi_n^{*(k)}+\rho_n^{*(j)}$, $j\in\mathbb{N}$ and note that $\frac{\mu_n}{|\lambda_n|}=O(1), (\phi_n^*,\phi)=O(1)$ as $n\to\infty$.
 Neglecting the difference between $\tilde{\phi}_n$ and $\phi_n^\ast$, 
we estimate\footnote{$(\sum_{i=1}^n a_i)^2 \leq n \sum_{i-1}^n a_i^2$.}
\begin{align*}
(\ref{eq:est-aux-ab}a)\leq (j+2) \sum_{n=1}^\infty \frac{\mu_n}{|\lambda_n|^2} \Big(|(A^{\half}f,A^{\half}\psi_n)|^2+
\sum_{k=1}^j|(A^{\half}f,A^{\half}\phi^{*(k)})|^2+|(A^{\half}f,A^{\half}\rho^{*(j)})|^2\Big),
\end{align*}
where $\sum_{n=1}^\infty \frac{\mu_n}{|\lambda_n|^2} |(A^{\half}f,A^{\half}\psi_n)|^2=\sum_{n=1}^\infty \frac{\mu_n^2}{|\lambda_n|^2} |(f,\mu_n^{-1/2}\psi_n)_V|^2<\infty$, because $\{\mu_n^{-1/2}\psi_n\}_{n\in\mathbb{N}}$ is a Riesz basis in $V$.
Next we choose $j$ large enough such that $2(j+1)(2\alpha+\gamma-1)>1$ and hence
\begin{align*}
\sum_{n=1}^\infty 2\frac{\mu_n}{|\lambda_n|^2} \big(|(A^{\half}f,A^{\half}\rho_n^{*(j)})|^2
\leq C \|f\|_V^2 \sum_{n=1}^\infty \|\mu_n^{-\half}\rho_n^{*(j)})\|_V^2
<\infty
\end{align*}
due to Lemma \ref{lemma:H1phi_est} (and $\rho_n^{*(j)}=\sum_{k=j+1}^\infty \phi_n^{*(k)}$).
To estimate the remaining middle terms we expand $f=\sum_{m=1}^\infty  f_m \mu_n^{-1/2}\psi_m$, where $f_m\in\mathbb{C}$ and $\sum_{m=1}^\infty |f_m|^2\ = \|f\|_V^2$.
We have
\begin{align*}
\sum_{n=1}^\infty |(A^{\half}f,A^{\half}\mu_n^{-1/2}\phi^{*(k)})|^2
=\sum_{n=1}^\infty |\sum_{m=1}^\infty f_m (\psi_m,A^{\half}\mu_n^{-1/2}\phi_n^{*(k)})|^2.
\end{align*}
Recall \eqref{eq:phink} that
$$
\phi_n^{*(k)} 
\ds = \sum_{j_1 \neq n} \frac{\ol{b(\psi_n,\psi_{j_1})}}{\mu_n - \mu_{j_1}} \, 
\sum_{j_2 \neq n} \frac{\ol{b(\psi_{j_1},\psi_{j_2})}}{\mu_n - \mu_{j_2}} \, \ldots\, 
\sum_{j_k \neq n} \frac{\ol{b(\psi_{j_{k-1}},\psi_{j_k})}}{\mu_n - \mu_{j_k}} \, \psi_{j_k},
$$
and hence
\begin{align*}
&\sum_{n=1}^\infty |\sum_{m=1}^\infty f_m (\psi_m,A^{\half}\mu_n^{-1/2}\phi_n^{*(k)})|^2\\
&=\sum_{n=1}^\infty \Big|\sum_{m\ne n} f_m \sum_{j_1 \neq n} \frac{\ol{b(\psi_n,\psi_{j_1})}}{\mu_n - \mu_{j_1}} \, 
\sum_{j_2 \neq n} \frac{\ol{b(\psi_{j_1},\psi_{j_2})}}{\mu_n - \mu_{j_2}} \, \ldots\, 
\sum_{j_{k-1} \neq n} \frac{\ol{b(\psi_{j_{k-2}},\psi_{j_{k-1}})}}{\mu_n - \mu_{j_{k-1}}} \,
\frac{\ol{b(\psi_{j_{k-1}},\psi_{m})}}{\mu_n - \mu_{m}} \frac{\mu_m^{1/2}}{\mu_n^{1/2}} \Big|^2\\  
&\leq \sum_{n=1}^\infty \bigg(\sum_{m\ne n} \vert f_m \vert
\sum_{j_1 \neq n} \frac{M_B}{n^\alpha j_1^\alpha|\mu_n - \mu_{j_1}|} \, 
\sum_{j_2 \neq n} \frac{M_B}{j_1^\alpha j_2^\alpha|\mu_n - \mu_{j_2}|} \, \ldots\, 
\sum_{j_{k-1} \neq n} \frac{M_B}{j_{k-2}^\alpha j_{k-1}^\alpha |\mu_n-\mu_{j_{k-1}}|} \,
\frac{M_B}{j_{k-1}^\alpha m^\alpha|\mu_n-\mu_{m}|} \frac{\mu_m^{1/2}}{\mu_n^{1/2}} \bigg)^2 \\
&= \sum_{n=1}^\infty \bigg(\sum_{m \ne n } \vert f_m \vert 
\Big(\sum_{j \neq n} \frac{M_B}{j^{2\alpha}|\mu_n - \mu_{j}|}\Big)^{k-1} \, 
\frac{M_B}{n^\alpha m^\alpha|\mu_n-\mu_{m}|} \frac{\mu_m^{1/2}}{\mu_n^{1/2}} \bigg)^2\\
&\leq \sum_{n=1}^\infty (C\sigma_{2\alpha,\gamma}(n))^{2(k-1)} \bigg(\sum_{m\ne n } \vert f_m \vert
\frac{M_B}{n^\alpha m^\alpha|\mu_n-\mu_{m}|} \frac{\mu_m^{1/2}}{\mu_n^{1/2}} \bigg)^2.
\end{align*}
Thus it suffices to estimate
\begin{align*}\sum_{n=1}^\infty \bigg(\sum_{m\ne n} \vert f_m\vert
\frac{M_B}{n^\alpha m^\alpha|\mu_n-\mu_{m}|} \frac{\mu_m^{1/2}}{\mu_n^{1/2}} \bigg)^2
=:\|\mathcal{M}\mathbf{f}\|_{\ell^2(\mathbb{N})}^2,
\end{align*}
where $\mathbf{f}:= \{ f_m\}_{m\in\mathbb{N}}\in\ell^2(\mathbb{N})$ and $\mathcal{M}_{nm}:=\frac{M_B}{n^\alpha m^\alpha|\mu_n-\mu_{m}|} \frac{\mu_m^{1/2}}{\mu_n^{1/2}} 
(1 - \delta_{nm})$.
We succeed, if we can show that $\mathcal{M}$ is bounded, i.e., $\mathcal{M}\in\mathcal{L}(\ell^2(\mathbb{N}))$.
To this end we apply the Schur test \cite[Thm.~{5.2}]{halmos-sunder78}
with $p_n=q_n=\frac{1}{\mu_n^{1/2}n^\alpha}$ by computing 
\begin{align*}
\sum_{m=1}^\infty |\mathcal{M}_{nm}| q_m&= \frac{1}{\mu_n^{1/2}n^\alpha} M_B \sum_{m\neq n} \frac{1}{m^{2\alpha}|\mu_m-\mu_n|}
\leq C p_n \sigma_{2\alpha,\gamma}(n) \leq C p_n, \\
\sum_{n=1}^\infty |\mathcal{M}_{nm}| p_n&= \frac{\mu_m^{1/2}}{m^\alpha} M_B \sum_{n\neq m} \frac{1}{\mu_nn^{2\alpha}|\mu_n-\mu_m|}\\
&=\frac{\mu_m^{1/2}}{m^\alpha} M_B \sum_{n\neq m} \frac{1}{\mu_mn^{2\alpha}|\mu_n-\mu_m|}
+\frac{\mu_m^{1/2}}{m^\alpha} M_B \sum_{n\neq m} \big(\frac{1}{\mu_n}-\frac{1}{\mu_m}\big)\frac{1}{n^{2\alpha}|\mu_n-\mu_m|}\\
&\leq \frac{1}{\mu_m^{1/2}m^\alpha} M_B \sum_{n\neq m} \frac{1}{n^{2\alpha}|\mu_n-\mu_m|}
+\frac{1}{\mu_m^{1/2}m^\alpha} M_B \sum_{n\neq m} \frac{1}{\mu_nn^{2\alpha}}\\
&\leq \frac{1}{\mu_m^{1/2}m^\alpha} M_B \sum_{n\neq m} \frac{1}{n^{2\alpha}|\mu_n-\mu_m|}
+\frac{1}{\mu_m^{1/2}m^\alpha} M_B \sum_{n\neq m} \frac{1}{n^{2\alpha+\gamma}}
\leq Cq_m.
\end{align*}
Hence it remains to estimate (\ref{eq:est-aux-ab}b).
Again, we use the expansion $\phi^*_n=\psi_n+\sum_{k=1}^j \phi_n^{*(k)}+\rho_n^{*(j)}$ with $j$ as previously specified and estimate
\begin{align}\label{eq:RVB-aux1}
\begin{split}
\sum_{n=1}^\infty &\frac{\mu_n}{|\lambda_n|^2}|b(f,\tilde\phi_n)|^2\\
&\leq
(j+2)\left\{ \sum_{n=1}^\infty \frac{\mu_n}{|\lambda_n|^2}|b(f,\psi_n)|^2
+\sum_{n=1}^\infty \sum_{k=1}^j\frac{\mu_n}{|\lambda_n|^2}|b(f,\tilde\phi_n^{*(k)})|^2
+ \sum_{n=1}^\infty \frac{\mu_n}{|\lambda_n|^2}|b(f,\rho_n^{*(j)})|^2 \right\}.
\end{split}
\end{align}
To estimate the last term in \eqref{eq:RVB-aux1} we note that the form $b(\cdot,\cdot)$ is $V$-bounded and thus
\begin{align*}
\sum_{n=1}^\infty \frac{\mu_n}{|\lambda_n|^2}|b(f,\rho_n^{*(j)})|^2
\leq C\sum_{n=1}^\infty \frac{\mu_n}{|\lambda_n|^2} \|f\|_V^2 \|\rho_n^{*(j)})\|_V^2
\leq C\|f\|_V^2\sum_{n=1}^\infty \|\mu_n^{-1/2}\rho_n^{*(j)})\|_V^2<\infty.
\end{align*}
Further, using $f=\sum_{m=1}^\infty f_m \mu_m^{-1/2}\psi_m$, $\mathbf{f}=\{f_m\}_{m\in\mathbb{N}}\in\ell^2(\mathbb{N})$ as previously we estimate the first term in \eqref{eq:RVB-aux1} as
\begin{align*}
\sum_{n=1}^\infty \frac{\mu_n}{|\lambda_n|^2}|b(f,\psi_n)|^2
&\leq \sum_{n=1}^\infty \frac{\mu_n}{|\lambda_n|^2} \big|\sum_{m=1}^\infty \frac{M_B f_m}{\mu_m^{1/2}m^\alpha n^\alpha}\big|^2
\leq C\sum_{n=1}^\infty \frac{1}{\mu_n n^{2\alpha}} \big|\sum_{m=1}^\infty \frac{f_m}{\mu_m^{1/2}m^\alpha}\big|^2\\
&\leq C\sum_{n=1}^\infty \frac{1}{n^{2\alpha+\gamma}} \big|\sum_{m=1}^\infty \frac{f_m}{m^{\alpha+\gamma/2}}\big|^2
\leq C \|\mathbf{f}\|_{\ell^2(\mathbb{N})} \sqrt{\sum_{m=1}^\infty \frac{1}{m^{2\alpha+\gamma}}}<\infty.
\end{align*}
For the remaining middle term in \eqref{eq:RVB-aux1} we proceed as follows:
\begin{align*}
\sum_{n=1}^\infty &\frac{\mu_n}{|\lambda_n|^2}|b(f,\tilde\phi_n^{*(k)})|^2\\
&=  \sum_{n=1}^\infty \frac{\mu_n}{|\lambda_n|^2} \Big| \sum_{m=1}^\infty f_m  \sum_{j_1 \neq n} \frac{\ol{b(\psi_n,\psi_{j_1})}}{\mu_n - \mu_{j_1}} \, 
\sum_{j_2 \neq n} \frac{\ol{b(\psi_{j_1},\psi_{j_2})}}{\mu_n - \mu_{j_2}} \, \ldots\, 
\sum_{j_k \neq n} \frac{\ol{b(\psi_{j_{k-1}},\psi_{j_k})}}{\mu_n - \mu_{j_k}}b(\mu_m^{-1/2}\psi_m,\psi_{j_k}) \Big|^2\\
&\leq \sum_{n=1}^\infty \frac{\mu_n}{|\lambda_n|^2} \left( \sum_{m=1}^\infty \vert f_m \vert \sum_{j_1 \neq n} \frac{M_B}{n^\alpha j_1^\alpha|\mu_n-\mu_{j_1}|} \, 
\sum_{j_2 \neq n} \frac{M_B}{j_{1}^\alpha j_{2}^\alpha|\mu_n-\mu_{j_2}|} \, \ldots\, 
\sum_{j_k \neq n} \frac{M_B}{j_{k-1}^\alpha j_k^\alpha|\mu_n-\mu_{j_k}|}
\frac{M_B}{\mu_m^{1/2}m^\alpha j_k^\alpha}\right)^2\\
&=\sum_{n=1}^\infty \frac{\mu_n}{|\lambda_n|^2} \left(
\sum_{m=1}^\infty \vert f_m \vert
\Big(\sum_{j \neq n} \frac{M_B}{j^{2\alpha}|\mu_n-\mu_{j}|}\Big)^k 
\frac{M_B}{\mu_m^{1/2}m^\alpha n^\alpha}\right)^2\\
&\leq C\sum_{n=1}^\infty \frac{\sigma_{2\alpha,\gamma}(n)^k}{\mu_n n^{2\alpha}} \left(
\sum_{m=1}^\infty \vert f_m \vert
\frac{M_B}{m^{\alpha+\gamma/2}}\right)^2
\leq C \|\mathbf{f}\|_{\ell(\mathbb{N})}^2 \big( \sum_{m=1}^\infty \frac{1}{m^{2\alpha+\gamma}} \big)^2<\infty.
\end{align*}
Hence the proof is complete. 
\end{proof}

\section{Properties of Jordan Chains}
\label{appendix:3}
In this section, we assume that $(H, (\cdot,\cdot)_H)$ is a Hilbert space and $A:\dom(H) \rightarrow H$ a densely defined 
closed, sectorial linear operator
with compact resolvent, i.e., $(A-z_0 \operatorname{I})^{-1}: H \rightarrow H$ is a compact operator for some $z_0 \in \doubleIC$. 
Then the spectrum $\sigma(A) \subset \doubleIC$ consists of isolated points only. 
The adjoint is denoted $A^\star$ with spectrum $\sigma(A^\star) = \overline{\sigma(A)}$. 
The Riesz projector for $A$ and $\lambda \in \doubleIC$ is defined as 
$$
P_\lambda(A) = -\frac{1}{2\pi i} \int_{\mathcal C} R(\zeta)\, d\zeta, 
\qquad R(\zeta) = (A - \zeta \operatorname{I})^{-1}, 
$$
where ${\mathcal C}$ is a positively oriented curve enclosing $\lambda$ and no point of $\sigma(A)\setminus \{\lambda\}$. 
We denote $H_\lambda(A):= \operatorname{Ran} P_{\lambda}(A) := P_{\lambda}(A) H$. We have, see, e.g., 
\cite{Gohberg_Krein_65} or \cite{Demkowicz_24}: 
\begin{itemize}
\item
$\operatorname{dim} H_\lambda(A) = \operatorname{dim} H_{\overline{\lambda}} (A^\star) < \infty$ for all $\lambda \in \sigma(A)$
\item
$P_\lambda(A) P_{\lambda'}(A) = 0$ for $\lambda \ne \lambda'$
\item
$P_\lambda(A)^\star = P_{\overline{\lambda}}(A^\star)$ (see, e.g.,  \cite[Cor.~1.12]{Demkowicz_24})
\end{itemize}
\begin{lemma}
\label{lemma:m-1}
It holds that $\operatorname{dim} H_\lambda(A) = \operatorname{dim} H_{\overline{\lambda}}(A^\star) < \infty$.
For $\lambda \ne \lambda'$ it holds that $H_\lambda(A) \perp_H H_{\overline{\lambda'}}(A^\star)$.
\end{lemma}
\begin{proof}
For the first statement we refer to \cite[Cor. 1.12]{Demkowicz_24}.
The second statement follows from 
$P_{\lambda}(A) P_{\lambda'}(A) = 0$: for $x \in H_\lambda$, $x' \in H_{\lambda'}$ we have
$(x,x')_H = (P_{\lambda}(A) x, P_{\overline{\lambda'}}(A^\star) x')_H = (P_{\lambda'}(A) P_{\lambda} (A) x,x')_H $ $=$ $ (0,x')_H  =0$. 
%
\end{proof}
\begin{lemma}
\label{lemma:m-3}
Assume that the spaces $\{H_\lambda(A)\colon \lambda \in \sigma(A)\}$ are complete in $H$.
Then, given a basis $\{x_j\}_{j=1}^N$ of $H_{\lambda}(A)$ there exists a (unique) biorthogonal (w.r.t.\ to $(\cdot,\cdot)_H)$) 
basis $\{x_j^\prime\}_{j=1}^N$
of $H_{\overline{\lambda}}(A^\star)$. 
\end{lemma}
\begin{proof}
Let $\{y^\prime_j\}_{j=1}^N$ be a basis of $H_{\overline{\lambda}}(A^\star)$.

\emph{Claim:} The Gram matrix $G_{ij}:= (x_i,y^\prime_j)_H$ is invertible. 
Suppose otherwise. Then there exists an $\alpha \in {\mathbb C}^N\setminus \{0\}$ with $G \alpha = 0$. Hence, 
$0 \ne y^\prime:= \sum_{k=1}^N \alpha_k y^\prime_k \in H_{\overline{\lambda}}(A^\star)$ satisfies $(x,y^\prime)_H = 0$ for all $x \in H_\lambda (A)$. 
{}From Lemma~\ref{lemma:m-1}, we conclude $y^\prime \perp_H H_{\lambda'}(A)$ for all $\lambda' \in \sigma(A)$. By the completeness
of $\{H_\lambda(A)\colon \lambda \in \sigma(A)\}$ in $H$, we get $y^\prime = 0$, a contradiction. 

\emph{Construction of $\{x^\prime_j\}_{j=1}^N$:} Making the ansatz $x^\prime_j = \sum_{k} \beta_{kj} y^\prime_k$ 
for suitable matrix $\beta \in {\mathbb C}^{N \times N}$, the biorthogonality 
requirement leads to 
\begin{align*}
\delta_{ij} &\stackrel{!}{ = } (x_i,x^\prime_j)_H = \sum_k \overline{\beta}_{kj}  (x_i,y^\prime_k)_H = \sum_{k} G_{ik} \overline{\beta}_{kj} 
 = (G \overline{\beta})_{ij}
\end{align*}
Hence, the choice $\beta = \overline{G}^{-1}$ yields the desired biorthogonal set. 
\end{proof}
\begin{lemma}
\label{lemma:m-4}
For $\lambda \ne \lambda'$ there holds 
$$
(A x_\lambda,x^\prime_{\lambda'})_H = 0 \qquad 
\forall x_\lambda \in H_\lambda(A), \quad x^\prime_{\lambda'} \in H_{\overline{\lambda'}}(A^\star). 
$$
\end{lemma}
\begin{proof}
Use Lemma~\ref{lemma:m-1} and the fact that $H_\lambda(A)$ is an $A$-invariant subspace. 
\end{proof}
\begin{lemma}
\label{lemma:m-5}
Let the basis $\{x_j\}_{j=1}^N$ of $H_\lambda(A)$ consist of $m \in \doubleIN$ Jordan chains, i.e., 
there are $m$ distinct values $j_1,\ldots,j_m \in \{1,\ldots,N\}$ with
$A x_{j_l} = \lambda x_{j_l}$, $l=1,\dots,m$
for the remaining $j \in \{1,\ldots,N\}\setminus \{j_1,\ldots,j_m\}$ there holds 
$A x_j = \lambda x_j + x_{j-1}$. 
Then, the corresponding biorthogonal 
sequence $\{x^\prime_j\}_{j=1}^N$ given by Lemma~\ref{lemma:m-3} consists of $m$ Jordan chains 
for the operator $A^\star$ if one reverses the order, i.e., for each $j \in \{1,\ldots,N\}$ there holds  
$A^\star x_j = \overline{\lambda} x_j$ or $A^\star x_j = \overline{\lambda} x_j + x_{j+1}$. 
The lengths of the Jordan chains of $A|_{H_\lambda(A)}$ and $A^\star|_{H_{\overline{\lambda}}(A^\star)}$ are 
in one-to-one correspondence, i.e., $x_{j_0}, \ldots, x_{j_0+n}$ is a Jordan chain for $A$, then 
$x^\prime_{N-(j_0+n)},\ldots,x^\prime_{N-j_0}$ is a Jordan chain for $A^\star$. 
The matrix 
$$
{\mathbf A}_{ij}:= (A x_j,x^\prime_i)_H, \qquad i,j=1,\ldots,N
$$
consists of $m$ Jordan blocks, all $\lambda$ on the diagonal. 
In particular, ${\mathbf A}$ is upper triangular. 
\end{lemma}
\begin{proof}
Let $x_{j_0},\ldots,x_{j_0+n}$ be a Jordan chain, i.e., 
$A x_{j_0} = \lambda x_{j_0}$, 
$A x_{j_0+j} = \lambda x_{j_0+j} + x_{j_0+j-1}$ for $j =1,\ldots,n$, and 
$A x_{j_0+n} \ne \lambda x_{j+0+n} + x_{j_0 +n-1}$. 

\emph{Case a:} $j = j_0$. Then $x_{j_0}$ is an eigenvector of $A$ and with the biorthogonality 
\begin{align*}
(A x_j,x^\prime_i)_H & = (\lambda x_j, x^\prime_i)_H = \lambda \delta_{ji} 
\end{align*}
\emph{Case b:} $j_0 < j \leq j_0 + n$. Then 
\begin{align*}
(A x_{j}, x^\prime_i)_H & = 
(\lambda x_{j} + x_{j-1}, x^\prime_i)_H  = 
 \lambda \delta_{j,i} + \delta_{j-1,i} 
\end{align*}
This shows that the matrix ${\mathbf A}$ has the desired bidiagonal and hence in particular upper triangular structure. 

For $A^\star$, we note that the matrix $({\mathbf A}^\star)_{ij}:= (A^\star x^\prime_j, x_i)_H$ is given by 
\begin{align*}
{\mathbf A}^\star_{ij} = (A^\star x^\prime_j, x_i)_H = 
\overline{(x_i, A^\star x^\prime_j)_H} = 
\overline{(A x_i, x^\prime_j)_H} = \overline{{\mathbf A}_{ji}}. 
\end{align*}
This implies that the vectors $(x_i^\prime)_{i=1}^N$ correspond to Jordan chains for $A^\star$ after reversing the order. 
\end{proof}

%
%

\section{Are Jordan Chains Really Possible ?
\label{appendix:1}
}
\subsection{Jordan Chains In The Presence Of Complex Impedance}
To investigate the possibility of the presence of Jordan chains, let us consider the algebraically more tractable case of 
a straight waveguide with impedance boundary condition
investigated in \cite{Demkowicz_Gopalakrishnan_Heuer_24}.  For simplicity, we assume that the waveguide width $a = 1$. The eigenvectors are of the form
\be
\cos z x ,\quad x \in (0,1)
\label{eq:straight_waveg_eigenmode}
\ee
where $z = \sqrt{\lambda -1}$ ,  with $\lambda$ denoting the complex eigenvalue.

Let $(\lambda, \phi)$ be an eigenpair.
As discussed in Section~\ref{appendix:1-example} below, 
 the presence of Jordan chains is equivalent to the condition
$$
\int_0^1 \phi^2 = 0 \, .
$$
The first question thus is whether it is possible to find $z$ such that
$$
\int_0^1 \cos^2 zx \, dx = 0 \, .
$$
In other words, does the equation above have complex roots $z$?.
Computing the integral, we get:
$$
2 \int_0^1 \cos^2 zx\, dx = \int_0^1 (1 + \cos 2 z x) \, dx = 1 + \frac{1}{2z} \sin 2zx \big\vert_0^1 = 1 + \frac{ \sin 2z}{2z }  = 0 \, .
$$
Note that $z = 0$ does not solve the equation, so we can assume $z \ne 0$.
We have thus the equation:
\be
\sin 2z + 2z = 0  \quad \text{or} \quad \sin z + z = 0  \quad (z \leftarrow 2z) \, .
\label{eq:root1}
\ee
Fig~\ref{fig:sinz+z} presents a {\tt Matlab} plot\footnote{All pictures are courtesy of Jonathan Zhang.} of 
the function $\ln ( \vert \sin z + z \vert)$ in the domain $(-20,20) \times (-10,10)$ and a zoom (with an increased resolution) into $(-5,5)\times (-5,5)$. The dark blue dots represent roots of the equations above.
Clearly, we have multiple roots.
\begin{figure} [h!]  
\centering
\includegraphics[angle=0,width=.35\textwidth]{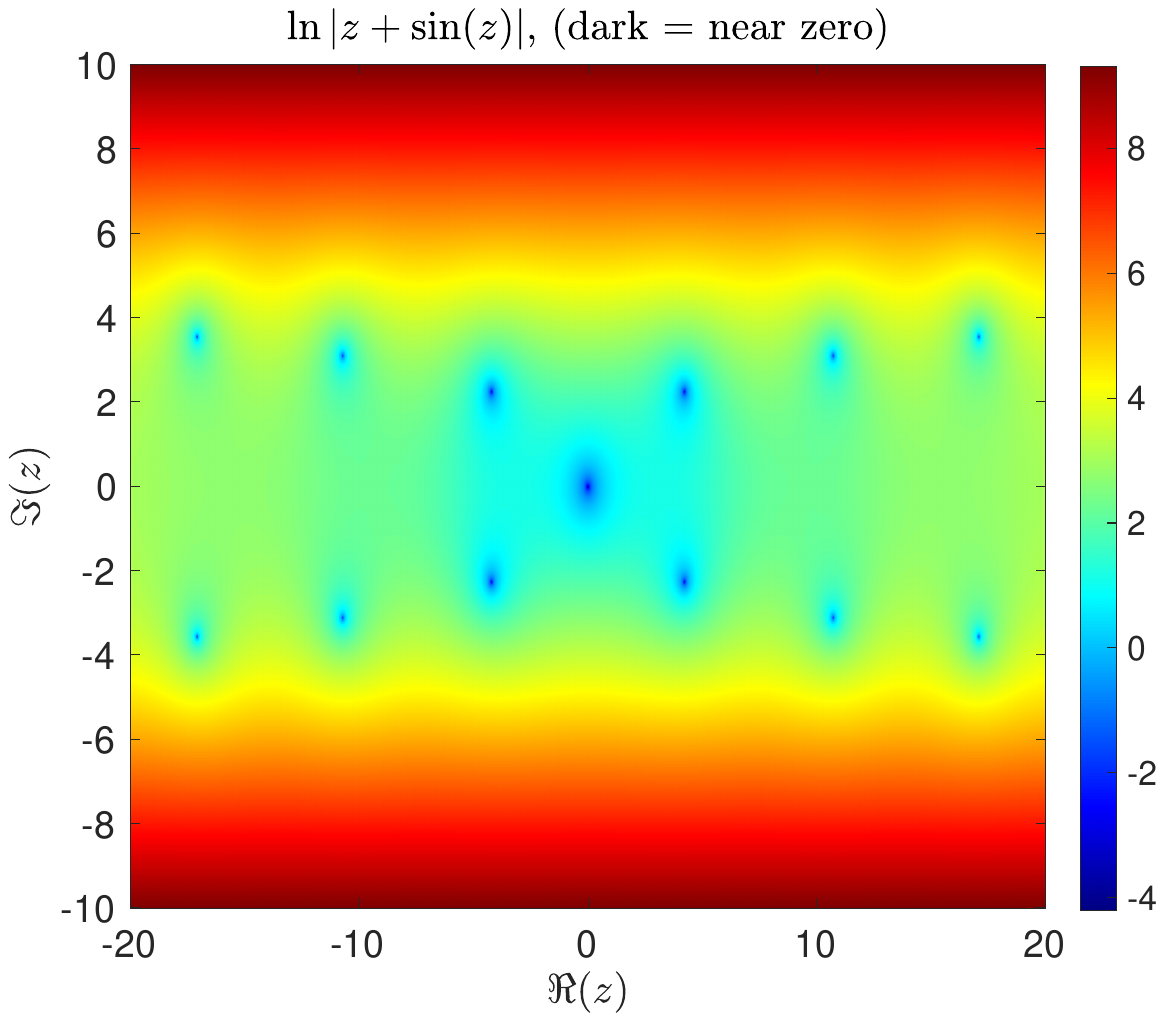}
\includegraphics[angle=0,width=.35\textwidth]{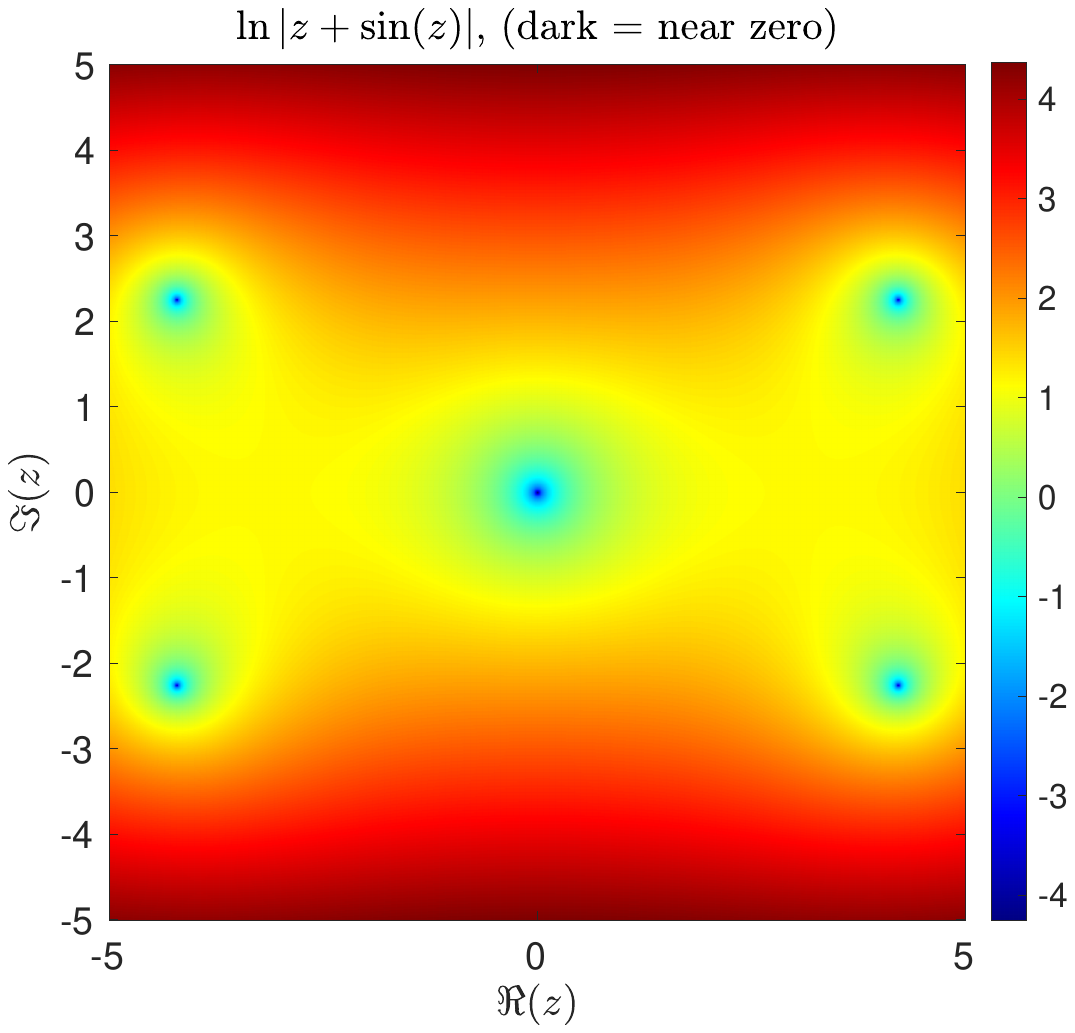}
\caption{Plot of $\vert \sin z + z \vert$ using a logarithmic scale.}
\label{fig:sinz+z}
\end{figure}
Alternatively, the first two plots in Fig.~\ref{fig:ReIm(sinz+z)} show zero contours of the real and imaginary parts of $\sin z + z$. The third one displays both families (the zero contours
of the product of the functions). The intersection points between the two families of curves correspond to zeros of function $\sin z + z$, cf.~Fig.\ref{fig:sinz+z}.
\begin{figure} [h!]  
\centering
\includegraphics[angle=0,width=.4\textwidth]{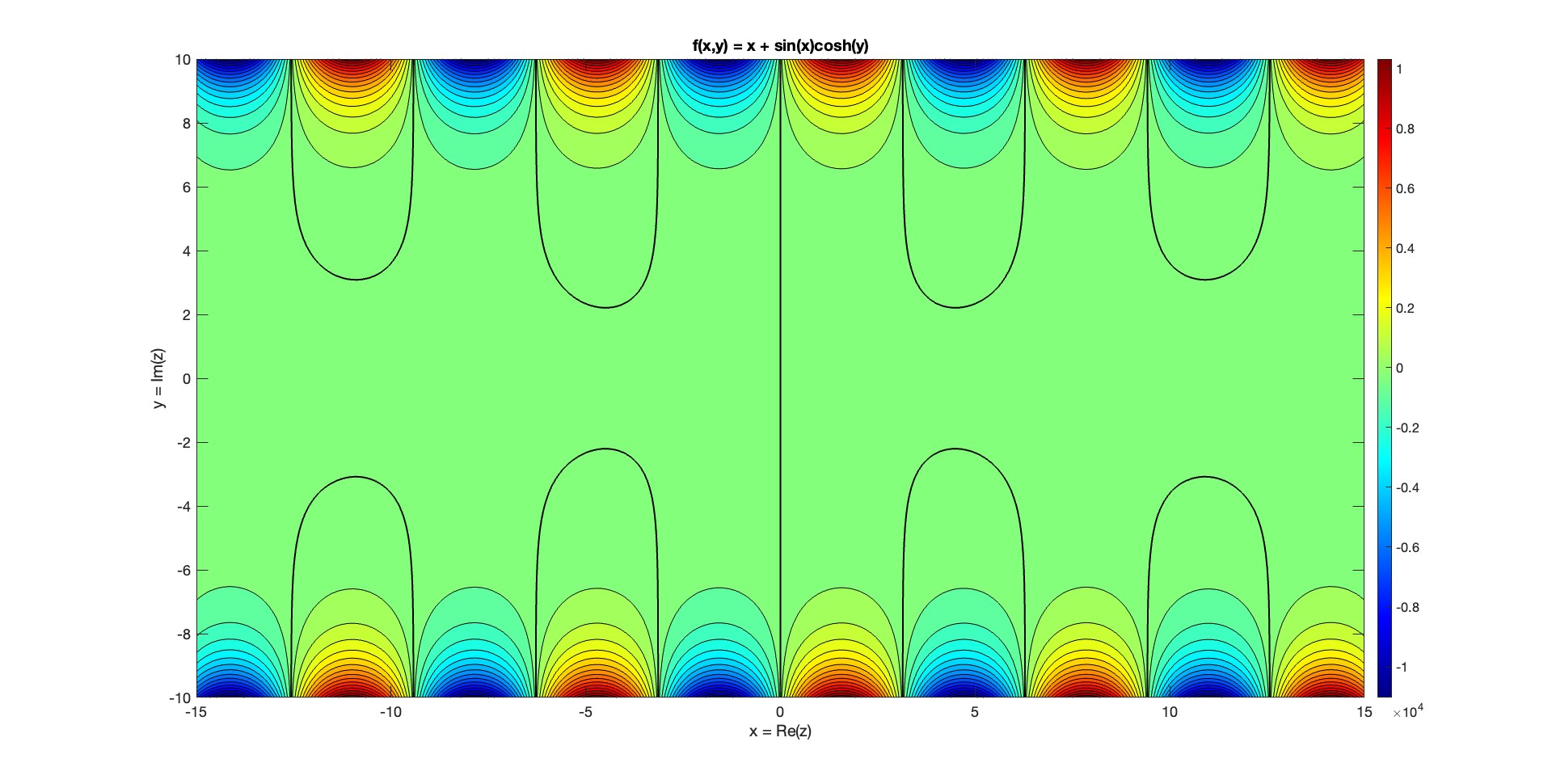}
\includegraphics[angle=0,width=.4\textwidth]{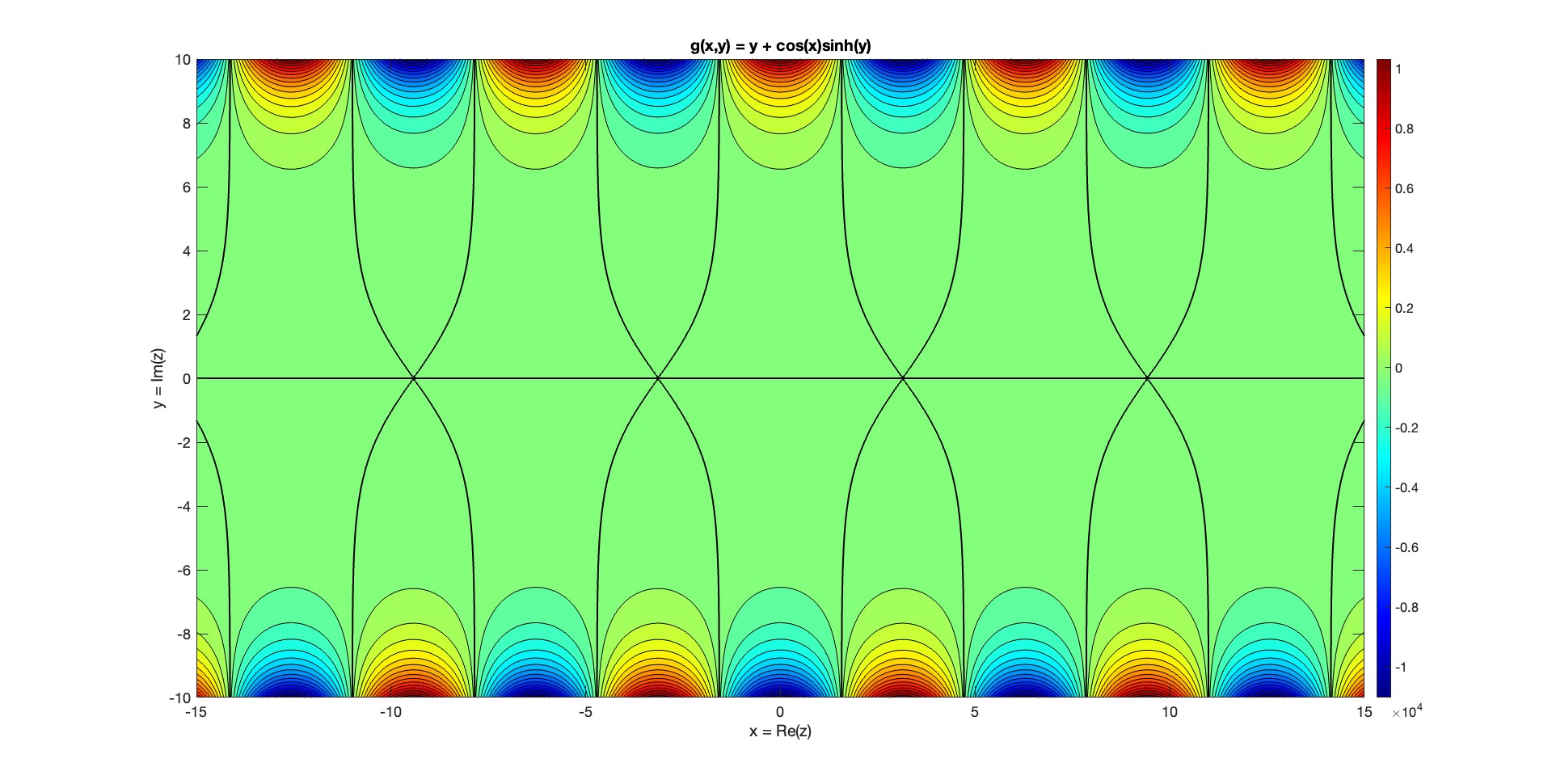}
\includegraphics[angle=0,width=.4\textwidth]{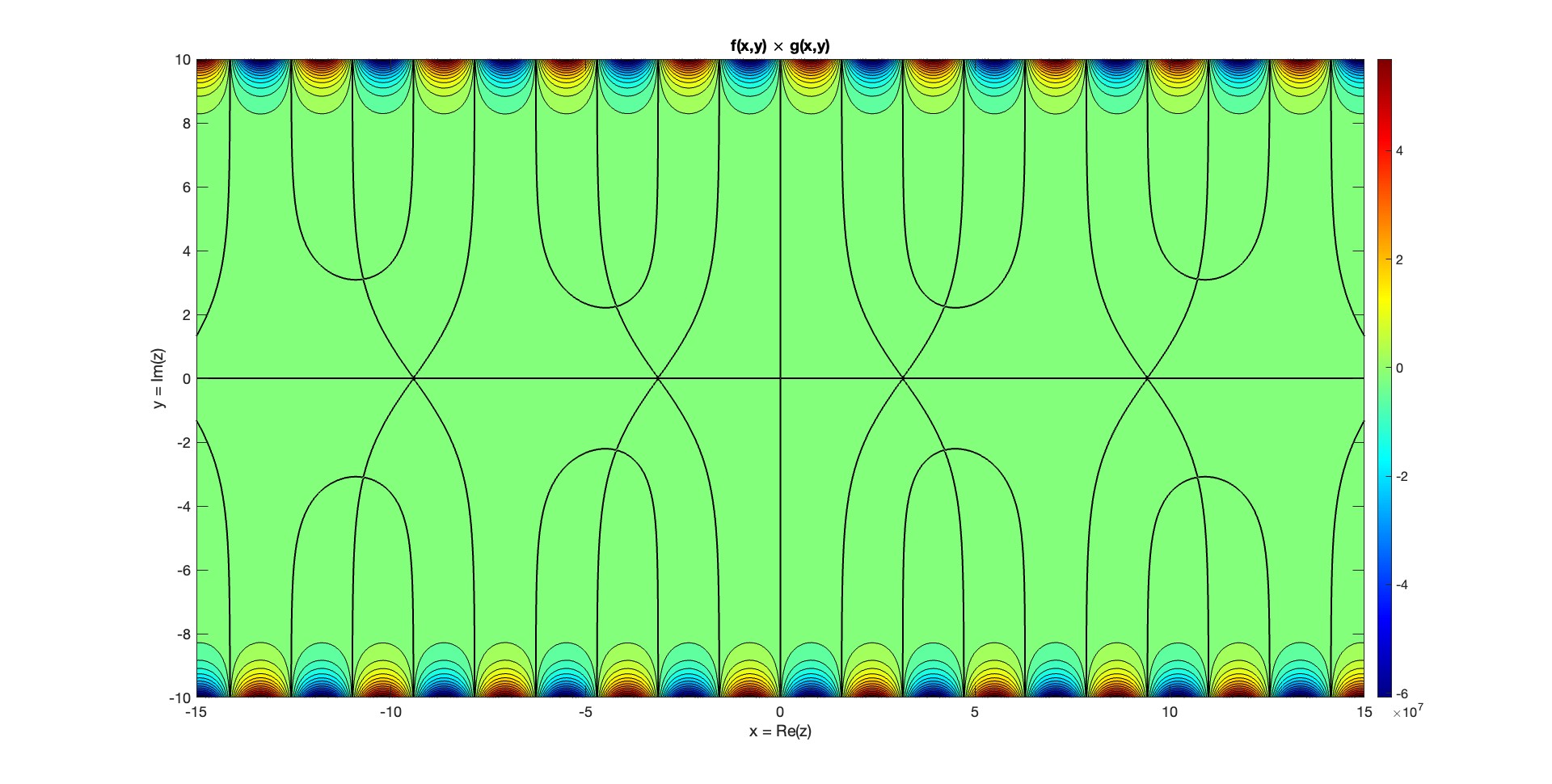}
\caption{Zero contours of $\Re (\sin z + z), \Im (\sin z + z)$ and the product of the two functions.}
\label{fig:ReIm(sinz+z)}
\end{figure}

Plugging the formula~(\ref{eq:straight_waveg_eigenmode}) into the impedance BC, we obtain (cf.~\cite{Demkowicz_Gopalakrishnan_Heuer_24}) another equation for $z$:
\be
z \tan z = i \omega \impedance \, .
\label{eq:root2}
\ee
Can we find a value for $z$ that solves both~(\ref{eq:root1}) and~(\ref{eq:root2}) ?
We can argue that such a situation is possible. Take any root of~(\ref{eq:root1}), substitute it into~(\ref{eq:root2}), and simply solve for the (complex) impedance constant $\impedance$.
Thus, at least with a complex impedance constant $\impedance$, we may run into a Jordan chain. 
\subsection{An Example of Jordan Chains Of Length $1$}
\label{appendix:1-example}
Recall that the first generalized eigenvector $\phi_1$ in a Jordan chain for an operator $T$ corresponding to an eigenpair $(\lambda,\phi)$ is given by the equation:
$$
(T - \lambda) \phi_1 = \phi \, ,
$$
The equation has a solution if and only if the right-hand side $\phi$ satisfies the compatibility condition:
$$
\phi \perp \mathcal{N}(T^\ast - \bar{\lambda}) 
\quad \Leftrightarrow \quad   (\phi,\bar{\phi}) = \int \phi^2 = 0 \, .
$$
The non-existence of the Jordan chain is therefore equivalent to the condition $\int \phi^2 \neq 0$. Just for illustration,  Table~\ref{table:eigen} presents
the first 20 eigenvalues and the corresponding integrals $\int \phi_n^2$ for the problem with data: $\omega = 10$, $r_1 = 99.5$, $r_2 = 100.5$, $\impedance = 1$.
We can see that the corresponding integrals are far from being zero, indicating that 
in this case also in the preasymptotic regime all Jordan chains have length $1$. 
\begin{table}[tb]
  \centering
  \begin{tabular}{|r|r|r|r|r|}
    \hline
 $n$  & $\Re \lambda_n$ &  $\Im \lambda_n$ & $\Re \int \phi_n^2$ & $\Im \int \phi_n^2$\\
 \hline
  1& -78.59325&  -4.75759&   0.78654&   0.14365\\
 2& -42.65673& -14.84895&   0.50072&   0.23124\\
 3&   0.78190& -27.44418&   0.55941&   0.13033\\
 4&  60.85205& -26.25547&   0.64502&  -0.10557\\
 5& 147.75558& -23.58280&   0.75264&  -0.09864\\
 6& 255.82226& -22.33018&   0.78701&  -0.04243\\
 7& 383.92616& -21.66149&   0.93699&  -0.06477\\
 8& 531.86706& -21.25960&   0.81674&  -0.17354\\
 9& 699.58884& -20.99763&   0.83734&  -0.01632\\
10& 887.07069& -20.81663&   0.91998&   0.04520\\
11&1094.30334& -20.68600&   0.95837&   0.03218\\
12&1321.28214& -20.58847&   0.97569&   0.00107\\
13&1568.00454& -20.51364&   0.97057&  -0.00984\\
14&1834.46902& -20.45490&   0.99136&  -0.02622\\
15&2120.67462& -20.40793&   0.96759&  -0.08980\\
16&2426.62075& -20.36976&   0.96724&  -0.04140\\
17&2752.30697& -20.33831&   0.98331&  -0.02085\\
18&3097.73301& -20.31208&   0.98677&  -0.04001\\
19&3462.89864& -20.28998&   0.98518&  -0.06569\\
20&3847.80372& -20.27117&   0.96087&  -0.10928\\
\hline
\end{tabular}
\caption{First 20 eigenvalues $\lambda_n$ and the corresponding integrals 
$\int_{r_1 }^{r_2 } \phi_n^2$ for data: $\omega = 10$, $r_1 = 99.5$, $r_2 = 100.5$, $\impedance = 1$. \label{table:eigen}
}
\end{table}

\end{document}